\documentclass{article}
\usepackage{graphicx} 
\usepackage{amsfonts, amsmath, amssymb, color, a4wide, algorithm2e,amsthm}
\usepackage{hyperref}
\usepackage{caption}
\usepackage[sorting=none,doi=false,isbn=false,giveninits=true,style=nature,url=false]{biblatex}
\usepackage{authblk}

\newtheorem{definition}{Definition}
\newtheorem{thm}{Theorem}
\newtheorem{prop}{Proposition}
\newtheorem{lem}{Lemma}
\newtheorem{cor}{Corollary}

\newcommand{\argmin}{\mathop{\mathrm{argmin}}}
\newcommand{\argmax}{\mathop{\mathrm{argmax}}}
\newcommand{\E}{\operatorname{\mathbb{E}}}
\renewcommand{\P}{\operatorname{\mathbb{P}}}
\newcommand{\tr}{\mathrm{Tr}}
\newcommand{\pa}[1]{\left(#1\right)}
\newcommand{\ac}[1]{\left\{#1\right\}}
\newcommand{\cro}[1]{\left[#1\right]}
\newcommand{\<}{\langle}
\renewcommand{\>}{\rangle}
\newcommand{\eps}{\varepsilon}

\newcommand{\R}{\mathbb{R}}
\newcommand{\1}{\mathbf{1}}
\newcommand{\N}{\mathbb{N}}

\usepackage[textwidth=2.0cm, textsize=tiny]{todonotes} 

\bibliography{biblio}
\title{Computation-information gap in high-dimensional clustering}
\author[1]{Bertrand Even}
\author[1]{Christophe Giraud}
\author[2]{Nicolas Verzelen}

\affil[1]{\textit{\scriptsize{Université Paris-Saclay, Laboratoire de mathématiques d’Orsay, Orsay, France}}}
\affil[2]{\textit{\scriptsize{INRAE, MISTEA, Univ. Montpellier, Montpellier, France}} }

\begin{document}

\maketitle

\begin{abstract}%
We investigate the existence of a fundamental computation-information gap for the problem of clustering a mixture of isotropic Gaussian in the high-dimensional regime, where the ambient dimension $p$ is larger than the number $n$ of points. The existence of a computation-information gap in a specific Bayesian high-dimensional asymptotic regime has been conjectured by \cite{lesieur2016phase}  based on the replica heuristic from statistical physics. 
We provide  evidence of the existence of such a gap generically in the high-dimensional regime $p\geq n$, by (i)  proving a non-asymptotic low-degree polynomials computational barrier for clustering in high-dimension, matching the performance of the best known polynomial time algorithms, and by (ii) establishing that the information barrier for clustering is smaller than the computational barrier, when the number $K$ of clusters is large enough.  These results are in contrast with the (moderately) low-dimensional regime $n\geq \text{poly}(p,K)$, where there is no computation-information gap for clustering a mixture of isotropic Gaussian. 
In order to prove our low-degree computational barrier, we develop sophisticated combinatorial arguments to upper-bound the mixed moments of the signal under a Bernoulli Bayesian model.

\end{abstract}

\section{Introduction}
We investigate the problem of clustering a mixture of isotropic Gaussian in a high-dimensional set-up. 
The problem of clustering a mixture of Gaussian is a classical problem, which has lead to a large literature both in statistics and in machine learning \cite{Dasgupta99,VEMPALA2004,lesieur2016phase,LuZhou2016,DBLP:journals/corr/abs-1711-07211,Regev2017,giraud2019partial,fei2018hidden,chen2021hanson,Kwon20,SegolNadler2021,romanov2022,LiuLi2022,diakonikolasCOLT23b}.

\paragraph{Set-up.}
We observe a set of $n$ points $Y_1,\ldots,Y_n\in \R^p$, which have been generated as follows. For some unknown vectors $\mu_1,\ldots,\mu_K\in\R^p$, some unknown $\sigma>0$, and an unknown partition $G^*=\ac{G^*_1,\ldots,G^*_K}$ of $\ac{1,\ldots,n}$, the points $Y_1,\ldots,Y_n$ are sampled independently with distribution
\[Y_i \sim \mathcal{N}(\mu_k,\sigma^2I_p),\quad \text{for}\ i\in G^*_k.\]
 We focus in this paper on the high-dimensional setting $p\geq n$, 
  with balanced clusters
\begin{equation}\label{eq:balanced}
    {\max_k |G^*_k|\over \min_k |G^*_k|}\leq \alpha,\quad \text{for some $\alpha\geq 1$.}
\end{equation}

\paragraph{Curse of dimensionality.}
For $K=2$ clusters, in low dimension $p\ll n$, it is well known that the probability of misclassifying a new data point given the label of all the others decays like $\exp(-c \Delta^2)$ with the separation
\begin{equation}\label{eq:Delta}
    \Delta^2=\min_{l\neq r\in[1,K]}{\|\mu_{r}-\mu_{l}\|^2\over 2\sigma^2}\ .
\end{equation}
In the high-dimensional regime $p\gg n$, the variance of the estimation of the high-dimensional means $\mu_k$ leads to the slower rate
  $\exp\pa{-c'{n\over p} \Delta^4}$ when 
 $\Delta^2 \leq p/n$, see \cite{giraud2019partial}.

 This curse of dimensionality for the classification problem has some repercussion on the clustering problem. 
When $K=2$, and $p\geq n$ are large, \cite{ndaoud2022sharp} proved that a separation
\[\Delta^2 > 2\sqrt{2p\log(n)\over n}\]
is necessary  in order to perfectly recover the clusters, and also sufficient to recover them in polynomial time. 
For larger $K$, \cite{giraud2019partial} proved that an SDP relaxation of Kmeans \cite{PengWei07} provides a non-trivial clustering for
 $\Delta^2 \gtrsim \sqrt{pK^2/ n}$,
when $p\geq n$,
where $\gtrsim$ hides a multiplicative constant depending only on $\alpha$. Perfect clustering is also possible with single-linkage hierarchical clustering when $\Delta^2 \gtrsim  \sqrt{p\log(n)}+\log(n)$ --see Appendix \ref{sec:hierarchical} for details on hierarchical clustering--,
so, when $p\geq n$, non-trivial clustering is possible in polynomial time  for
\begin{equation}\label{eq:separation}
    \Delta^2 \gtrsim \sqrt{pK^2\over n} \wedge \sqrt{p\log(n)}.
\end{equation}
Some non-rigorous arguments from statistical physics suggest that this minimal separation for non-trivial clustering in polynomial time may be optimal, up to a possible $\sqrt{\log(n)}$ factor for the second term. Indeed, building on the replica heuristic from statistical physics, \cite{lesieur2016phase} conjectures that, when the means $\mu_k$ are drawn i.i.d. with Gaussian $\mathcal{N}(0,{p^{-1}\bar\Delta^2}I_p)$ distribution in $\R^p$, in the asymptotic regime where $n,p$ go to infinity with $p/n\to \gamma \in [(K/2-2)^{-2},+\infty)$,
non-trivial clustering is possible in polynomial time only for $\bar\Delta^2> \sqrt{\gamma K^2}$, while it is possible without computational constraints for $\bar\Delta^2> 2\sqrt{\gamma K\log(K)}$, see also \cite{banks2018information} for the problem of cluster detection.

These results are in contrast with the moderately low-dimensional setting, where it follows from \cite{LiuLi2022} that for $n\geq \text{poly}(p,K)$, non-trivial clustering is possible in polynomial time when 
$\Delta^2 \gtrsim (\log(K))^{1+c}$, with $c>0$, almost matching the information minimal separation $\Delta^2 \gtrsim\log(K)$ from \cite{Regev2017,Kwon20,romanov2022}, up to a small power of $\log(K)$.
 This set of results leaves open two fundamental questions:
\begin{enumerate}
   \item Can we design a polynomial-time algorithm achieving non-trivial clustering  for a separation smaller than (\ref{eq:separation}) in the high-dimensional setting $p\geq n$?
    \item What is the minimal separation $\Delta^2$ necessary for non-trivial clustering in high-dimension, and is there a computation-information gap as conjectured in \cite{lesieur2016phase}?
\end{enumerate}

\paragraph{Our contribution.} We provide an answer to these two  fundamental questions. 
\begin{enumerate}
    \item Our first contribution is to prove a non-asymptotic low-degree polynomial lower bound suggesting that the separation 
(\ref{eq:separation}) is minimal, up to a possible polylog$(n)$ factor, for clustering in polynomial time in the high-dimensional setting $p\geq n$. 
    \item Our second contribution is to prove that the information barrier for non-trivial clustering  is
\begin{equation}\label{eq:stat-separation}
    \Delta^2 \gtrsim \log(K) \vee \sqrt{pK\log(K)\over n},
\end{equation}
    with the exact Kmeans algorithm beeing (without surprise) information rate-optimal.
\end{enumerate}
These two results provide evidence for the existence of a  computation-information gap for the problem of clustering a mixture of isotropic Gaussians in high-dimension $p\geq n$, when the number $K$ of clusters is larger than some constant $K_0$; confirming and generalizing the gap conjectured in \cite{lesieur2016phase}. 

\begin{center}
    \begin{tabular}{c|c|c}
   $\vphantom{\Bigg|}\displaystyle{\Delta^2\lesssim \sqrt{pK\log(K)\over n}}$ & $\displaystyle{\sqrt{pK\log(K)\over n}\ \lesssim \Delta^2 \widetilde{\ll}  \sqrt{pK^2\over n}\wedge \sqrt{p}}$ &  $\displaystyle{\Delta^2 \gtrsim  \sqrt{pK^2\over n}\wedge \sqrt{p\log(n)}}$  \\\hline
      \vphantom{\bigg|}Impossible & Hard & Easy\\
\end{tabular}\\ \bigskip
\textit{Clustering hardness in high-dimension $p\geq n$. Here, $\ \widetilde{\ll}\ $ hides polylog$(n)$ factors.}
\end{center}

The main difficulty of the proof of the low-degree {computational barrier}
is to bound mixed moments of the high-dimensional signal, drawn under a Bernoulli Bayesian model defined in Section \ref{sec:low-degree}. {To derive these pivotal bounds, we develop sophisticated combinatorial arguments.}

\paragraph{Literature review.}
The problem of clustering in high-dimension has been investigated in \cite{LuZhou2016,fei2018hidden,giraud2019partial}. The latter provide some state-of-the-art controls on the (partial or perfect) recovery of the clusters in polynomial time, based on an SDP relaxation of Kmeans \cite{PengWei07}. \cite{ndaoud2022sharp} considers the problem of perfect recovery when there are $K=2$ clusters, identifying a sharp threshold for information-possible perfect recovery,   and proving that perfect clustering is possible in polynomial time above this threshold with a simple Lloyd algorithm. In particular, there is no computation-information gap for a mixture of $K=2$ isotropic Gaussian, whatever the ambient dimension $p$. In a Bayesian setting with a Gaussian prior on the $\mu_k$,  \cite{lesieur2016phase} conjectures a computation-information gap for clustering in the asymptotic limit where $p/n\to \gamma \in [(K/2-2)^{-2},+\infty)$. Similarly, \cite{banks2018information} proves that the information threshold for detecting the existence of clusters is smaller than the spectral detection threshold, when $K$ is large enough. Interestingly, this information barrier for cluster detection in this Bayesian setting matches, up to a possible constant, the information barrier (\ref{eq:stat-separation}) for clustering, so that there is no test-estimation gap at the information level. 
On a different perspective, the estimation of the parameters of a Gaussian mixture distribution in high-dimension has been addressed in \cite{doss2020optimal}. 

Contrary to our high-dimensional setting $p\geq n$,  there is no computation-information gap for learning mixture of isotropic Gaussian \cite{Regev2017,Kwon20,romanov2022,LiuLi2022} in a moderately low-dimensional setting $n\geq \text{poly}(p,K)$.
Some computation-information gaps have yet been shown in moderately low-dimension for learning mixture of non-isotropic Gaussian with unknown covariance. In such a setting, \cite{DiakonikolasFOCS17} and  \cite{diakonikolasCOLT23b} establish some  lower-bounds for the running time of any Statistical-Query
algorithm (SQ-algorithm), enforcing a computation-information gap between SQ-algorithms and
information optimal algorithms.
We emphasize that in our high-dimensional setting, contrary to the moderately low-dimensional case,  the computation-information gap is not induced by some non-isotropic effects. Indeed, when $p\geq n$, the computation-information gap shows up for isotropic Gaussian mixture. In addition, the performances of polynomial time estimators are similar for isotropic Gaussian and anisotropic subGaussian mixtures \cite{giraud2019partial}. We refer to  Section~\ref{sec:discussion} for (i) a detailed discussion on the differences between the high and moderately low-dimensional settings, and (ii) a discussion highlighting that optimal clustering rates cannot be simply derived from estimation rates.

The low-degree polynomial model of computation requires the output of the algorithm to be computed by a low-degree polynomial of the entries of the input data. 
Many state-of-the-art algorithms, including spectral methods and approximate message passing algorithms, can be approximated by low-degree polynomials, and the class of low-degree polynomials is as powerful as the
best known polynomial-time algorithms for many canonical problems, including planted clique \cite{Barak19}, community detection \cite{Hopkins17}, sparse PCA \cite{Ding23}, and tensor PCA \cite{HopkinsFOCS17}. A low-degree polynomial lower bound is then a compelling evidence for computational hardness of a learning problem. 
Low-degree lower bounds have been first introduced in \cite{Barak19} --see also~\cite{hopkins2018statistical}--, and then extended in many settings. 
\cite{KuniskyWeinBandeira} and \cite{WeinSchramm} provide some generic techniques for proving low-degree lower bounds in a wide range of situations. For example, building on these results, \cite{luo2023computational} provides evidence for the computation-information gap conjectured in \cite{Decelle2011}  for clustering in the Stochastic Block Model. 

In our setting, the random partition and the high number of dimensions cause a high dependence between the signal vectors. Thus, we have to appeal to delicate combinatorial arguments in order to upper-bound the mixed moments of the signal. We discuss this more precisely in the sketch of the proof of Theorem \ref{thm:lowdegreeclustering}, in Section \ref{sec:low-degree}, and in the proof of the theorem in Appendix \ref{prf:lowdegreeclustering}.

\paragraph{Outline and notation.}
We state our computational lower bound in Section~\ref{sec:low-degree}, we analyse the information-barrier for partial and perfect recovery in Section~\ref{sec:information}, 
and we discuss these results and their connections with the literature in Section~\ref{sec:discussion}. All the proofs are deferred to the appendices, though a sketch of the proof of the computational lower bound is provided in Section~\ref{sec:low-degree}.

Throughout this documents, we use $\|\cdot\|_{q}$ for the $L_{q}$ norm of a vector or of the entries of a matrix.
For $q=2$, we simply write $\|\cdot\|$ for the Euclidean norm of a vector, and $\|\cdot\|_F$ for the Frobenius norm of a matrix.
The notations $\|\cdot\|_{op}$ and $\|\cdot\|_{*}$ respectively stand for the operator norm  and the nuclear norm of a matrix.

We denote by $\mathcal{P}_{\alpha}$ the set of partitions of $[1,n]$ fulfilling (\ref{eq:balanced}).
For a partition $G=\ac{G_1,\ldots,G_K}$, we define $k^G_i$ as the integer such that $Y_i\in G_{k_i}$, and
the partnership matrix $M^G_{ij}=\1_{k^G_i=k^G_j}$. For $G=G^*$, we simply write $k^*_i=k^{G^*}_i$ and $M^*=M^{G^*}$.
We also define the proportion of misclassified points as

\begin{equation}\label{eq:error}
err({G},G^{*})=\frac{1}{2n}\min_{\pi\in\mathcal{S}_{K}}\sum_{k=1}^{K}|G_{k}^{*}\Delta{G}_{\pi(k)}|\enspace , 
\end{equation}

where $A\Delta B$ stands the symmetric difference between the sets $A$ and $B$, and $\mathcal{S}_{K}$ denotes the set of all permutations of $[1,K]$.

\section{Low degree polynomial lower bound}\label{sec:low-degree}

Low degree polynomials are not well suited for directly outputting a partition $\widehat{G}$, which is combinatorial in nature. Instead, we focus on the problem of estimating the partnership matrix $M^*_{ij}=\1_{k^*_i=k^*_j}$ with low-degree polynomials. 
It turns out that estimating $M^*$ and (partially) recovering the partition $G^*$ are closely related.
On the one hand, for any partition $G$, we have

\begin{align*}
    {1\over n(n-1)} \|M^G-M^*\|_F^2&\leq {1\over n(n-1)} \ \min_{\pi\in\mathcal{S}_{K}}\sum_{i\neq j=1}^n \pa{\1_{\pi(k^G_i)\neq k^*_i} \vee \1_{\pi(k^G_j)\neq k^*_j}}\\
    & \leq {2\over n} \ \min_{\pi\in\mathcal{S}_{K}}\sum_{i=1}^n \1_{\pi(k^G_i)\neq k^*_i} \ \leq 2\ err(G,G^*),
\end{align*}

so, if it is possible to cluster (in polynomial time) with error $err(\hat G,G^*)\leq \rho$, then we can estimate $M^*$ (in polynomial time) with error
$n^{-2}\|M^{\hat G}-M^*\|_F^2\leq 2\rho$.
On the other hand, for any estimator $\widehat M$, we have
\[\sum_{i,j=1}^n \1_{|\widehat M_{ij}-M^*_{ij}|\geq 1/2}\leq 4 \|\widehat M-M^*\|_F^2.\]
So, when $\|\widehat M-M^*\|_F^2< 1/4$, pairing together points $i,j$ fulfilling $\widehat M_{i,j}\geq 1/2$ provides a valid partition, equal to $G^*$.

To provide evidence of a computational barrier for the clustering problem, we build on this connection by proving a low-degree polynomial lower bound for the estimation of $M^*$. In this section, we consider the following generative prior for the partition $G^*$ and the means $\mu_1,\ldots,\mu_K$. 

\begin{definition}\label{def:prior}
    We draw $k_{1},\ldots, k_{n}$ i.i.d. uniformly on $[1,K]$. For a given $\bar{\Delta}>0$, independently from $(k_{i})_{i\in[1,n]}$,  we draw $\mu_{1},\ldots,\mu_{K}\in\R^{p}$ i.i.d. uniformly distributed on the hypercube $\mathcal{E}=\{+\eps,-\eps\}^{p}$, with $\eps^{2}=\frac{1}{p}\bar{\Delta}^{2}\sigma^2$. Then, conditionally on $(k_i)_{i=1,\ldots,n}$ and $(\mu_k)_{k=1,\ldots,K}$, the $Y_i$ are independent with $\mathcal{N}(\mu_{k_i}, \sigma^2I_p)$ distribution. The partition $G^{*}$ is obtained from the $k_{i}$'s with the canonical partitioning $G^{*}_{k}=\{i\in[1,n],\enspace k_{i}=k\}$, and $M^*_{ij}=\1_{k_i=k_j}$.
\end{definition}
We observe that in this model, for any $k\neq \ell$ the normalized square distance $\|\mu_k-\mu_\ell\|^2/(2\sigma^2)$ is equal to $\bar \Delta^2\pa{1+O_{\P}(p^{-1/2})}$.

Let $\R_{D}[Y]$ be the set of all polynomials in the observations $(Y_{ij})_{i,j\in[1,n]\times [1,p]}$ of degree at most $D$. We consider the degree-$D$ minimum mean squared error defined similarly as in \cite{WeinSchramm} by 
\begin{equation}\label{def:MMSE}
    MMSE_{\leq D}:=\inf_{f_{ij}\in\R_{D}[Y]}\ \ {1\over n(n-1)}\sum_{i\neq j=1}^n\E\cro{(f_{ij}(Y)-M^*_{ij})^{2}}\enspace.
\end{equation}
We observe that the trivial estimator $\widehat M_{ii}=1$ and $\widehat M_{ij}=1/K$ for $i\neq j$ has a mean square error 
${1\over n(n-1)} \E\cro{\|\widehat M-M^*\|_F^2}=\frac{1}{K}-\frac{1}{K^{2}}$. Our main result is the next theorem, which identifies a regime where low-degree polynomials cannot perform significantly better than the trivial estimator in the high-dimensional setting $p\geq n$. We refer to Appendix \ref{prf:lowdegreeclustering} for a proof of this theorem.

\begin{thm}\label{thm:lowdegreeclustering}
    Let $D\in \N$. If $p\geq n$ and $\zeta_{n}:=\frac{\bar\Delta^{4}D^{8}(1+D)^{4}}{p}\max\pa{\frac{n}{K^{2}},1}<1$, then under the prior of Definition \ref{def:prior}, we have
    \begin{equation}
        MMSE_{\leq D}\geq {1\over K} - \frac{1}{K^{2}}\pa{1 +  \frac{\zeta_n}{(1-\sqrt{\zeta_n})^3}} .
    \end{equation}
    In particular, if $\bar \Delta^2 \ll D^{-6}\pa{\sqrt{pK^2\over n}\wedge\sqrt{p}}$, then $MMSE_{\leq D}=\frac{1}{K}-\frac{1+o(1)}{K^{2}}$.
\end{thm}
{\bf Remark:} The same result holds (with a different power of $D$), when,  in  Definition \ref{def:prior}, the prior  on the $\mu_k$ is i.i.d.\  $\mathcal{N}(0,\eps I_p)$, instead of i.i.d.\ uniform on $\mathcal{E}=\{+\eps,-\eps\}^{p}$.
\medskip

The second part of Theorem \ref{thm:lowdegreeclustering} ensures that, when $p\geq n$,  low-degree polynomials with degree $D\leq (\log(n))^{1+\eta}$ do not perform better than the trivial estimator when 
\[\bar \Delta^2 \ll (\log n)^{-6(1+\eta)}\pa{\sqrt{pK^2\over n}\wedge\sqrt{p}}.\]
Since lower-bounds for low-degree polynomials with degree $D\leq (\log(n))^{1+\eta}$ are considered as evidence of the computational hardness of the problem, Theorem \ref{thm:lowdegreeclustering} suggests computational hardness of estimating $M^*$ when  $\bar \Delta^2 \ll (\log n)^{-6(1+\eta)}\pa{\sqrt{pK^2\over n}\wedge\sqrt{p}}$ and $p\geq n$. Since, as made explicit above, estimation of $M^*$ is possible in polynomial time  when clustering is possible in polynomial time,  this provides compelling evidence for the computational hardness of the clustering problem in this regime.
Conversely, we explained in the introduction, that non-trivial clustering is possible in polynomial-time under the almost matching condition $\Delta^2\gtrsim \sqrt{\frac{pK^2}{n}}\wedge \sqrt{p\log(n)}$.

\paragraph{Sketch of the proof} By linearity of the loss function, we only  consider, without loss of generality, the problem of estimating $x= M^*_{12}$ when $\sigma^2=1$. We need to prove that $$\inf_{f\in\R_{D}[Y]}\ \ \E\cro{(f(Y)-x)^{2}}\geq \frac{1}{K}- \frac{1}{K^{2}}\pa{1 +  \frac{\zeta_n}{(1-\sqrt{\zeta_n})^3}}.$$ Since $\E[x^2]= 1/K$, the problem boils down --see~\cite{WeinSchramm}--  to proving that, the so-called low degree correlation $corr_{\leq D}$ satisfies the following
\[
  corr_{\leq D}:=\underset{\E[f^{2}(Y)]=1}{\sup_{f\in\R_{D}[Y]}}\E(f(Y)x)\leq \frac{1}{K}\sqrt{1 +  \frac{\zeta_n}{(1-\sqrt{\zeta_n})^3}}\enspace . 
\]  

Interestingly, we can rewrite the observed matrix $Y$ as a Gaussian additive model $Y=X+E$, where $E\in \mathbb{R}^{n\times p}$ is made of independent standard normal entries, and $X= A \underline{\mu}$ where the matrix $A\in \{0,1\}^{n\times K}$ contains exactly one non-zero entry on each row and its position is sampled uniformly at random, and where the matrix $\epsilon^{-1}\underline{\mu}\in \mathbb{R}^{p\times K}$ is made of independent Rademacher random variables.

This allows us to apply the general results of~\cite{WeinSchramm}, which bound the low degree correlation in terms of a sum of cumulants
\begin{equation}\label{eq:corr_cumulant}
    corr^{2}_{\leq D}\leq \underset{|\alpha|\leq D}{\sum_{\alpha\in\N^{n\times p}}}\frac{\kappa_{\alpha}^{2}}{\alpha!}\enspace,
\end{equation}
where the $\alpha$'s run over all integer valued matrices whose sum is at most $D$, and where $\kappa_{\alpha}$ is the cumulants of the random variables $(x, \underset{\alpha_{1,1}}{\underbrace{X_{1,1},\ldots, X_{1,1}} },\ldots,  \underset{\alpha_{i,j}}{\underbrace{X_{i,j},\ldots, X_{i,j}}},\ldots)$. The bound~\eqref{eq:corr_cumulant}
turned out to be instrumental for establishing low degree polynomials lower bounds for submatrix estimation~\cite{WeinSchramm}, and for Stochastic Block model (SBM) estimation~\cite{luo2023computational}. In these two works, the authors follow a two-steps approach: first, they prune the sum in~\eqref{eq:corr_cumulant} by characterizing all the cumulants that are equal to zero. Second, they bound the cumulants as a polynomial sum of mixed moments. 

In comparison to the above works, we use the same general strategy, but the structure of the signal matrix $X$ is more involved. Indeed, in submatrix problem, the matrix $X$ only contains a single non-zero blocks whereas, for SBM, $X$ is, up to a permutation, a block-diagonal matrix. Here, we need to leverage on the fact that the rectangular matrix $X=A\underline{\mu}$ jointly involves a random partition matrix $A$ and an high-dimensional random matrix $\underline{\mu}$. 
As a consequence, we need to rely on more subtle arguments both for the pruning step, that is for characterizing null cumulants, and for bounding mixed moments with respect to the entries of $X$. 

For that purpose, we represent $\alpha\in \mathbb{N}^{n\times p}$ as a bi-partite multigraph $\mathcal{G}_{\alpha}$ between the set $[n]$ of points, and the set $[p]$ of variables and we write $\mathcal{G}^{-}_{\alpha}$ for its restriction to non-isolated nodes. In Lemma~\ref{lem:nullcumulant}, we first establish that the cumulant $\kappa_{\alpha}$ is null unless the graph $\mathcal{G}^{-}_{\alpha}$ satisfies the three following properties: (i) $\mathcal{G}^{-}_{\alpha}$ is connected, (ii)
Both the first and the second points belong to the connected component, and (iii)
Each variable in $\mathcal{G}^{-}_{\alpha}$ is connected to at least two distinct nodes.
Indeed, if at least one of these properties is not satisfied, it is possible to partition $(x, \ldots,  \underset{\alpha_{i,j}}{\underbrace{X_{i,j},\ldots, X_{i,j}}},\ldots)$ into two set of independent random variable, which implies the nullity of the cumulant.  

Now that we have pruned the sum in~\eqref{eq:corr_cumulant} by restricting ourselves to such matrices $\alpha$, we need to control the non-zero cumulants $\kappa_{\alpha}$. Since cumulants express as linear combination of moments, we bound mixed moments of the form $\mathbb{E}[X^{\gamma}]= \mathbb{E}[\prod_{i=1}^n\prod_{j=1}^p X_{i,j}^{\gamma_{ij}}]$ and $\mathbb{E}[xX^{\gamma}]$, for matrices $\gamma\in \mathbb{N}^{n\times p}$. We establish in Lemma~\ref{lem:moments} that 
\begin{equation}\label{eq:upper_moments}
\E[X^{\gamma}]\leq \eps^{|\gamma|}\min\pa{1,|\gamma|^{|\gamma|}\pa{\frac{1}{K}}^{l_{\gamma}-\frac{|\gamma|}{2}-CC_{\gamma}}}\enspace , 
\end{equation}
where $|\gamma|=\sum_{ij}\gamma_{ij}$ is the number of edges of $\mathcal{G}_{\gamma}^{-}$, $CC_{\gamma}$ is the number of connected components of $\mathcal{G}_{\gamma}^{-}$, and $l_{\gamma}$ is the number of nodes. 

For establishing~\eqref{eq:upper_moments}, we first rely on the fact that the entries of $\underline{\mu}$ are independent and follow a symmetric distributions. Since $X=A\underline{\mu}$,
where we recall that $A_{ik}=\1_{k_i=k}$ encodes the partition of the $n$ points, we have
\[
X^{\gamma}= \prod_{k=1}^{K}\prod_{j=1}^p \underline{\mu}_{kj}^{\sum_{i=1}^n A_{i,k}\gamma_{ij}}\ .
\]
The conditional expectation of $X^{\gamma}$ given $A$  is therefore non-zero (and is equal to $\epsilon^{|\gamma|}$), if and only if, $\sum_{i=1}^n A_{i,k}\gamma_{ij}$ is even for all $(k,j)$. We call the latter a $(A, \gamma)$ parity property. Since $\epsilon^{-1}\underline{\mu}_{kj}$ is a Rademacher random variable, it follows from the above that 
$$\mathbb{E}[X^{\gamma}]= \epsilon^{\gamma}\P[(A,\gamma) \text{ satisfies the parity property}]\ .$$
Next, we characterize in Lemma~\ref{lem:numbergroups} the partition matrices $A$ (or equivalently the partition $G^*$) that satisfy the parity property. In particular, we show that the partition induced by $G^*$ on the set of non-zero rows of $\gamma$ only contains a small number of groups. More precisely, we bound this number of groups in terms of $|\gamma|$, the number of non-zero rows of $\gamma$, the number of non-zero columns of $\gamma$, and $CC_{\gamma}$, the number of connected components of $\mathcal{G}_{\gamma}^-$. In turn, this condition on the number of groups enforces that $\P[(A,\gamma) \text{ satisfies the parity property}]$ is small. This combinatorial argument for establishing Lemma~\ref{lem:numbergroups}   is the main technical result in our proof. 

Finally, we build upon the mixed moment bounds~\eqref{eq:upper_moments} to control the cumulants $\kappa_{\alpha}$. Coming back to~\eqref{eq:corr_cumulant}, this allows us to conclude.

\section{Information barrier}\label{sec:information}
\subsection{Clustering below the computational barrier with exact Kmeans}
For a mixture of isotropic Gaussian, the partition $\hat{G}$ maximizing the likelihood is the exact $K$-means partitioning, which minimizes the criterion \begin{equation}\label{def:k-means}
    \hat{G}\in \argmin_{G\in\mathcal{G}_{K}} \mathrm{Crit(G)}\enspace ,\quad
    \text{where}\quad \mathrm{Crit(G)}=\sum_{k=1}^{K}\sum_{a\in G_{k}}\bigg\|Y_{a}-\frac{1}{|G_{k}|}\sum_{b\in G_{k}}Y_{b}\bigg\|^{2},
\end{equation}
with $\mathcal{G}_{K}$ the set of partitions of $[1,n]$ in $K$ groups. Minimizing $\mathrm{Crit}(G)$ is  NP-hard in general, and even hard to approximate \cite{awasthi2015hardness}.

Next theorem proves that exact $K$-means succeeds to produce non-trivial clustering for a separation smaller than the computational barrier (\ref{eq:separation}) for $p\geq n$, and for $K$ larger than some constant $K_0$.
\begin{thm}\label{thm:error_K_means:simple}
Assume that $G^*$ belongs to the set of balanced partitions $\mathcal{P}_\alpha$. 
    Then, there exist some constants $c$, $c'$, $c''$ depending only on $\alpha$, such that the following holds. If 
    \begin{equation}\label{eq:condition_snr:K}
    \Delta^{2}\geq c\pa{\log(K)\vee \sqrt{pK\log(K)\over n}} \enspace , 
    \end{equation}
    then, we have  with probability at least $1-c'/n^2$
    \begin{equation}\label{eq:exponentialdecrease}
        err(\hat{G},G^{*})\leq e^{-c''s^{2}},\quad \text{where} \quad s^2= \Delta^2 \wedge {n \Delta^4\over pK}.
    \end{equation}
\end{thm}
This result follows from the more precise Theorem \ref{thm:error_K_means} stated and proved in Appendix \ref{prf:thm_error_K_means:simple}.
The rate $e^{-c''s^{2}}$ for the proportion of misclassified points, matches the optimal probability of wrongly classifying a data point given the label of all the others \cite{giraud2019partial}. The term $\Delta^2$ in $s^2$ corresponds to the rate in low-dimension, while the term ${n \Delta^4\over pK}$ is induced by the minimal error $\sigma\sqrt{pK/n}$ for estimating the means $\mu_k$ in dimension $p$ with $n/K$ observations. 
We underline yet in Section~\ref{sec:discussion}, that the minimal separation   (\ref{eq:condition_snr:K}) for clustering cannot be readily derived from the minimal estimation rate for the means. 

In the high-dimensional setting $p\geq n$, Theorem \ref{thm:error_K_means:simple} ensures that for 
\[\Delta^2 \gtrsim \sqrt{pK\log(K)\over n} \]
the exponential exponent $c''s^2$ is larger than $(1+\eta)\log(K)$, so 
the proportion of misclustered points by exact Kmeans  is smaller than $1/K^{1+\eta}$. Exact Kmeans then performs a non-trivial clustering in this regime, breaking the computational barrier (\ref{eq:separation}) established in the previous section, when $K$ is larger than some constant $K_0$.

\subsection{Information lower bound}

For $\bar\Delta>0$, let $\Theta_{\bar\Delta}$ denote the set of $K$-tuples $\mu_{1},\ldots,\mu_{K}\in(\R^{p})^{K}$ that satisfy
$\Delta \geq \bar\Delta$, with $\Delta$ defined in (\ref{eq:Delta}).
Given $\mu_{1},\ldots, \mu_{K}\in(\R^{p})^{K}$ and $G$ a partition of $[1,n]$, we denote by $\P_{\mu,G}$ the probability distribution of the random variables $(Y_{1},\ldots ,Y_n)\in(\R^{p})^{n}$ generated as follows: $Y_{1},\ldots ,Y_n$ are independent, and $Y_{i}\sim\mathcal{N}(\mu_{k},\sigma^{2}I_{p})$ when $i\in G_{k}$. The next result, proved in Appendix \ref{prf:lowerboundpartial}, provides a minimax lower bound on the partial recovery of a partition.

\begin{thm}\label{thm:lowerboundpartial}
    There exist $c$,$c'$, $C$ and $K_{0}$ numerical constants such that the following holds. Assume that $p\geq c\log(K)$, $K\geq K_{0}$, $n\geq 2K$, and $\alpha\geq \frac{3}{2}$. Then, for any estimator $\hat{G}$, we have $$\sup_{G\in\mathcal{P}_{\alpha}}\sup_{\mu\in\Theta_{\bar \Delta}}\E_{\mu,G}[err(\hat{G},G)]\geq C\,,\quad\text{when}\quad \bar\Delta^{2}\leq c'\pa{\log(K)\vee \sqrt{\frac{pK\log(K)}{n}}}.$$
\end{thm}

 For exact Kmeans, according to 
(\ref{eq:exponentialdecrease}) from Theorem \ref{thm:lowerboundpartial},
non-trivial clustering is achieved for balanced clusters when
\[s^2= \Delta^2 \wedge {n \Delta^4\over pK} \gtrsim  \log(K),\]
or equivalently $\Delta^{2}\gtrsim {\log(K)\vee \sqrt{\frac{pK\log(K)}{n}}}$.
 Exact Kmeans is then information optimal for non-trivial clustering, and the information threshold for non-trivial clustering is 
\[\Delta^2\gtrsim {\log(K)\vee \sqrt{\frac{pK\log(K)}{n}}}.\]
We observe that this information barrier for clustering matches, up to a possible constant,  the information barrier $\bar\Delta^2\geq 2\sqrt{\gamma K\log(K)}+2\log(K)$  established in \cite{banks2018information} for detecting clusters in a Bayesian setting with a Gaussian prior  $\mathcal{N}(0,{p^{-1}\bar\Delta^2}I_p)$  on the $\mu_k$. Hence, there is no (significant) test-estimation gap at the information level. We refer to Section~\ref{sec:discussion} for a comparison of the information rates for clustering and estimation. 

We complement this result with a lower bound for perfect recovery of the planted partition, proved in Appendix \ref{prf:lowerboundexact}.

\begin{thm}\label{thm:lowerboundexact}
    There exist numerical constants $c$, $C$ and $n_{0}$ such that the following holds. Assume that  $n\geq 9K/2$, $\alpha\geq \frac{3}{2}$ and $n\geq n_{0}$. Then, for any estimator $\hat{G}$, $$\sup_{G\in\mathcal{P}_{\alpha}}\sup_{\mu\in\Theta_{\bar\Delta}}\P_{\mu,G}\cro{\hat{G}\neq G}>C\,, \quad\text{when}\quad \bar\Delta^{2}\leq c\pa{\log(n)\vee \sqrt{\frac{pK\log(n)}{n}}}.$$
\end{thm}
For $K\leq \log(n)$, we recover the optimal separation from \cite{ndaoud2022sharp} ($K=2$) and \cite{chen2021cutoff}, up to a multiplicative constant.
Perfect recovery corresponds to a proportion of misclustered points $err(\hat G,G^*)$ smaller than $1/n$. For exact Kmeans, according to 
(\ref{eq:exponentialdecrease}), perfect recovery is achieved for 
\[s^2= \Delta^2 \wedge {n \Delta^4\over pK} \gtrsim  \log(n),\]
or equivalently $\Delta^{2}\gtrsim {\log(n)\vee \sqrt{\frac{pK\log(n)}{n}}}$.
 Exact Kmeans is then also optimal for exact recovery, and the 
information threshold for perfect clustering is then
\begin{equation}\label{eq:perfect:separation}
    \Delta^2\gtrsim {\log(n)\vee \sqrt{\frac{pK\log(n)}{n}}}.
\end{equation}
When $K\lesssim \log(n)$, \cite{giraud2019partial} shows that an SDP relaxation of Kmeans \cite{PengWei07} also succeeds to perfectly recover the clusters when (\ref{eq:perfect:separation}) is met, so there is no separation in this regime --see also \cite{chen2021cutoff}. Yet, in the high-dimensional setting $p\geq n$, we observe that the threshold (\ref{eq:perfect:separation}) is smaller than the computational barrier (\ref{eq:separation}) when $K\gtrsim \log(n)$, so there is also a computation-information gap for perfect recovery  in this regime, thereby confirming the conjecture of~\cite{chen2021cutoff}.

\section{Discussion}\label{sec:discussion}
\subsection{Comparison to the moderately low-dimensional setting}
\paragraph{No non-isotropic effect.}
We emphasize that compared to the moderately low-dimensional setting $n\geq \text{poly}(p,K)$,  the computational hardness of clustering in the high-dimensional regime $p\geq n$ is not driven by any non-isotropic effect. Indeed, contrary to the low-dimensional setting where there is no computation-information gap for learning mixture of isotropic Gaussian \cite{Regev2017,LiuLi2022}, we prove the computation-information gap for a mixture of Gaussians with  covariances known to be all equal to the identity. Furthermore, there is no difference between the Gaussian and the sub-Gaussian setting, in the sense that for a mixture of possibly anisotropic sub-Gaussian distribution, clustering is also possible in polynomial time above the computational barrier (\ref{eq:separation}) established for Gaussian mixture, for example with an SDP relaxation of Kmeans \cite{PengWei07, giraud2019partial}
for $K\leq \sqrt{n}$, or with single linkage hierarchical clustering \cite{HDS2} for $K>\sqrt{n}$.  

\paragraph{Comparison to moderately low dimension.}
Contrary to the high-dimensional setting, where the computational hardness seems tightly related to the BBP transition  for the largest eigenvalue of the Gram matrix of the observations \cite{BBP05,Paul07}, the computational hardness in moderately low-dimensional settings $n\geq \text{poly}(p,K)$ is completely driven by the unknown non-isotropy  of the components of the mixture. A first example of non-isotropic mixture giving rise to a computation-information gap is the so-called example of "parallel pancakes" \cite{DiakonikolasFOCS17}. In this example, the $K$ unknown centers of the Gaussian distribution are aligned along an unknown direction $v$, and the unknown covariances are all equal to the identity, except in the direction $v$, where they are very thin. The key feature of this construction, is that the $2K-1$ first moments of the mixture distribution match those of a standard Gaussian, so that it is impossible to figure out the direction $v$ from the $2K-1$ first moments. As a consequence, \cite{DiakonikolasFOCS17} proves a lower-bound for the running time of any Statistical-Query algorithm (SQ-algorithm), enforcing a computation-information gap between SQ-algorithms and information-optimal algorithms in this moderately  low-dimensional setting.  

This approach has been extended by \cite{diakonikolasCOLT23b}, who has adapted this construction for centers with separation $\Delta^2\geq k^{\eta}$ much larger than the information-minimal separation $\Delta^2\gtrsim\log(K)$ \cite{Regev2017} in moderately  low dimension.
For such a large separation, the centers are drawn according to a standard Gaussian on a (unknown) random subspace of dimension $d\approx \Delta^2$. 
\cite{diakonikolasCOLT23b} then proves again  a lower-bound for the running time of any Statistical-Query algorithm (SQ-algorithm), enforcing a computation-information gap between SQ-algorithms and information optimal algorithms in this (moderately low-dimensional) setting.  

\subsection{Comparing estimation and clustering rates}
Assume with no loss of generality that $\sigma^2=1$.
Theorems \ref{thm:error_K_means:simple}-\ref{thm:lowerboundpartial} show that the information-minimal separation for clustering is 
$\Delta^2\gtrsim \log(K)\vee \sqrt{Kp\log(K)/n}$. When the separation $\Delta^2$ is larger than $\log(K)$, it is known that the information-minimal rate for estimating the means $\mu_k$ is at least $\sqrt{Kp/n}$. This rate stems from the fact that we estimate $p$-dimensional vectors with about $n/K$ observations for each of them. In the moderately low-dimensional regime $n\geq pK^3$,
estimation at this rate (up to possible log factors) is actually information-possible  \cite{Kwon20}.  We then underline that for $\log(K)\leq Kp/n$, at the information-minimal separation for clustering $\Delta^2\asymp \sqrt{Kp\log(K)/n}\geq \log(K)$, we cannot estimate the means $\mu_k$ better than with a precision $\sqrt{Kp/n}\asymp \Delta^2$, up to log factors. This precision  is much larger than the minimum distance $\sqrt{2}\Delta$ between the means. In particular, a natural {\it estimate-then-cluster} strategy that would consist in (i) estimating the means with precision at least $\Delta$, and then (ii) apply Linear Discriminant Analysis with the estimated means $\hat \mu_k$, would require a separation at least $\Delta^2\gtrsim \log(K)\vee \pa{Kp/n}$ (up to possible log factors), which is much larger than the   information-minimal separation (\ref{eq:condition_snr:K}) for clustering. 
The message is then that optimal rates for the estimation problem do not directly provide useful information for the clustering problem.

\subsection{Computation-information gap for partnership matrix estimation}
While our primary interest is on clustering, we point out below that our results provide evidence for the existence of a computation-information gap for the estimation of the partnership matrix $M^*$ in Frobenius norm, in the high-dimensional regime $p\geq n$. 

As discussed in Section \ref{sec:low-degree}, starting from a partition $\hat G$, we can estimate the partnership matrix $M^*_{ij}=\1_{k^*_i=k^*_j}$ with $M^{\hat G}=\1_{k^{\hat G}_i=k^{\hat{G}}_j}$.  The mean squared error ${1\over n(n-1)}\|M^{\hat G}-M^*\|^2_F$ is then upper bounded by twice the clustering error $err(\hat G, G^*)$. 
Relying on this connection, we prove below that, when the dimension is high $p\geq n$, and when there is a large number of points $n\gtrsim K^{2}\log(n)$, the exact Kmeans provides an estimation of $M^*$  below the computational barrier in the generative model of Definition \ref{def:prior}.

\begin{cor}\label{cor: squareKmeans}
Let us consider the generative model of Definition \ref{def:prior}. Assume that $n\geq cK^{2}\log(n)$, with $c>0$ a numerical constant, and $p\geq n$. There exists $c'$ and $C$, two numerical constants, such that  the partnership matrix estimation induced by the exact Kmeans partition $\hat G$ fulfills
\[ {1\over n(n-1)}\E\cro{\|M^{\hat G}-M^*\|_F^2}\leq {C\over n^2}\ ,\quad \text{when}\quad \Delta^{2}\geq c'\sqrt{\frac{pK\log(n)}{n}}\,.\]
\end{cor}
We refer to Section \ref{prf:squareKmeans} for a proof of this corollary of Theorem \ref{thm:error_K_means} stated in Appendix \ref{prf:thm_error_K_means:simple},  which slightly generalizes Theorem \ref{thm:error_K_means:simple}. 
Thus, when $ \sqrt{\frac{pK}{n}\log(n)}\lesssim\Delta^{2}\ll \sqrt{\frac{pK^{2}}{n}}$, the error obtained with the exact $K$-means estimator decays much  faster than the trivial error $\frac{1}{K}$ obtained with the trivial estimator $\rho$, and with the best polynomial of degree at most $D(n)=\log(n)^{1+\eta}$, when $\Delta^{2}\ll {(\log n)^{-6(1+\eta)}}\sqrt{{pK^{2}}/{n}}$.

\subsection{Limitations}
Inspired from \cite{WeinSchramm}, our analysis for establishing a low-degree polynomial lower bound has the nice feature to  provide a rigorous and non-asymptotic computational lower-bound, but it has the drawback 
to be a bit rough, and a spurious polylog$(n)$ factor shows up  in the computational lower bound.  
Removing this polylog$(n)$ factor in the proof would probably require a different strategy for bounding the correlation $corr^2_D$, by precisely keeping track of all terms. 
The complexity of such an analysis would go well beyond the complexity of our proof of Theorem~\ref{thm:lowdegreeclustering}. Given the already high complexity of our proof, we do not intend to pursue in this direction.

Another drawback of our analysis is that it is limited to the dimension range $p\geq n$, while a computation-information gap may also exist for smaller values of the ambient dimension $p$. Indeed, when the means $\mu_k$ are drawn i.i.d. with  Gaussian $\mathcal{N}(0,{p^{-1}\bar\Delta^2}I_p)$  distribution in $\R^p$, and in the asymptotic regime where $n$, $p$ go to infinity with $p/n\to \gamma \in [(K/2-2)^{-2},+\infty)$, \cite{lesieur2016phase} conjectures that 
non-trivial clustering is possible  in polynomial time only  for $\bar\Delta^2> \sqrt{\gamma K^2}$, while it is possible without computational constraints for $\bar\Delta^2\gtrsim \sqrt{\gamma K\log(K)}\vee \log(K)$, which is smaller for $K\gtrsim 1\vee (\gamma^{-1/2}\log K)$.
Similarly, for the problem of detecting the existence of clusters in the same specific setting, \cite{banks2018information} shows that spectral detection is not possible at the information threshold  for $\gamma\gtrsim (\log(K)/K)^2$. 
Indeed, from BBQ transition, the largest eigenvalue of the Gram matrix of the data points singles out of the bulk of the spectrum only for $\bar\Delta^2\geq \sqrt{\gamma K^2}$, while detection is information possible for $\bar\Delta^2\geq 2\sqrt{\gamma K\log(K)}+2\log(K)$.
These two results suggest the existence of a computation-information gap not only for $p\geq n$, as considered in this paper, but more generally for $p\gtrsim n(\log(K)/K)^2$ and $K\geq K_0$.
In Appendix~\ref{sec:p-smaller-than-n}, we adapt the proof of Theorem~\ref{thm:lowdegreeclustering} in order to provide a computational lower-bound when $p\leq n$. We believe that our computational barrier is not tight in this regime, yet it already provides evidence for the existence of a computation-information gap  when
$${n\over K}\vee K\ \widetilde{\ll}\ p\leq n\,,$$
where $\ \widetilde{\ll}\ $ hides polylog$(n)$ factors. 
Proving a more tight  computational barrier for the range $n(\log(K)/K)^2\lesssim p\leq n$ is left for future investigation.

\newpage
\printbibliography
\appendix
\newpage

\section{Proof of Theorem \ref{thm:lowdegreeclustering}}\label{prf:lowdegreeclustering}

With no loss of generality, we assume in all the proof that $\sigma^2=1$.
Let $D\in \N$. We recall the assumption that $p\geq n$, and $$\zeta_{n}=\frac{\Bar{\Delta}^{4}D^{8}(1+D)^4}{p}\max\pa{\frac{n}{K^2},1}<1.$$ 
Since the minimization problem defining $MMSE_{\leq D}$ in Equation (\ref{def:MMSE}) is separable, and since the random variables $M^*_{ij}$ are exchangeable, the $MMSE_{\leq D}$ can be reduced to
\begin{align*}
    MMSE_{\leq D}&=\frac{1}{n(n-1)}\sum_{i\neq j=1}^{n}\inf_{f_{ij}\in \R_{D}(Y)}\E\cro{\pa{f_{ij}(Y)-M^*_{ij}}^2}\\
&=\inf_{f\in \R_{D}(Y)}\E\cro{\pa{f(Y)-M^*_{12}}^2}\enspace.
\end{align*}
In the remaining of the proof, we write $x=M^*_{12}=\1_{k_1=k_2}$. Then, our goal is to upper-bound 
$$MMSE_{\leq D}=\inf_{f\in \R_{D}(Y)}\E\cro{\pa{f(Y)-x}^2}\enspace.$$
As noticed by \cite{WeinSchramm}, the  $MMSE_{\leq D}$ can be further decomposed as 
\begin{equation*}
    MMSE_{\leq D}=\E[x^{2}]-corr^2_{\leq D}= {1\over K}-corr^2_{\leq D} \enspace,
\end{equation*}
where $corr^2_{\leq D}$ is the  
 degree-D maximum correlation 
\begin{equation}\label{def:corr}
    corr_{\leq D}:=\underset{\E[f^{2}(Y)]=1}{\sup_{f\in\R_{D}[Y]}}\E(f(Y)x)=\underset{\E\pa{f^{2}(Y)}\neq 0}{\sup_{f\in\R_{D}[Y]}}\frac{\E[f(Y)x]}{\sqrt{\E(f^{2}(Y))}}\enspace.
\end{equation}
Hence, in order to prove Theorem \ref{thm:lowdegreeclustering}, it is enough to prove that 
\begin{equation*}
    corr^2_{D}\leq \frac{1}{K^2}\pa{1+\frac{\zeta_{n}}{(1-\sqrt{\zeta_n})^3}}\enspace.
\end{equation*}
The model of Definition \ref{def:prior} is a particular instance of the Additive Gaussian Noise Model considered in \cite{WeinSchramm}. 
Hence, we can use Theorem 2.2 from \cite{WeinSchramm} that we recall here. For a matrix $\alpha\in\N^{n\times p}$, we define $|\alpha|=\sum_{i=1}^{n}\sum_{j=1}^{p}\alpha_{ij}$ and $\alpha!=\prod_{i=1}^{n}\prod_{j=1}^{p}\alpha_{ij}!$. Given another matrix $\beta\in\N^{n\times p}$, we write $\binom{\alpha}{\beta}=\prod_{i=1}^{n}\prod_{j=1}^{p}\binom{\alpha_{ij}}{\beta_{ij}}$. Finally, given a matrix $Q\in\R^{n\times p}$, we write $Q^{\alpha}=\prod_{i=1}^{n}\prod_{j=1}^{p}Q_{ij}^{\alpha_{ij}}$. We define $X\in\R^{n\times p}$ the signal matrix whose $i$-th row is the vector $\mu_{k_{i}}$, so that, conditionally on $X$, the $Y_{ij}$ are independent with $\mathcal{N}(X_{ij},1)$ distribution. Throughout this proof, we write  $X_i=\mu_{k_{i}}$.
\begin{prop}\label{thm:schrammwein}\cite{WeinSchramm}
    The degree $D$ maximum correlation satisfies the upper-bound 
    \begin{equation}\label{eq:corr_D}
    corr^{2}_{\leq D}\leq \underset{|\alpha|\leq D}{\sum_{\alpha\in\N^{n\times p}}}\frac{\kappa_{\alpha}^{2}}{\alpha!}\enspace,
    \end{equation}
    where $\kappa_{\alpha}$ for $\alpha\in \N^{n\times p}$ is defined recursively by 
\begin{equation}\label{eq:def:kappa}
    \kappa_{\alpha}=\E[xX^{\alpha}]-\sum_{\beta \lneq \alpha}\E[X^{\alpha-\beta}]\binom{\alpha}{\beta}\kappa_{\beta}\enspace.
\end{equation}
\end{prop}
\cite{WeinSchramm} observe that, for $\alpha\in\N^{n\times p}$, the quantity $\kappa_{\alpha}$ corresponds to the cumulant $\kappa(x,X_{a_{1}},\ldots ,X_{a_{m}})$, where $\{a_{1},\ldots ,a_{m}\}$ is the multiset that contains $\alpha_{ij}$ copies of $(i,j)$, for $i,j\in[1,n]\times [1,p]$. We refer e.g. to the lecture notes~\cite{novak2014three} for more details on cumulants.
In the remainder of the proof, we first characterize the matrices $\alpha$ for which $\kappa_{\alpha}\neq 0$, and we provide an upper bound on the corresponding cumulants.

Let us first provide sufficient conditions on $\alpha$, so that the corresponding cumulant $\kappa_{\alpha}$ is zero.  For that purpose, it is convenient to represent $\alpha\in \N^{n\times p}$ as a bipartite multi-graph. 
More precisely, we define $\mathcal{G}_{\alpha}$ as the bipartite multi-graph on two disjoint sets of nodes $U=\{u_{i},i\in[1,n]\}$ and $V=\{v_{j}, j\in[1,p]\}$, with $\alpha_{ij}$ edges between $u_i$ and $v_j$, for $i,j\in [1,n]\times [1,p]$. 
 For a given $\alpha \in \N^{n\times p}$, the $\ell^1$-norm  $|\alpha|= \sum_{i=1}^n\sum_{j=1}^{p}\alpha_{ij}$ corresponds to the number of multi-edges of $\mathcal{G}_{\alpha}$. Besides, we also denote by $\mathcal{G}^-_{\alpha}$ the graph $\mathcal{G}_{\alpha}$ from which we have removed isolated nodes, and we write $l_{\alpha}$  for the number of nodes of  $\mathcal{G}^{-}_{\alpha}$. We define $m_{\alpha}$ (resp. $r_{\alpha})$ as the number of nodes of $\mathcal{G}^{-}_{\alpha}\cap U$  (resp. $\mathcal{G}^{-}_{\alpha}\cap V$), so that $l_{\alpha}= m_{\alpha}+ r_{\alpha}$. The following lemma is proved in Section~\ref{prf:boundcumulants}.

\begin{lem}\label{lem:reductioncumulant}
    Let $\alpha\in\N^{n\times p}$ be non-zero such that $\kappa_{\alpha}\neq 0$. Then, $\mathcal{G}^{-}_{\alpha}$ is connected and contains both $u_{1}$ and $u_{2}$. Moreover, all the nodes $v_{j}$ of  $\mathcal{G}^{-}_{\alpha}$
    are connected to at least two distinct nodes of  $\mathcal{G}^{-}_{\alpha}$.
\end{lem}
This lemma is proved in Section \ref{prf:recuctioncumulant}. The proof  mostly relies on the  property that a cumulant of two sets of independents random variable is zero. 
As a corollary, we deduce the following properties for matrices $\alpha$, such that $\kappa_{\alpha}$ is non-zero.

\begin{lem}\label{lem:condition_topology}
Let $\alpha\in \N^{n\times p}$ be non-zero.   If $\kappa_{\alpha}\neq 0$, then $m_{\alpha}\geq 2$, $|\alpha|\geq 2r_{\alpha}$ and $|\alpha|\geq r_{\alpha}+m_{\alpha}-1$.
\end{lem}

\begin{proof}[Proof of Lemma~\ref{lem:condition_topology}]
The first property ($m_{\alpha}\geq 2$) holds because $u_{1}$ and $u_{2}$ are spanned by $\mathcal{G}^{-}_{\alpha}$. The last property ($|\alpha|\geq r_{\alpha}+m_{\alpha}-1$) holds because the multigraph $\mathcal{G}^{-}_{\alpha}$ is connected. Finally, we know that each node of $\mathcal{G}^{-}_{\alpha}$ that is also in $V$ has degree at least $2$. Since the graph $\mathcal{G}^{-}_{\alpha}$ is bipartite, all these edges are distinct and we deduce that $|\alpha|\geq 2r_{\alpha}$.
\end{proof}

In order to upper-bound $corr_{\leq D}^{2}$, we need to upper-bound the cumulants $\kappa_{\alpha}$ for all $\alpha\in\N^{n\times p}$ that satisfy the conditions of Lemma \ref{lem:reductioncumulant}. 

\begin{lem}\label{lem:boundcumulants}
    Let $\alpha\in\N^{n\times p}$ be such that $\kappa_{\alpha}\neq 0$. We have 
\begin{equation}\label{eq:borne:cumulants}
     |\kappa_{\alpha}|\leq \eps^{|\alpha|}\pa{1+|\alpha|}^{|\alpha|}\min\pa{\frac{1}{K},|\alpha|^{|\alpha|}\pa{\frac{1}{K}}^{l_{\alpha}-\frac{|\alpha|}{2}-1}}\enspace.
\end{equation}
\end{lem}

Now we gather the three above lemmas to control the sum $\sum_{\alpha:\  |\alpha|\leq D}\kappa_{\alpha}^{2}/\alpha!$ in \eqref{eq:corr_D}. We reorganize the sum over $\alpha$ by organizing it according to the values of $m_{\alpha}$, $r_{\alpha}$, and $|\alpha|$, which respectively correspond to the number of $u$-nodes, $v$-nodes and edges in $\mathcal{G}^-_{\alpha}$.

\begin{lem}\label{lem:combinatorics_kappa}
Given $m\geq 2$, $r\geq 1$, $d\geq \max(r+m-1,2r)$, there exists at most $p^{r}n^{m-2}d^{2d}$ matrices $\alpha\in \N^{n\times p}$ such $\kappa_{\alpha}\neq 0$,  $m_{\alpha}=m$, $r_{\alpha}=r$ and $|\alpha|=d$.
\end{lem}

\begin{proof}[Proof of Lemma~\ref{lem:combinatorics_kappa}]
Since $\kappa_{\alpha}\neq 0$, Lemma~\ref{lem:reductioncumulant} ensures that both $u_1$ and $u_2$ are nodes of $\mathcal{G}^{-}_{\alpha}$. Hence, there are less than $n^{m-2}p^{r}$ possibilities for choosing the remaining nodes. By assumption $d\geq \max(m,r)$.
For each edge, there are at most $mr\leq d^2$ possibilities. Since $\mathcal{G}^-_{\alpha}$ has $d$ edges, we have less than $d^{2d}$ possibilities for choosing these edges. Since $\mathcal{G}^{-}_{\alpha}$ is one to one with $\alpha$, we conclude that there are less than $p^{r}n^{m-2}d^{2d}$ matrices $\alpha$ satisfying the given constraints. 
\end{proof}

Combining the bounds on the cumulants of Lemma \ref{lem:boundcumulants} and Lemma \ref{lem:combinatorics_kappa}, we are in position to control  $    corr_{D}^{2}$. We slice the sum of the cumulants according to $r_{\alpha},m_{\alpha}$ and $d_{\alpha}$, and we use below the notation $\mathcal{D}_d=\ac{(r,m)\in [1,d]\times [2,d]: \max(m+r-1,2r)\leq d}$. We recall that  $X_{ij}$  takes value in $\ac{-\eps,+\eps}$, with $\eps=\bar \Delta/\sqrt{p}$, so that
\begin{align}\nonumber
    corr_{D}^{2}\leq&  \underset{|\alpha|\leq D}{\sum_{\alpha\in\N^{n\times p}}}\kappa_{\alpha}^{2}\\ \nonumber
    \leq &\kappa_{0}^{2}+\sum_{d=1}^{D} \sum_{(r,m)\in\mathcal{D}_d}
    p^{r}n^{m-2}d^{2d}\eps^{2d}(1+d)^{2d}\min\pa{\frac{1}{K^{2}},d^{2d}\pa{\frac{1}{K}}^{2m+2r-d-2}} \\
    \leq & \frac{1}{K^2}+\sum_{d=1}^{D}\sum_{(r,m)\in\mathcal{D}_d}
    p^{r}n^{m-2}(\eps^{2}D^4(1+D)^{2})^d\min\pa{\frac{1}{K^{2}},\pa{\frac{1}{K}}^{2m+2r-d-2}}. \label{eq:upper_corr_D:first} 
\end{align}    
Since $\zeta_{n}=\frac{\bar\Delta^4D^{8}(1+D)^{4}}{p}\max\pa{\frac{n}{K^{2}},1}$, with $\bar\Delta^2 = p \epsilon^2$, we get
\begin{align} 
    \nonumber
\lefteqn{corr_{D}^{2}-\frac{1}{K^2}}  \\ \nonumber
    \leq & \sum_{d=1}^{D}\sum_{(r,m)\in\mathcal{D}_d}\zeta_n^{d/2}p^{-(d/2-r)}n^{m-2}\pa{\frac{1}{\max(1,n/K^2)}}^{d/2}\min\pa{\frac{1}{K^{2}},\pa{\frac{1}{K}}^{2m+2r-d-2}}  \\ 
    \leq &  \frac{1}{K^2}\sum_{d=1}^{D}\sum_{(r,m)\in\mathcal{D}_d}\zeta_n^{d/2}n^{r+m-d/2- 2}\pa{\frac{1}{\max(1,n/K^2)}}^{d/2}\min\pa{1,\pa{\frac{1}{K}}^{2m+2r-d-4}}, \label{eq:upper_corrD}
\end{align}
where we used in the last line that $n\leq p$ and $r\leq d/2$.
Let us check that each term in the sum is upper-bounded by $\zeta_n^{d/2}$, by considering apart the cases $d/2\geq m+r-2$ and $d/2< m+r-2$.  

 When $d/2\geq m+r-2$, the exponent of $n$ is non-positive, so that
\[
n^{r+m-d/2- 2}\pa{\frac{1}{\max(1,n/K^2)}}^{d/2}\min\pa{1,\pa{\frac{1}{K}}^{2m+2r-d-4}} \leq  1\ . 
\]
When $d/2 < m+r-2$, we can upper bound the minimum by $K^{-(2m+2r-d-4)}$, so that
\begin{align*}
    n^{r+m-d/2- 2}\pa{\frac{1}{\max(1,n/K^2)}}^{d/2}&\min\pa{1,\pa{\frac{1}{K}}^{2m+2r-d-4}}\\ &\leq \pa{\frac{n}{K^2}}^{m+r-d/2-2}\pa{\frac{1}{\max(1,n/K^2)}}^{d/2}. 
\end{align*}
If $n\leq  K^2$, the latter expression is smaller or equal to one. If $n\geq K^2$, this last expression, is equal to $(n/K^2)^{m+r-d-2}$ and is also smaller or equal to one since $d\geq m+r-1$ for $(r,m)\in\mathcal{D}_d$. Back to~\eqref{eq:upper_corrD}, and relying on the assumption $\zeta_n<1$, we conclude that 
\begin{align} \nonumber
    corr_{D}^{2}&\leq\frac{1}{K^2}+  \frac{1}{K^2}\sum_{d=1}^{D} \sum_{(r,m)\in\mathcal{D}_d}
    \zeta_n^{d/2} \\ \nonumber
    &\leq \frac{1}{K^2}\left[1 + \sum_{d=2}^{D}\frac{d(d-1)}{2} \zeta_n^{d/2} \right]\\
    &\leq \frac{1}{K^2}\left[1 +  \frac{\zeta_n}{(1-\sqrt{\zeta_n})^3} \right]
    \enspace \  . \label{eq:upper_corrD_bis}
\end{align}

\subsection{Proof of Lemma \ref{lem:reductioncumulant}}\label{prf:recuctioncumulant}

In order to prove Lemma \ref{lem:reductioncumulant}, we will use a classical property of cumulants, that we recall here. 
\begin{lem}\label{lem:nullcumulant}[e.g.~\cite{novak2014three}]
    Let $X_{1},\ldots ,X_{r}$ be random variables on the same space $\Omega$. If there exists a partition $A, B$ of $[1,r]$ such that $(X_{i})_{i\in A}$ is independant from $(X_{i})_{i\in B}$, then the cumulant $\kappa(X_{1},\ldots ,X_{r})$ is zero. 
\end{lem}

We  prove below that  $\kappa_{\alpha}=0$ if one of the three following properties is satisfied:
\begin{itemize}
    \item[(i)] $u_1$ or $u_2$ are not spanned by $\mathcal{G}^{-}_{\alpha}$,
    \item[(ii)] a node of $\mathcal{G}^{-}_{\alpha}$ which also belong to $V$ is connected to at most one node in $U$,
    \item[(iii)] $\mathcal{G}^{-}_{\alpha}$ is not connected.
\end{itemize}

\medskip

We denote by $U_{\alpha}\subset U$ (resp. $V_{\alpha}\subset V$) the set of nodes of $\mathcal{G}^{-}_{\alpha}$ that also belong to $U$ (resp. $V$). We denote by $E_{\alpha}$ the set of edges of $\mathcal{G}_{\alpha}^{-}$.

\medskip 

Let us first show that (i) is a sufficient condition for $\kappa_{\alpha}=0$. 
By symmetry, we suppose that $u_{1}$ is not spanned by $\mathcal{G}^{-}_{\alpha}$. Then, $k_{1}$ (the group corresponding to $u_1$) is independent from the family of random variables $(X_{ij})_{u_{i},v_{j}\in E_{\alpha}}$. Hence, the random variable $x=\1_{k_{1}=k_{2}}$ is also independent from $(X_{ij})_{u_{i},v_{j}\in E_{\alpha}}$. Together with Lemma \ref{lem:nullcumulant}, this implies the nullity of the cumulant $\kappa_{\alpha}$.

\medskip

Then, let us show that (ii) is also a sufficient condition for $\kappa_{\alpha}=0$.
We suppose that there exists $j_{0}\in[1,p]$ such that $v_{j_{0}}$ is connected with only one node $u_{i_{0}}\in U$. Conditionally on $(k_{i})_{i\in [1,K]}$ and on $(\mu_{k,j})_{k,j\in [1,k]\times ([1,p]\setminus \{j_{0})\}}$, the variable $X_{i_{0},j_{0}}$ is uniformly distributed on $\{-\eps,\eps\}$. This implies the independence of $X_{i_{0},j_{0}}$ with $((X_{ij})_{u_{i},v_{j}\in E_{\alpha}\setminus (i_{0},j_{0})}, x)$. Lemma \ref{lem:nullcumulant} then leads to the nullity of the cumulant $\kappa_{\alpha}$.

\medskip

Finally, let us show that (iii) is a sufficient condition for $\kappa_{\alpha}=0$.
Let $\alpha\in\N^{n\times p}$ such that  $\mathcal{G}^{-}_{\alpha}$ has at  least two connected components.  Let us denote $C_1$ and $C_2$ two of these connected components. At least one of them does not contain both $u_1$ and $u_2$. We suppose by symmetry that $C_1$ does not contain both these nodes (we suppose that it does not contain $u_{2}$ for example). We denote $E_{1}=E_{\alpha}\cap \pa{(U\cap C_{1})\times (V\cap C_{1})}$ which corresponds to the edges of $\mathcal{G}_{\alpha}^{-}$ which connect points from $U_{\alpha}\cap C_{1}$ to points from $V_{\alpha}\cap C_{1}$. We will show that the families of random variables $(X_{ij})_{u_{i},v_{j}\in E_{1}}$ and $(x,(X_{ij})_{u_{i},v_{j}\in E_{\alpha}\setminus E_{1}})$ are independent, which will lead to the nullity of $\kappa_{\alpha}$, using Lemma \ref{lem:nullcumulant}. 

For sake of clarity, we begin by dealing with the simple case where $C_1$ also does not contain $u_{1}$. So, the intersection of $C_{1}$ with $\{u_{1},u_{2}\}$ is empty. For $u_{i},v_{j}\in E_{\alpha}$, $X_{ij}=\mu_{k_{i},j}$. By definition of our model, the family $\pa{(k_{i})_{u_{i}\in C_{1}\cap U},(\mu_{k,j})_{k\in[1,K], v_{j}\in C_{1}\cap V}}$ is independent from the family $\pa{\pa{k_{i}}_{u_{i}\in U\setminus C_{1}}, (\mu_{k,j})_{k\in[1,K], v_{j}\in V\setminus C_{1}}}$. 
On the one hand, the random variables $(X_{ij})_{u_{i},v_{j}\in E_{1}}$ are measurable with respect to $(k_{i})_{u_{i}\in C_{1}\cap U}$ and $(\mu_{k,j})_{k\in[1,K], v_{j}\in C_{1}\cap V}$. 
On the other hand, the random variables $(x=\1_{k_{1}=k_{2}},(X_{ij})_{u_{i},v_{j}\in E_{\alpha}\setminus E_{1}})$ are measurable with respect to $(k_{i})_{u_{i}\in U\setminus C_{1}}$ and  $(\mu_{k,j})_{k\in[1,K], v_{j}\in V\setminus C_{1}}$. 
This leads to the independence of $(X_{ij})_{u_{i},v_{j}\in E_{1}}$ with $(x,(X_{ij})_{u_{i},v_{j}\in E_{\alpha}\setminus E_{1}})$. Lemma \ref{lem:nullcumulant} implies the nullity of $\kappa_{\alpha}$.

\medskip

Now, let us deal with the more complex case where $C_1$ contains $u_{1}$. Again, the random variables $x$ and $(X_{ij})_{u_{i},v_{j}\in E_{\alpha}\setminus E_{1}}$ are measurable with respect to $(k_{i})_{u_{i}\in U\setminus C_{1}}$ and  $(\mu_{k,j})_{k\in[1,K], v_{j}\in V\setminus C_{1}}$. The difference with the previous case lies in the fact that, since $u_{1}\in C_{1}$, we lose the independence of $(k_{i})_{i\in C_{1}}$ with 
$(x=\1_{k_{1}=k_{2}},(X_{ij})_{u_{i},v_{j}\in E_{\alpha}\setminus E_{1}})$. Instead, we will show the independence of the partition induced by the $k_{i}$'s on $C_{1}\cap U$ with 
$(x=\1_{k_{1}=k_{2}},(X_{ij})_{u_{i},v_{j}\in E_{\alpha}\setminus E_{1}})$. We denote $\hat{G}$ this partition. Two nodes $u_{i}$ and $u_{i'}$  of $C_{1}\cap U$ are in the same group of $\hat{G}$ if and only if $k_{i}=k_{i'}$. 

Let $(u_{i},v_{j})\in E_{1}$. We denote $A\in \hat{G}$ the group of $\hat{G}$ containing $u_{i}$. Then, $X_{ij}=\frac{1}{|A|}\sum_{i'\in A }\mu_{k_{i'},j}$. So, the family $(X_{ij})_{u_{i},v_{j}\in E_{1}}$ is entirely defined by $\hat{G}$ and by the centers of this partition for coordinates $j\in[1,p]\cap C_{1}$. For $A\in \hat{G}$, we denote $\mu_{A,j}=\frac{1}{|A|}\sum_{u_{i}\in A}\mu_{k_{i},j}$, which is the $j$-th coordinate of the center of the group $A$.

The family $(x,(X_{ij})_{u_{i},v_{j}\in E_{\alpha}\setminus E_{1}})$ is measurable with respect to the family
$$\mathcal{X}_{1}:=\pa{(k_{i})_{u_{i}\in U\setminus C_{1}}, (\mu_{k,j})_{k,u_{j}\in [1,K]\times (V\setminus C_{1})}, \1_{k_{1}=k_{2}}}\enspace.$$
Since the intersection of $U\cap C_{1}$ with $(U\setminus C_{1})\cup \{u_{1},u_{2}\}$ contains only  $u_{1}$, it is clear that the family $\pa{\1_{k_{i}=k_{i'}}}_{(u_{i},u_{i'})\in (U\cap C_{1})^{2}}$ is independent from $\mathcal{X}_{1}$. Then, $\hat{G}$ is independent from $\mathcal{X}_{1}$.

We now condition to the $k_{i}$'s for $i\in[1,n]$, and to $\mathcal{X}_{1}$. Then, $\hat{G}$ is fixed. Since the $\mu_{k,j}$'s are drawn independently, we deduce that, with our conditioning, the $\mu_{k,j}$'s, for $j\in V\cap C_{1}$, are still drawn independently and uniformly on $\{-\eps,\eps\}$. For $A\in \hat{G}$, there exists $k_{A}\in[1,K]$ which satisfies; for all $j\in V\cap C_{1}$ $\mu_{A,j}=\mu_{k_{A},j}$. The application $A\to k_{A}$ being an injection, we deduce that the $\mu_{k_{A},j}$'s, for $A\in\hat{G}$ and $j\in V\cap C_{1}$ are independent and uniformly drawn from $\{-\eps,+\eps\}$.

Let us summarize this; the partition $\hat{G}$ is independent from $\mathcal{X}_{1}$, and conditionally on $\hat{G}$ and $\mathcal{X}_{1}$, the $\mu_{A,j}$'s, for $A\in\hat{G}$ and $v_{j}\in V\cap C_{1}$, are independently and uniformly drawn from $\{-\eps,+\eps\}$. Together with the fact that the family 
$(X_{ij})_{u_{i},v_{j}\in E_{1}}$ is measurable with respect to $\hat{G}$ and the $\mu_{A,j}$'s, for $A\in\hat{G}$ and $v_{j}\in V\cap C_{1}$, this leads to the independence of $(X_{ij})_{u_{i},v_{j}\in E_{1}}$ with $\mathcal{X}_{1}$.

Since $(x, \pa{X_{ij}}_{u_{i},v_{j}\in E_{\alpha}\setminus E_{1}})$ is measurable with respect to $\mathcal{X}_{1}$, we deduce the independence of the families $(X_{ij})_{u_{i},v_{j}\in E_{1}}$ and $(x,(X_{ij})_{u_{i},v_{j}\in E_{\alpha}\setminus E_{1}})$. Hence, Lemma \ref{lem:nullcumulant} leads to the nullity of the cumulant $\kappa_{\alpha}$. This concludes our proof.

\subsection{Proof of Lemma \ref{lem:boundcumulants}}\label{prf:boundcumulants}

First, we provide an upper-bound on the moments of the form $\E[X^{\gamma}]$ and $\E[xX^{\gamma}]$.

\begin{lem}\label{lem:moments}
    Let $\gamma\in \N^{n\times p}$. We denote by $CC_{\gamma}$ the number of connected components of $\mathcal{G}_{\gamma}$ and by $l_{\gamma}$ the number of nodes of $U\cup V$ spanned by $\mathcal{G}_{\gamma}$. Then, we both have $$\E[X^{\gamma}]\leq \eps^{|\gamma|}\min\pa{1,|\gamma|^{|\gamma|}\pa{\frac{1}{K}}^{l_{\gamma}-\frac{|\gamma|}{2}-CC_{\gamma}}}\enspace,$$
    and $$\E[xX^{\gamma}]\leq \eps^{|\gamma|}\min\pa{\frac{1}{K},|\gamma|^{|\gamma|}\pa{\frac{1}{K}}^{l_{\gamma}-\frac{|\gamma|}{2}-CC_{\gamma}}}\enspace.$$
\end{lem}
This lemma, which is the core of our arguments, is shown in the next subsection. Here, we deduce the upper-bound (\ref{eq:borne:cumulants}) on the cumulant from this lemma. 

We proceed by induction on $\alpha\in\N^{n\times p}$. For the initialization , we have $\kappa_{0}=\E[x]=\frac{1}{K}$.
Then, we take $\alpha\in\N^{n\times p}$, and we suppose that, for all $\beta\lneq \alpha$, we have $$|\kappa_{\beta}|\leq \eps^{|\beta|}(1+|\beta|)^{|\beta|}\min\pa{\frac{1}{K},|\beta|^{|\beta|}\pa{\frac{1}{K}}^{l_{\beta}-\frac{|\beta|}{2}-1}}\enspace.$$
From Lemma \ref{lem:reductioncumulant}, we can suppose that $\mathcal{G}_{\alpha}$ only has one connected component,  otherwise $\kappa_\alpha=0$. We recall that, given $\gamma\in\N^{n\times p}$, $l_{\gamma}$ and $CC_{\gamma}$ respectively stand  for the number of nodes spanned by $\mathcal{G}_{\gamma}$, and the number of connected components of $\mathcal{G}_{\gamma}$.  The Definition (\ref{eq:def:kappa}) of $\kappa_{\alpha}$, Lemma \ref{lem:moments}, and the induction hypothesis imply that 
\begin{align}\nonumber
    |\kappa_{\alpha}|\leq &\E[xX^{\alpha}]+\sum_{0<\beta\lneq \alpha}\binom{\alpha}{\beta}\E[X^{\alpha-\beta}]|\kappa_{\beta}|+|\kappa_{0}|\E[X^{\alpha}]\\ 
    \nonumber
    \leq&\eps^{|\alpha|}|\alpha|^{\alpha}\pa{\frac{1}{K}}^{l_{\alpha}-\frac{|\alpha|}{2}-1}\\ \nonumber
    &+\sum_{0<\beta\lneq \alpha}\binom{\alpha}{\beta}|\eps|^{|\beta|+|\alpha-\beta|}(1+|\beta|)^{|\beta|}|\beta|^{|\beta|}|\alpha-\beta|^{|\alpha-\beta|}\pa{\frac{1}{K}}^{l_{\beta}+l_{\alpha-\beta}-\frac{|\beta|}{2}-\frac{|\alpha-\beta|}{2}-1-CC_{\alpha-\beta}}\\ \nonumber
    &+ \frac{1}{K}\eps^{|\alpha|}|\alpha|^{\alpha}\pa{\frac{1}{K}}^{l_{\alpha}-\frac{|\alpha|}{2}-1}\\
    \leq&2\eps^{|\alpha|}|\alpha|^{\alpha}\pa{\frac{1}{K}}^{l_{\alpha}-\frac{|\alpha|}{2}-1} \nonumber \\ \label{eq:upper:kappa}
    &+\eps^{|\alpha|}|\alpha|^{\alpha}\pa{\frac{1}{K}}^{-\frac{|\alpha|}{2}}\sum_{0<\beta\lneq \alpha}\binom{\alpha}{\beta}(1+|\beta|)^{|\beta|}\pa{\frac{1}{K}}^{l_{\beta}+l_{\alpha-\beta}-1-CC_{\alpha-\beta}}\enspace.
\end{align}
\textbf{Claim}: For any $0<\beta\lneq\alpha$, we have  $l_{\beta}+l_{\alpha-\beta}-C_{\alpha-\beta}\geq l_{\alpha}$. 

We first show this claim.
We have supposed that $\mathcal{G}_{\alpha}$ only has one connected component. We denote $C_{1},\ldots, C_{CC_{\alpha-\beta}}\subset U\cup V$ the connected components of $\mathcal{G}_{\alpha-\beta}$. For all $s\in [1,CC_{\alpha-\beta}]$, there exists $x\in C_{s}$ which is spanned by $\mathcal{G_{\beta}}$. Indeed, otherwise, since the set of edges of $\mathcal{G}_{\alpha}$ is the union of the edges of $\mathcal{G}_{\beta}$ and $\mathcal{G_{\alpha-\beta}}$, $C_{s}$ would also be a connected component of $\mathcal{G}_{\alpha}$ which does not span the nodes spanned by $\mathcal{G}_{\beta}$. This contradicts the connectivity of $\mathcal{G}_{\alpha}$. So, there exist at least $CC_{\alpha-\beta}$ distinct points of $U\cup V$ which are spanned both by $\mathcal{G}_{\alpha-\beta}$ and $\mathcal{G}_{\beta}$. Since the nodes spanned by $\mathcal{G}_{\alpha}$ are spanned by $\mathcal{G}_{\alpha-\beta}$ or $\mathcal{G}_{\beta}$, this leads to $l_{\beta}+l_{\alpha-\beta}-CC_{\alpha-\beta}\geq l_{\alpha}$. 

\medskip

Now that we proved the claim, we plug it in~\eqref{eq:upper:kappa}. This leads to 
 \begin{align*}
    |\kappa_{\alpha}|\leq&\eps^{|\alpha|}|\alpha|^{|\alpha|}\pa{\frac{1}{K}}^{l_{\alpha}-\frac{|\alpha|}{2}-1}\pa{2+\sum_{0<\beta\lneq \alpha}\binom{\alpha}{\beta}(1+|\beta|)^{|\beta|}}\\
    \leq& \eps^{|\alpha|}|\alpha|^{|\alpha|}\pa{\frac{1}{K}}^{l_{\alpha}-\frac{|\alpha|}{2}-1}\pa{2+\sum_{0<\beta\lneq \alpha}\binom{\alpha}{\beta}|\alpha|^{|\beta|}}\\
    \leq& \eps^{|\alpha|}|\alpha|^{|\alpha|}\pa{\frac{1}{K}}^{l_{\alpha}-\frac{|\alpha|}{2}-1}\pa{2+\sum_{w=1}^{|\alpha|-1}\binom{|\alpha|}{w}|\alpha|^{w}}\\
    \leq& \eps^{|\alpha|}|\alpha|^{|\alpha|}(1+|\alpha|)^{|\alpha|}\pa{\frac{1}{K}}^{l_{\alpha}-\frac{|\alpha|}{2}-1}\enspace.
\end{align*}
This inequality proves the first part of the sought upper-bound of Lemma \ref{lem:boundcumulants}.

\medskip
It remains to prove that $|\kappa_{\alpha}|\leq \eps^{|\alpha|}(1+|\alpha|)^{|\alpha|}\frac{1}{K}$. For all $\beta\lneq \alpha$, we know from Lemma \ref{lem:moments} that $\E[X^{\beta}]\leq \eps^{|\beta|}$ and $\E[xX^{\alpha}]\leq \frac{1}{K}\eps^{|\alpha|}$. Together with the induction hypothesis, this leads to \begin{align*}
    |\kappa_{\alpha}|\leq& \eps^{|\alpha|}\frac{1}{K}+\sum_{\beta\lneq \alpha}\binom{\alpha}{\beta}(1+|\beta|)^{|\beta|}\frac{1}{K}\eps^{|\beta|}\eps^{|\alpha-\beta|}\\
    \leq&\eps^{|\alpha|}\frac{1}{K}\pa{1+\sum_{\beta\lneq \alpha}\binom{\alpha}{\beta}|\alpha|^{|\beta|}}\\
    \leq &\eps^{|\alpha|}(1+|\alpha|)^{|\alpha|}\frac{1}{K}\enspace.
\end{align*}
This concludes the induction; for all $\alpha\in\N^{n\times p}$, we have $$|\kappa_{\alpha}|\leq \eps^{|\alpha|}(1+|\alpha|)^{|\alpha|}\min\pa{\frac{1}{K},|\alpha|^{|\alpha|}\pa{\frac{1}{K}}^{l_{\alpha}-\frac{|\alpha|}{2}-1}}\enspace.$$

\subsection{Proof of Lemma \ref{lem:moments}}\label{prf:moments}

Let $\gamma\in\N^{n\times p}$ such that $\mathcal{G}_{\gamma}$ has $CC_{\gamma}$ connected components and spans $l_{\gamma}$ nodes. Let us first upper bound $\E[X^{\gamma}]$. We denote $m_{\gamma}$ the number of nodes of $U$ spanned by $\mathcal{G}_{\gamma}$ and $r_{\gamma}$ the  number of nodes of $V$ spanned by $\mathcal{G}_{\gamma}$. We suppose by symmetry that $\gamma$ is supported on $[1,m_{\gamma}]\times [1,r_{\gamma}]$. We denote $\hat{G}^{\gamma}$ the partition induced on $[1,m_{\gamma}]$ by $k_{1},\ldots ,k_{m_{\gamma}}$.

\begin{lem}\label{lem:esp}
    If, for all $j\in [1,r_{\gamma}]$ and for all groups $A\in \hat{G}^{\gamma}\subset [1,m_{\gamma}]$, we have $\sum_{i\in A}\gamma_{ij}\equiv 0\enspace [2]$, then, $\E[X^{\gamma}|\hat{G}^{\gamma}]=\eps^{|\gamma|}$. Otherwise, we have $\E[X^{\gamma}|\hat{G}^{\gamma}]=0$.
\end{lem}

This lemma, proved in Section \ref{prf:esp}, implies
\begin{equation}\label{eq:esper}
    \E[X^{\gamma}]=\eps^{|\gamma|}\P\cro{\forall j\in [1,r_{\gamma}],\enspace \forall A\in \hat{G}^{\gamma},\enspace \sum_{i\in A}\gamma_{ij}\equiv 0\enspace [2]}\enspace.
\end{equation}
Let $G$ be a partition of $[1,m_{\gamma}]$ with $|G|\leq K$, and let us upper-bound $\P[\hat{G}^{\gamma}=G]$. We write $G=\ac{G_{1},\ldots, G_{|G|}}$. We take, for $k\in[1,|G|]$, $i_{k}\in G_{k}$. Then, $\hat{G}^{\gamma}=G$ implies that, for all $k\in[1,|G|]$, for all $i\in G_{k}$, the equality $k_{i}=k_{i_{k}}$ holds. And, \begin{align*}
    \P\cro{\forall k\in[1,|G|],\enspace \forall i\in G_{k}\setminus \ac{i_k},\enspace k_{i}=k_{i_{k}}}=& \prod_{k\in[1,|G|]}\P\cro{\forall i\in G_{k}\setminus\ac{i_k}, \enspace k_{i}=k_{i_{k}}}\\
    =&\prod_{k\in [1,|G|]}\pa{\frac{1}{K}}^{|G_{k}|-1}\\
    =&\pa{\frac{1}{K}}^{m_{\gamma}-|G|}\enspace.
\end{align*}
This equality leads to 
$$\P[\hat{G}^{\gamma}=G]\leq \pa{\frac{1}{K}}^{m_{\gamma}-|G|}\enspace.$$
The next lemma upper-bounds the number of groups of a partition $G$ of $[1,m_{\gamma}]$ such that, for all $j\in [1,r_{\gamma}]$ and all groups $A\in G$, we have $\sum_{i\in A}\gamma_{ij}\equiv 0\enspace [2]$. Its proof, given in Section \ref{prf:numbergroups}, relies on delicate combinatorial arguments. 

\begin{lem}\label{lem:numbergroups}
    Let $G$ be a partition of $[1,m_{\gamma}]$ satisfying, for all $j\in [1,r_{\gamma}]$, all groups $A\in G$, the equality $\sum_{i\in A}\gamma_{ij}\equiv 0\enspace [2]$. Then, the number of groups satisfies the following inequality $$|G|\leq \frac{|\gamma|}{2}-r_{\gamma}+CC_{\gamma}\enspace.$$
\end{lem}

Applying this lemma together with the fact that there are at most $m_{\gamma}^{m_{\gamma}}$ partitions of $[1,m_{\gamma}]$ leads to $$\P\cro{\forall j\in [1,r_{\gamma}],\enspace \forall A\in \hat{G}^{\gamma},\enspace \sum_{i\in A}\gamma_{ij}\equiv 0\enspace [2]}\leq m_{\gamma}^{m_{\gamma}}\pa{\frac{1}{K}}^{m_{\gamma}+r_{\gamma}-CC_{\gamma}-\frac{|\gamma|}{2}}\enspace.$$ We plug this inequality in \eqref{eq:esper} and get, using $m_{\gamma}\leq |\gamma|$, $$\E[X^{\gamma}]\leq\eps^{|\gamma|}|\gamma|^{|\gamma|}\pa{\frac{1}{K}}^{m_{\gamma}+r_{\gamma}-CC_{\gamma}-\frac{|\gamma|}{2}}\enspace.$$
Moreover, since $\P\cro{\forall j\in [1,r_{\gamma}],\enspace \forall A\in \hat{G}^{\gamma},\enspace \sum_{i\in A}\gamma_{ij}\equiv 0\enspace [2]}\leq 1$, we get $$\E[X^{\gamma}]\leq\eps^{|\gamma|}\min\pa{1,\enspace|\gamma|^{|\gamma|}\pa{\frac{1}{K}}^{m_{\gamma}+r_{\gamma}-CC_{\gamma}-\frac{|\gamma|}{2}}}\enspace.$$
Now, let us upper bound $\E[xX^{\gamma}]$. Since the random variables $x$  and $|X^{\gamma}|/\epsilon^{|\gamma|}$ belong to $[0,1]$ almost surely, we obtain 
\begin{equation}\label{eq:prb}
 \E[xX^{\gamma}]\leq \min(\epsilon^{|\gamma|}\mathbb{E}[x], \mathbb{E}[X^{\gamma}])\ . 
\end{equation}
Then, since $\mathbb{E}[x]= 1/K$, we can deduce the desired bound from the previous case.

\subsection{Proof of Lemma \ref{lem:esp}}\label{prf:esp}

We suppose first that, for all $j\in[1,r_{\gamma}]$ and  for all groups $A\in \hat{G}^{\gamma}$, we have $\sum_{i\in A}\gamma_{ij}\equiv 0\enspace [2]$. Consider a specific $(i,j)$. If $k_{i}=k$, we have $X_{ij}=\mu_{k,j}$. This implies that  

$$X^{\gamma}=\prod_{k,j\in[1,K]\times [1,r_{\gamma}]}\mu_{k,j}^{\sum_{i, k_{i}=k}\gamma_{i,j}}\enspace.$$

Moreover, by hypothesis, conditionally on $\hat{G}^{\gamma}$ for $k\in[1,K]$ and $j\in[1,r_{\gamma}]$, $\sum_{i, k_{i}=k}\gamma_{i,j}\equiv 0\enspace[2]$ and $|\mu_{k,j}|=\eps$. Hence, we have 
$$X^{\gamma}=|\eps|^{\sum \gamma_{ij}}=|\eps|^{|\gamma|}\enspace.$$
This leads to the sought equality $$\E[X^{\gamma}|\hat{G}^{\gamma}]=|\eps|^{\sum \gamma_{ij}}=|\eps|^{|\gamma|}\enspace.$$

Now, let us suppose that there exists $j\in[1,r_{\gamma}]$, a group $A\in \hat{G}^{\gamma}$, such that $\sum_{i\in A}\gamma_{ij}\equiv 1\enspace [2]$. We have as before

$$X^{\gamma}=\prod_{k,j\in[1,K]\times [1,r_{\gamma}]}\mu_{k,j}^{\sum_{i, k_{i}=k}\gamma_{i,j}}\enspace.$$

By independence of the $\mu_{k,j}$'s, for $k\in[1,K]$ and $j\in[1,p]$, both between themselves and with the $k_{i}$'s, for $i\in[1,n]$, we deduce that, conditionally on the $k_{i}$'s, for such $k_i$'s that induce $\hat{G}^{\gamma}$, that

$$\E[X^{\gamma}|(k_{i})_{i\in[1,n]}]=\prod_{k,j\in[1,K]\times [1,r_{\gamma}]}\E[\mu_{k,j}^{\sum_{i, k_{i}=k}\gamma_{i,j}}|(k_{i})_{i\in[1,n]}]\enspace.$$

Let us denote $k'\in[1,K]$ and $j'\in[1,r_{\gamma}]$ that satisfies; $\sum_{i, k_{i}=k'}\gamma_{i,j'}\equiv 1\enspace [2]$. This, together with the fact that the probability distribution of $\mu_{k',j'}$ is symmetric,  leads to $\E\bigg[\mu_{k',j'}^{\sum_{i, k_{i}=k'}\gamma_{i,j'}}|(k_{i})_{i\in[1,n]}\bigg]=0$. This implies $\E\big[X^{\gamma}|(k_{i})_{i\in[1,n]}\big]=0$. Thus, we get the sought equality $$\E[X^{\gamma}|\hat{G}^{\gamma}]=0\enspace.$$

\subsection{Proof of Lemma \ref{lem:numbergroups}}\label{prf:numbergroups}

In this proof, given a partition $G$ of a subset of $[1,n]$, and given $\gamma\in\N^{n\times p}$, we say $\gamma$ is even with respect to $G$ if the following holds: $\sum_{i\in A}\gamma_{ij}\equiv 0\enspace [2]$ for all $A\in G$ and $j\in [1,p]$. 
We recall that $m_{\gamma}$, $r_{\gamma}$,  and $CC_{\gamma}$ respectively stand for the number of nodes in $U\cap \mathcal{G}^{-}_{\gamma}$, the number of nodes in $V\cap \mathcal{G}^{-}_{\gamma}$, and the number of connected components of the graph $\mathcal{G}^{-}_{\gamma}$.
We prove in this section the following claim, which rephrases Lemma \ref{lem:numbergroups}. For $\gamma\in \N^{n\times p}$ and a partition $G$ of $\{i\in[1,n],\enspace \exists j\in[1,p],\enspace \gamma_{ij}>0\}$, if $G$ is even with respect to $\gamma$, then 
\begin{equation}\label{eq:objective:lemma18}
|G|\leq \frac{|\gamma|}{2}-r_{\gamma}+CC_{\gamma}\enspace .
\end{equation}
 For $\beta\leq \gamma$ a restriction of $\gamma$ to one of the connected component of $\mathcal{G}^{-}_{\gamma}$, and for $j\in[1,p]$, the vector $\beta_{:j}$ is either null or equal to $\gamma_{:j}$. So, if each restriction of $\gamma$ to a connected component of $\mathcal{G}^{-}_{\gamma}$ is even with respect to  a partition $G$,
 then  $\gamma$ is also even with respect to $G$. Hence, it is sufficient to prove this lemma when $\mathcal{G}^{-}_{\gamma}$ is connected.  We proceed with by induction on $r_{\gamma}>0$.

 \medskip

\noindent 
\textbf{Initialization}. If $r_{\gamma}=1$; we suppose by symmetry that $\gamma$ is supported on $[1,m_{\gamma}]\times \{1\}$. Again, we proceed by induction. If $|\gamma|=1$, then there exists no partition that satisfy the conditions of Lemma \ref{lem:numbergroups} and so the proposition is true. We suppose that $|\gamma|>1$ and that the proposition is true for all $\beta \lneq \gamma$. We distinguish two cases. 

First, we consider the case where, for all $i\in [1,m_{\gamma}]$, $\gamma_{i1}\equiv 1\enspace [2]$. Let $G$ a partition of $[1,m_{\gamma}]$ for which $\gamma$ is even. Then, each group of $G$ must have an even number of elements. So, each group of $G$ is of cardinality at least $2$. Hence, there are at most $\frac{m_{\gamma}}{2}$ groups in the partition $G$. Moreover, $|\gamma|\geq m_{\gamma}$. Thus, $|G|\leq \frac{|\gamma|}{2}$. 

Now, let us consider the case where there exists $i_{0}\in [1,m_{\gamma}]$ such that $\gamma_{i_{0}1}\equiv 0\enspace [2]$. We define $\gamma'$ by $\gamma'_{ij}=\gamma_{ij}\1_{i\neq i_{0}}$. Let $G$ a partition of $[1,m_{\gamma}]$ for which $\gamma$ is even. We define $G'$ the partition induced on $[1,m_{\gamma}]\setminus \{i_{0}\}$. Then $|G|\leq |G'|+1$. The fact that $\gamma_{i_{0}1}\equiv 0\enspace [2]$ implies that $\gamma$ is also even with respect to $G'$. Since $\gamma\equiv \gamma'\enspace [2]$, we deduce that $\gamma'$ is even with respect to $G'$. Applying the induction hypothesis on $G'$ leads to $|G'|\leq \frac{|\gamma'|}{2}$. Thus, $|G|\leq \frac{|\gamma'|}{2}+1$. Since $\gamma_{i_{0}1}\geq 2$, we obtain $|G|\leq \frac{|\gamma|}{2}$, which concludes the proof of the initialization. 

\medskip

\noindent 
\textbf{Induction step}. Now, we suppose that the following holds. For all $r'<r_{\gamma}$, all $\beta$ satisfying $r_{\beta}=r'$ and $\mathcal{G}^{-}_{\beta}$ connected, for all partition $G$ of $\{i\in [1,n],\exists j\in [1,p],\enspace \beta_{ij}>0\}$ for which $\beta$ is even, we have the inequality $|G|\leq \frac{|\beta|}{2}-r_{\beta}+1$. Let us prove that, for $G$ a partition of $\{i\in [1,n],\exists j\in [1,p],\enspace \gamma_{ij}>0\}$, if $\gamma$ is even with respect to $G$, the inequality $|G|\leq \frac{|\gamma|}{2}-r_{\gamma}+1$ also holds.

We suppose by symmetry that $\gamma$ is supported on $[1,m_{\gamma}]\times [1,r_{\gamma}]$. We call $\mathcal{V}_{\gamma}$ the graph on $[1,r_{\gamma}]$ where we connect nodes $j$ and $j'$, if they have a common neighbour in $\mathcal{G}^{-}_{\gamma}$, see Figure \ref{fig:graphe1}.  More formally, for all $j,j'\in [1,p]$, there exists an edge  in $\mathcal{V}_{\gamma}$ between $j$ and $j'$,  if and only if, there exists $i\in [1,n]$ such that  $\gamma_{ij}\gamma_{ij'}>0$. 

\fbox{
\begin{minipage}{0.94\textwidth}
{\bf Example.}

Let us illustrate the construction of the graphs $\mathcal{V}_{\gamma}$ and $\mathcal{V}_{\gamma'}$  on an example. Let $\gamma$ be the matrix below, which is even with respect to the partition $\{1,2,3,4\},\{5\}$
$$\gamma=\begin{pmatrix}
    2&1&0&2&0&0&0&0\\
    0&1&2&0&0&1&0&0\\
    0&2&0&0&2&1&2&0\\
    0&0&0&0&0&0&2&0\\
    0&0&0&0&0&0&2&0
\end{pmatrix}.$$
Figure \ref{fig:graphe1} represents respectively $\mathcal{G}^{-}_{\gamma}$ and the corresponding graph $\mathcal{V}_{\gamma}$. We also represent the set $L:=\{i\in [1,m_{\gamma}],\enspace \exists i'\in U_{\gamma'},\enspace \exists A\in G,\enspace \{i,i'\}\subset A\}$. We do not represent here multi-edges.

    \centerline{\includegraphics[height=6cm]{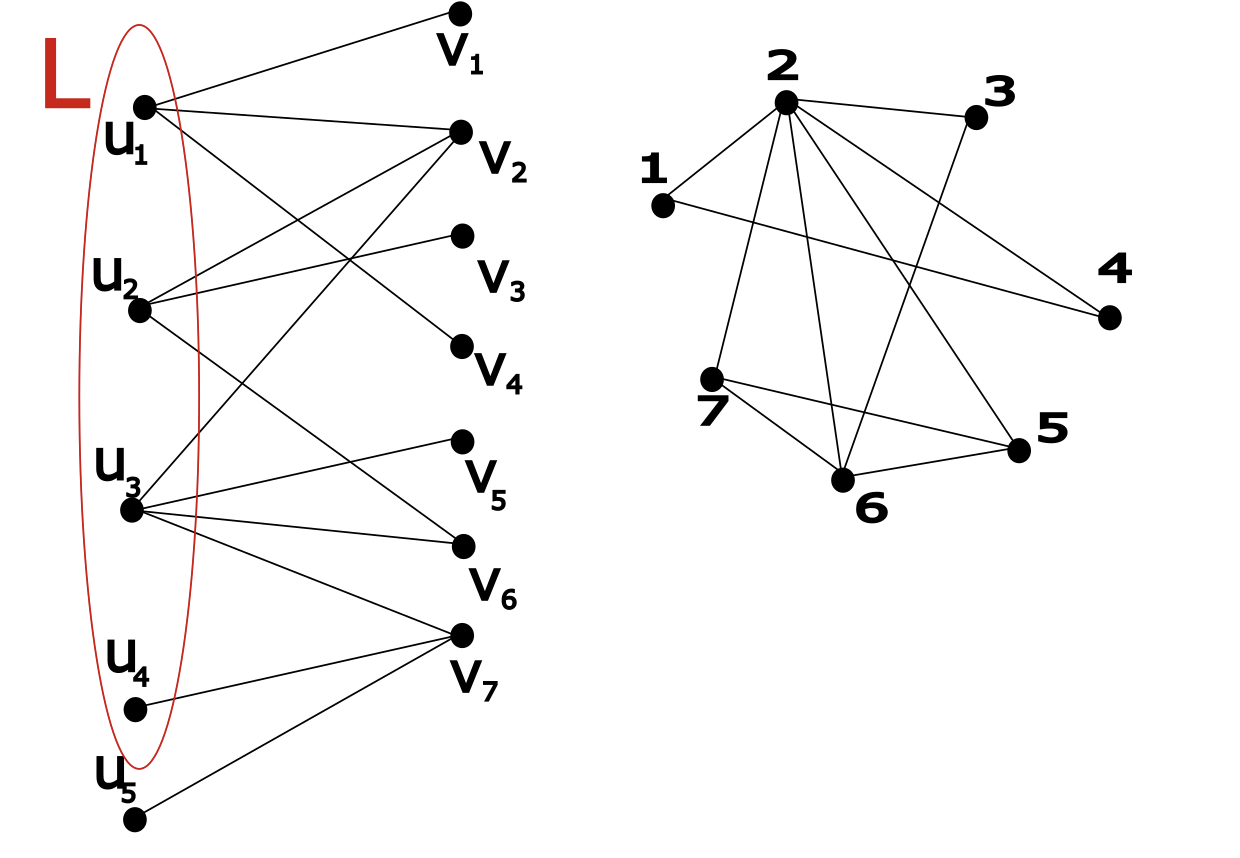}}
    \captionof{figure}{The graph $\mathcal{G}^{-}_{\gamma}$ (on the left), and the corresponding graph $\mathcal{V}_{\gamma}$ (on the right).}\label{fig:graphe1}

\end{minipage}}

For $j,j'\in[1,p]$, the nodes $j$ and $j'$ are in the same connected component for $\mathcal{V}_{\gamma}$, if and only if, $v_{j}$ and $v_{j'}$ are in the same connected component for $\mathcal{G}_{\gamma}$. Hence, the graph $\mathcal{V}_{\gamma}$ is connected. As a consequence, there exists a spanning tree of this connected component. By symmetry, we can suppose that the node $r_{\gamma}$ is a leaf of this tree. This implies that the graph induced by $\mathcal{V}_{\gamma}$ on $[1,r_{\gamma}-1]$ is also connected. We define $\gamma'$ by $\gamma'_{ij}=\gamma_{ij}\1_{j\neq r_{\gamma}}$. We write $\mathcal{V}_{\gamma'}$ for the graph induced by $\mathcal{V}_{\gamma}$ on $[1,r_{\gamma}-{1}]$. The graph $\mathcal{V}_{\gamma'}$ is 
obtained by removing the node $r_\gamma$ and the edges connected to it, as represented in Figure \ref{fig:graphe2}.
This implies that $\mathcal{V}_{\gamma'}$, and therefore also $\mathcal{G}^{-}_{\gamma'}$ are connected graphs. We denote $U_{\gamma'}=\{i\in [1,n], \exists j\in[1,p], \gamma'_{ij}>0\}$. 

\fbox{
\begin{minipage}{0.94\textwidth}
{\bf Example (continued).}
In our example, we have $r_\gamma=7$, and the node 7 is a leaf of a spanning tree of $\mathcal{V}_{\gamma}$. After setting to zero the column $j=7$, we obtain 
$$\gamma'=\begin{pmatrix}
    2&1&0&2&0&0&0&0\\
    0&1&2&0&0&1&0&0\\
    0&2&0&0&2&1&0&0\\
    0&0&0&0&0&0&0&0\\
    0&0&0&0&0&0&0&0
\end{pmatrix}.$$
The corresponding graphs $\mathcal{G}^{-}_{\gamma'}$ and $\mathcal{V}^-_{\gamma'}$, built by removing $v_7$ and the edges connecting it, are represented in Figure \ref{fig:graphe2}.

  \centerline{\includegraphics[height=6cm]{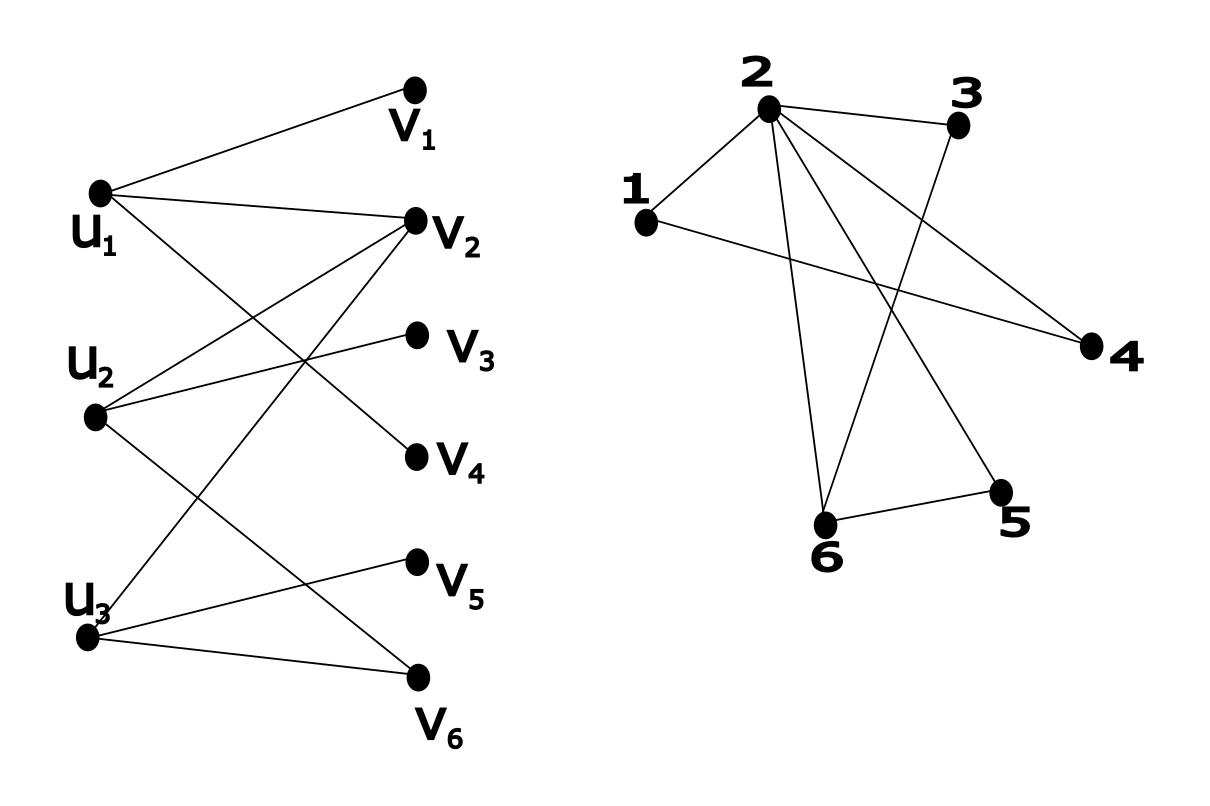}}
    \captionof{figure}{The graph $\mathcal{G}_{\gamma'}$ (on the left), and the corresponding graph $\mathcal{V}_{\gamma'}$ (on the right).}
    \label{fig:graphe2}
\end{minipage}}
\medskip

Let $G$ a partition of $[1,m_{\gamma}]$ such that $\gamma$ is even with respect to $G$. The $r_{\gamma}$-th column of $\gamma'$ being null, we have that, for all $A\in G$, $\sum_{i\in A}\gamma'_{ir_{\gamma}}=0$.

For $j\neq r_{\gamma}$, the $j$-th column of $\gamma$ is equal to that of  $\gamma'$. This implies, since $\gamma$ is even with respect to $G$, that $\sum_{i\in A}\gamma'_{ij}\equiv 0\enspace [2]$ for any $A$ in $G$.

Thus, for all groups $A\in G$, all $j\in [1,r_{\gamma}]$, we have $\sum_{i\in A}\gamma'_{ij}\equiv 0\enspace [2]$. This means that $\gamma'$ is even with respect to the partition $G$. Now, we distinguish in the set $U_{\gamma}$ different type of nodes. First, we have the set $U_{\gamma'}=\{i\in [1,n], \exists j\in[1,r_{\gamma}-1], \gamma'_{ij}>0\}$ of nodes of $U$ spanned by $\mathcal{G}^{-}_{\gamma'}$. Then, we define $L=\{i\in [1,m_{\gamma}],\enspace \exists i'\in U_{\gamma'},\enspace \exists A\in G,\enspace \{i,i'\}\subset A\}$, the set of nodes of $U$ which are linked to $U_{\gamma'}$ through the partition $G$.  Note that $U_{\gamma'}\subset L$. Finally, we consider the remaining nodes $[1,m_{\gamma}]\setminus L$.

We respectively define $G'$ the restriction of $G$ to $U_{\gamma'}$ and $G''$ the restriction of $G$ to $[1,m_{\gamma}]\setminus L$. By definition of $L$, we have $|G|=|G'|+|G''|$. 

Let us first use the induction hypothesis on $|G'|$. To do so, let us prove that $\gamma'$ is even with respect to $G'$. We know that $\gamma'$ is even with respect to $G$. Moreover, every group $A'$ of $G'$ is the restriction of a group $A$ of $G$ to $U_{\gamma'}$. This implies that, for all $j\in [1,r_{\gamma}-1]$, the equality $\sum_{i\in A'}\gamma'_{ij}=\sum_{i\in A}\gamma'_{ij}$ holds. Hence, $\gamma'$ is even with respect to $G'$. This allows us to apply the induction hypothesis and leads to $|G'|\leq \frac{|\gamma'|}{2}-r_{\gamma}+2$.  

Now, we upper bound $|G''|$. For all group $A$ of $G''$, it is clear by definition of $L$ that $A$ is also a group of $G$ which contains only elements from $[1,m_{\gamma}]\setminus L$. Hence, $\sum_{i\in A}\gamma_{ir_{\gamma}}\equiv 0 \enspace [2]$. Thus, we know from the initialization step of the induction that $|G''|\leq \frac{|\gamma''|}{2}$, where we define $\gamma''$ as the restriction of $\gamma$ to $(\cup_{A \in G''}A)\times \{r_{\gamma}\}$. This leads to $|G|= |G'|+|G''|\leq \frac{|\gamma''|}{2}+\frac{|\gamma'|}{2}-r_{\gamma}+2$. It remains to prove that $|\gamma'|+|\gamma''|\leq |\gamma|-2$. It is clear that $|\gamma|=|\gamma'|+|\gamma''|+\sum_{i\in L}\gamma_{ir_{\gamma}}$. We distinguish three cases. 

\medskip

\noindent 
\textbf{First case}: If $|\{i\in[1,m_{\gamma}],\enspace \gamma_{ir_{\gamma}}\equiv 1\enspace [2]\}\cap U_{\gamma'}|=1$. We call $i_{0}$ the only point in this set. We denote $A$ the group of $G$ such that $i_{0}\in A$. By hypothesis, $\sum_{i\in A}\gamma_{ir_{\gamma}}\equiv 0\enspace [2]$. Hence, there exists $i\neq i_{0}\in A$ which satisfies $\gamma_{ir_{\gamma}}\equiv 1\enspace [2]$. Thus, since $A\subset L$, we get $\sum_{i\in L}\gamma_{ir_{\gamma}}\geq 2$. This leads to $|\gamma|\geq |\gamma'|+|\gamma''|+2$. 

\medskip

\noindent 
\textbf{Second case}: If $|\{i\in[1,m_{\gamma}],\enspace \gamma_{ir_{\gamma}}\equiv 1\enspace [2]\}\cap U_{\gamma'}|=0$. Since $\mathcal{G}_{\gamma}$ is connected, there exists $i\in U_{\gamma'}$ such that $\gamma_{ir_{\gamma}}>0$. Since $|\{i\in[1,m_{\gamma}],\enspace \gamma_{ir_{\gamma}}\equiv 1\enspace [2]\}\cap U_{\gamma'}|=0$, we have $\gamma_{ir_{\gamma}}\geq 2.$ The fact that $U_{\gamma'}\subset L$ leads to $\sum_{i\in L}\gamma_{ir_{\gamma}}\geq 2$. Hence, $|\gamma|\geq |\gamma'|+|\gamma''|+2$.

\medskip

\noindent 
\textbf{Third case}: If $|\{i\in[1,m_{\gamma}],\enspace \gamma_{ir_{\gamma}}\equiv 1\enspace [2]\}\cap U_{\gamma'}|\geq 2$. In this case, $\sum_{i\in U_{\gamma'}}\gamma_{ir_{\gamma}}\geq 2$. This leads to $|\gamma|\geq |\gamma'|+|\gamma''|+2$.

\medskip

This concludes our induction and leads to the sought inequality $|G|\leq \frac{|\gamma|}{2}-r_{\gamma}+1$ for any $\gamma$ such that $\mathcal{G}^{-}_{\gamma}$ is connected and for any partition $G$ of $\{i\in[1,n],\enspace \exists j\in[1,p],\enspace \gamma_{ij}>0\}$ that is even with respect to $\gamma$. In the general case where $\mathcal{G}_{\alpha}$ has $CC_{\gamma}$ connected components, we readily get $|G|\leq \frac{|\gamma|}{2}-r_{\gamma}+CC_{\gamma}$.

\subsection{A computational barrier for $p\leq n$}\label{sec:p-smaller-than-n}
In this section, we adapt the proof of Theorem~\ref{thm:lowdegreeclustering} to  provide a computational lower bound when $p\leq n$. This lower-bound provides evidence for the existence of a computation-information gap for ${n\over K}\vee K \ \widetilde{\ll}\  p \leq n$, where $\widetilde{\ll}$ hides some polylog$(n)$ factors.
We believe yet, that our computational lower-bound is not tight in this regime.

\begin{prop}\label{prp:lowdegreeclustering}
    Let $D\in \N$. If $p\leq n$ and $\overline{\zeta}_{n}:=\frac{\bar\Delta^{4}D^{8}(1+D)^{4}n}{p^2}\max\pa{\frac{n}{K^{2}},1}<1$, then under the prior of Definition \ref{def:prior}, we have
    \begin{equation*}
        MMSE_{\leq D}\geq {1\over K} - \frac{1}{K^{2}}\pa{1 +  \frac{\overline{\zeta}_n}{(1-\sqrt{\overline{\zeta}_n})^3}}\enspace .
    \end{equation*}
    In particular, if $\bar \Delta^2 \ll D^{-6}\pa{\frac{pK}{n}\wedge\frac{p}{\sqrt{n}}}$, then $MMSE_{\leq D}=\frac{1}{K}-\frac{1+o(1)}{K^{2}}$.
\end{prop}

The second statement of Proposition \ref{prp:lowdegreeclustering}
states that low-degree polynomials with degree $D\leq (\log(n))^{1+\eta}$ do not perform better than the trivial estimator when 
\[\bar \Delta^2\ \widetilde{\ll}\ \pa{\frac{pK}{n}\wedge\frac{p}{\sqrt{n}}},\]
where $\widetilde{\ll}$ hides polylog$(n)$ factors.
Since lower-bounds for low-degree polynomials with degree $D\leq (\log(n))^{1+\eta}$ are considered as evidence of the computational hardness of the problem, this suggests computational hardness of estimating $M^*$ when  $\bar \Delta^2 \ \widetilde{\ll}\ \pa{\frac{pK}{n}\wedge\frac{p}{\sqrt{n}}}$ and $p\leq n$. Since, as made explicit in Section \ref{sec:low-degree}, estimation of $M^*$ is possible in polynomial time  when clustering is possible in polynomial time,  this provides compelling evidence for the computational hardness of the clustering problem in this regime. 

Comparing the computational lower bound $\bar \Delta^2\ \widetilde{\gg}\ \pa{\frac{pK}{n}\wedge\frac{p}{\sqrt{n}}}$ for $p\leq n$, to the information barrier 
\[\Delta^2\gtrsim {\log(K)\vee \sqrt{\frac{pK\log(K)}{n}}},\]
we observe that there is a computation-information gap when 
$${1\vee \sqrt{\frac{pK}{n}}}\ \widetilde{\ll}\ \pa{\frac{pK}{n}\wedge\frac{p}{\sqrt{n}}},$$
which happens when
$${n\over K}\vee K\ \widetilde{\ll}\ p\leq n\,.$$
\medskip

\begin{proof}[Proof of Proposition~\ref{prp:lowdegreeclustering}.] \ 
We argue exactly as  in the proof of Theorem~\ref{thm:lowdegreeclustering}
by upper bounding $corr_{D}^{2}$. In the proof of Theorem~\ref{thm:lowdegreeclustering}, we only used the assumption that $n\leq p$ in~\eqref{eq:upper_corrD}. Hence, we start from~\eqref{eq:upper_corr_D:first} by plugging the definition of 
$\overline{\zeta}_n=\bar\Delta^4 D^{4}(1+D)^{2}\frac{n}{p^2}\max\pa{\frac{n}{K^{2}},1}$. This leads us to 
\begin{align}    \nonumber
\lefteqn{corr_{D}^{2}-\frac{1}{K^2}}  \\ \nonumber
    \leq & \frac{1}{K^2}\sum_{d=1}^{D}\sum_{(r,m)\in\mathcal{D}_d}\overline{\zeta}_n^{d/2}p^{r}n^{m-2-d/2}\pa{\frac{1}{\max(1,n/K^2)}}^{d/2}\min\pa{1,\pa{\frac{1}{K}}^{2m+2r-d-4}} \nonumber \\
        \leq & \frac{1}{K^2}\sum_{d=1}^{D}\sum_{(r,m)\in\mathcal{D}_d}\overline{\zeta}_n^{d/2}n^{m+r-2-d/2}\pa{\frac{1}{\max(1,n/K^2)}}^{d/2}\min\pa{1,\pa{\frac{1}{K}}^{2m+2r-d-4}}  , 
    \label{eq:upper_corrD_low_dim}
\end{align}
 where we used in the last line that $n\geq p$. Note that this upper bound is exactly the same as in~\eqref{eq:upper_corrD} except that $\zeta_n$ has been replaced by $\overline{\zeta}_n$. Hence, arguing as in  the proof of Theorem~\ref{thm:lowdegreeclustering},  we arrive at the similar conclusion 
\begin{align} \nonumber
    corr_{D}^{2}&\leq \frac{1}{K^2}\left[1 +  \frac{\overline{\zeta}_n}{(1-\sqrt{\overline{\zeta}_n})^3} \right] . 
\end{align}
\end{proof}

\section{Proof of Theorem \ref{thm:error_K_means:simple}}\label{prf:thm_error_K_means:simple}

Without loss of generality, we assume throughout this proof that $\sigma=1$. We will use the following notation. For $i\in[1,n]$, we decompose $Y_{i}=\E(Y_{i})+E_{i}=\mu_{k}+E_{i}$, if $i\in G^{*}_{k}$. Then, $(E_{i})_{i\in[1,n]}$ are independent vectors, with distribution $\mathcal{N}(0,I_{p})$.
We denote:
\begin{itemize}
    \item $Y\in\R^{n\times p}$ whose $i$-th row is the vector $Y_{i}$,
    \item $A\in\R^{n\times K}$ the membership matrix defined by $A_{ik}=\1_{i\in G^{*}_{k}}$,
    \item $E\in\R^{n\times p}$ the noise matrix whose $i$-th row is the Gaussian vector $E_{i}$,
    \item $\mu \in\R^{K\times p}$ whose $k$-th row is $\mu_{k}$.
\end{itemize}
We then have the relation $$Y=A\mu +E.$$
Let us also denote $m=\min_{k\in[1,K]}$ the minimal size of the clusters, and $m^{+}=\max_{k\in[1,K]}$ the maximal size. The hypothesis $G^*\in\mathcal{P}_{\alpha}$ is equivalent to $\frac{m^+}{m}\leq \alpha$. We define the signal-to-noise ratio $\tilde{s}^{2}=\Delta^{2}\wedge\frac{\Delta^{4}m}{p}$. We note that $\frac{1}{\alpha} s^{2}\leq \tilde{s}^{2}\leq s^{2}$, where $s^2$ is the signal-to-noise ratio defined in \eqref{eq:exponentialdecrease}.

We prove in this section a more general theorem, which induces directly Theorem \ref{thm:error_K_means:simple}. 

\begin{thm}\label{thm:error_K_means}
    There exist positive numerical constants $c$, $c'$, $c''$ such that the following holds. If 
    \begin{equation}\label{eq:condition_snr}
    \tilde{s}^{2}\geq c (\log(n/m)\vee \frac{m^{+}}{m})\enspace , 
    \end{equation}
    then, we have $err(\hat{G},G^{*})\leq e^{-c''\tilde{s}^{2}}$ with probability at least $1-c'/n^2$. 
\end{thm}

First, let us formulate the $K$-means criterion in an alternative way. Given $G=\ac{G_{1},...,G_{K}}$ a partition, let us define the normalized-partnership matrix $B(G)$ by:

$$B_{ij}=\sum_{k\leq K}\1_{i,j\in G_{k}}\frac{1}{|G_{k}|}\enspace.$$
The application $G\to B(G)$ is a bijection on 
$$\mathcal{B}=\{ B\in S_{n}(\mathbb{R})^{+}:\ B_{ij}\geq 0, \tr(B)=K, B1=1, B^{2}=B \}\enspace.$$ 
We refer to \cite{PengWei07} and to Chapter 12.4 of \cite{HDS2} for this last statement. It implies the following proposition.

\begin{prop}[\cite{PengWei07}]\label{prop:equivalent_K_means}
    Finding $\widehat{G}\in\argmin_{G} \mathrm{Crit}(G)$ is equivalent to finding $$\widehat{B}\in\argmax_{B\in \mathcal{B}} \<YY^{T}, B\>\enspace .$$ 
\end{prop}

We will denote in the following $\hat{B}$ a minimiser of $\<YY^{T}, B\>$ over the set $\mathcal{B}$. We note that the convex relaxation of this problem introduced in \cite{PengWei07} and studied in \cite{giraud2019partial} is the minimiser of $\<YY^{T}, B\>$ over the convex set 

$$\mathcal{C}=\ac{B\in S_{n}(\mathbb{R})^{+}: \ B_{ij}\geq 0,\  \tr(B)=K,\ B1=1}\enspace .$$
Since we have $\mathcal{B}\subset \mathcal{C}$, all the bounds obtained in \cite{giraud2019partial} for the matrices in $\mathcal{C}$ are {\it de facto} valid for $\mathcal{B}$.

The proof of Theorem \ref{thm:error_K_means} follows the same main steps as the analysis of relaxed $K$-means  in \cite{giraud2019partial}. The main difference is in the delicate proof of Lemma~\ref{lem:noise} below.
First, we use that, for any partition $G$, the proportion $err(G,G^*)$ of misclustered points is controlled as a function of the $l_1$ norm $\|B^{*}-B^{*}B\|_{1}$, where $B$ is the normalized-partnership matrix associated to $G$. More precisely, we show, in Section \ref{prooflem:error}, the following Lemma.

\begin{lem}\label{lem:erreur}
Consider two partitions $G^*$ and $G$. Write $B^*$ and $B$ for the corresponding normalized-partnership matrices. For some numerical constant $c>0$, we have 
 $$err(G,G^{*})\leq c \left(\frac{m^{+}}{m}\right)\frac{\|B^{*}-B^{*}B\|_{1}}{n}\enspace .$$
\end{lem}

As a consequence, we only have to control the $l_1$ error $\|B^{*}-B^{*}\hat{B}\|_{1}$. Again, as in~\cite{giraud2019partial}, we start from the optimality condition that defines $\hat{B}$, that is 
\[
\<YY^{T}, \hat{B}-B^{*}\>\geq 0\enspace . 
\]
By definition, we have $Y= A\mu+E$. Plugging this expression in the above inequality leads to 
\begin{equation}\label{eq:optimalite}
 \<A\mu(A\mu)^{T},B^{*}-\hat{B}\>\leq \<A\mu E^{T}+E(A\mu)^{T},\hat{B}-B^{*}\>+\<EE^{T},\hat{B}-B^{*}\> \enspace .
\end{equation}
We call $\<A\mu(A\mu)^{T},\hat{B}-B^{*}\>$ the signal term, $\<EE^{T},\hat{B}-B^{*}\>$ the quadratic noise term, and $\<A\mu E^{T}+E(A\mu)^{T},\hat{B}-B^{*}\>$ the cross term. The three following lemmas control each of these three terms. 
For any $B\in\mathcal{B}$, we denote $\delta_{B}=\|B^{*}-B^{*}B\|_{1}$.

\begin{lem}\label{lem:signal}
    For all $B\in\mathcal{B}$, we have 
    \[
    \<A\mu(A\mu)^{T},B^{*}-B\>=\<S,B^{*}-B\>\geq \frac{1}{2}\Delta^{2}\delta_{B}\enspace ,
    \]
    where $S_{ab}=-\frac{1}{2}\|\mu_{a}-\mu_{b}\|^{2}$ for $a$ and $b$ in $[1,n]$.
\end{lem}

\begin{lem}\label{lem:noise}
    There exists numerical constants $c_{1}$ and $c_{2}$ such that the following holds with probability at least $1-\frac{c_{1}}{n^{2}}$. Simultaneously for all $B\in\mathcal{B}$, we have:
    \begin{align*}
        \<EE^{T},B-B^{*}\>\leq& c_2\delta_{B}\left(\log(n/m)\vee \frac{m^{+}}{m}+\sqrt{\frac{p}{m}\pa{\log(n/m)\vee \frac{m^{+}}{m}}}\right)\\
        &+c_2\delta_{B}\left(\sqrt{\frac{p}{m}\log(n\frac{K^{3}}{\delta})}+\log(n\frac{K^{3}}{\delta})\right)\enspace.
    \end{align*}
\end{lem}

\begin{lem}\label{lem:cross}
    There exist constants $c_{3}$ and $c_{4}$ such that the following holds with probability at least $1-\frac{c_{3}}{n^{2}}$. Simultaneously for all $B\in\mathcal{P}$, we have 
    $$\langle A\mu E^T +E(A\mu)^T  ,B-B^{*}\rangle \leq c_{4} \sqrt{\langle S,B^{*}-B\rangle}\sqrt{\delta_{B} \log(\frac{nK^{3}}{\delta})}\enspace , $$ where $S_{ab}=-\frac{1}{2}\|\mu_{a}-\mu_{b}\|^{2}$.
\end{lem}
\medskip

Lemmas~\ref{lem:signal} and~\ref{lem:cross}, taken from~\cite{giraud2019partial}, are in fact valid for the larger class of matrices $B\in \mathcal{C}$. The difference with \cite{giraud2019partial} lies in the quadratic noise term controlled by Lemma \ref{lem:noise}. For this quadratic term, we can get an upper-bound over the class $\mathcal{B}$, that is significantly smaller than over the class $\mathcal{C}$.
Indeed, here, we can leverage the fact that the class $\mathcal{B}$ of matrices  is finite. Instead of having an upper-bound  of the order of $n/m$, we get a smaller upper-bound of the order of $\log(n/m)\vee \frac{m^{+}}{m}$. 
This is the main reason for which the  Condition~\eqref{eq:condition_snr} for exact Kmeans is less stringent than the condition $\tilde s \geq n/m$ for the SDP relaxation of Kmeans. We refer to Section \ref{sec:noise} for a proof of Lemma \ref{lem:noise}, which is the main hurdle for proving Theorem \ref{thm:error_K_means}. We note that when the constant c of \eqref{eq:condition_snr} is large enough, the term $c_{2}\delta_{B}\pa{\log(n/m)\vee \frac{m^{+}}{m}+\sqrt{\frac{p}{m}\pa{\log(n/m)\vee \frac{m^{+}}{m}}}}$ is bounded by $\frac{\Delta^{2}\delta_{B}}{8}$. That allows us to neglect it in (\ref{eq:optimalite}) up to a multiplicative constant.

Combining~\eqref{eq:optimalite} with Lemmas~\ref{lem:signal}--\ref{lem:cross}, we deduce that, for some constants $c'$ and $c''$ the following holds with probability at least $1-\frac{c'}{n^{2}}$
$$\<S,B^{*}-\hat{B}\>\leq c''\left[ \sqrt{\langle S,B^{*}-\hat{B}\rangle}\sqrt{\delta_{\hat{B}} \log(\frac{nK^{3}}{\delta_{\hat{B}}})}+\delta_{\hat{B}}\pa{\sqrt{\frac{p}{m}\log(n\frac{K^{3}}{\delta_{\hat{B}}})}+\log(n\frac{K^{3}}{\delta_{\hat{B}}})}\right]\ . 
$$
Below, we write for convenience $a\lesssim b$ for $a\leq c\, b$, with $c$ a positive numerical constant that may vary from line to line. 
The above bound implies
$$\<S,B^{*}-\hat{B}\>\lesssim \delta_{\hat{B}}\log{(\frac{nK^{3}}{\delta_{\hat{B}}})}\vee \delta_{\hat{B}} \sqrt{\frac{p}{m}\log{(\frac{nK^{3}}{\delta_{\hat{B}}})}}\enspace .$$
Moreover, by Lemma~\ref{lem:signal}, we have  $\frac{1}{2}\Delta^{2}\delta_{\hat{B}}\leq \<S,B^{*}-\hat{B}\>$. This leads us to $$\Delta^{2}\delta_{\hat{B}}\lesssim \delta_{\hat{B}}\log{(\frac{nK^{3}}{\delta_{\hat{B}}})}\vee \delta_{\hat{B}} \sqrt{\frac{p}{m}\log{(\frac{nK^{3}}{\delta_{\hat{B}}})}} \enspace, $$
which, in turn, implies that
$$\log(\frac{nK^{3}}{\delta_{\hat{B}}})\gtrsim \Delta^{2}\wedge \frac{\Delta^{4}m}{p}=\tilde{s}^{2}\enspace . $$
Thus, for some numerical constant $c_0$, we have $$\delta_{\hat{B}}\leq nK^{3}e^{-c_0 \tilde{s}^{2}}\enspace .$$
Coming back to Lemma \ref{lem:erreur}, we conclude that the proportion of misclassified points satisfies $err(\hat{G}, G^*)\leq (\frac{m^{+}}{m})K^{3}e^{-c_0 \tilde{s}^{2}}\leq (\frac{n}{m})^{4}e^{-c_0 \tilde{s}^{2}}$. Therefore, provided that the constant $c$ in Condition~\eqref{eq:condition_snr} is large enough, there exists a constant $c''$ such that $$\mathrm{err}(\hat{G},G^{*})\leq e^{-c'' \tilde{s}^{2}}\enspace .$$ This concludes the proof.

\subsection{The quadratic noise term (Proof of Lemma~\ref{lem:noise})}\label{sec:noise}
In this section, we will upper-bound uniformly over all $\mathcal{B}$ the noise term $\<EE^{T},\hat{B}-B^{*}\>$.

Let us decompose, for $B\in\mathcal{B}$, the term $B-B^{*}$. We get: $B-B^{*}=B^{*}(B-B^{*})+(B-B^{*})B^{*}+(I_{n}-B^{*})(B-B^{*})(I_{n}-B^{*})-B^{*}(B-B^{*})B^{*}$. We also remark that, since for all $B\in\mathcal{B}$, $\tr(B)=K$, we have $\<I_{n},B-B^{*}\>=0$. We thus get \begin{align*}
    \<EE^{T},B-B^{*}\>=&\,2\<EE^{T}-pI_{n},B^{*}(B-B^{*})\>\\
    &+\<EE^{T}-pI_{n}, (I_{n}-B^{*})(B-B^{*})(I_{n}-B^{*})\>\\
    &-\<EE^{T}-pI_{n}, B^{*}(B-B^{*})B^{*}\>.
\end{align*} We will control each of these terms.

First, let us deal with the term in $(I_{n}-B^{*})(B-B^{*})(I_{n}-B^{*})$. We show, using Hanson-Wright lemma and union bounds over the sets $$\mathcal{B}_{j}:=\{B\in\mathcal{B},\delta_{B}\in[j-1,j]\},\enspace j\leq 2n\enspace,$$ the following lemma (see section \ref{prooflem:quad1} for the proof of this lemma).

\begin{lem}\label{lem:quad1}
    There exists $c_{5}$, $c_{6}$ two positive numerical constants such that the following holds with probability at least $1-\frac{c_{5}}{n^{2}}$. Simultaneously for all $B\in\mathcal{B}$, we have $$\langle EE^{T}-pI_{n},(I-B^{*})(B-B^{*})(I-B^{*})\rangle\leq c_{6}\delta_{B}\pa{\log(n/m)\vee \frac{m^{+}}{m}+\sqrt{\frac{p}{m}\pa{\log(n/m)\vee \frac{m^{+}}{m}}}}.$$
\end{lem}

\medskip

Now we will upper bound the other quadratic terms (the ones corresponding to $B^{*}(B-B^{*})$ and to $B^{*}(B-B^{*})B^{*}$). To do so, we will distinguish two cases:\begin{enumerate}
    \item We will use the bound obtained in \cite{giraud2019partial} for the matrices $B$ such that $\delta_{B}\leq m$;
    \item We will then show a similar result than Lemma \ref{lem:quad1} for the matrices $B$ that fulfill $\delta_{B}\geq m$.
\end{enumerate}
For the first point, we will use Lemma 7 of \cite{giraud2019partial} which states:

\begin{lem}
    There exist positive numerical constants $c_{7}$ and $c_{8}$ such that the following holds with probability at least $1-\frac{c_{7}}{n^{2}}$. Simultaneously for all $B\in\mathcal{B}$, we have $$\langle EE^{T}-pI_{n},B^{*}(B-B^{*})\rangle \leq \frac{c_{8}\delta_{B}}{\sqrt{m}}\pa{\sqrt{p\log(n\frac{K^{3}}{\delta_{B}})}+\sqrt{\delta_{B}+1}\log(n\frac{K^{3}}{\delta_{B}})}\enspace.$$ The same result holds for the term in $B^{*}(B-B^{*})B^{*}$.
\end{lem}
This lemma implies that, with probability at least $1-\frac{c_{7}}{n^{2}}$, for all matrices $B$ fulfilling $\delta_{B}\leq m$, $$\langle EE^{T}-pI_{n},2B^{*}(B-B^{*})-B^{*}(B-B^{*})B^{*}\rangle \leq 6c_{8}\delta\pa{\sqrt{\frac{p}{m}\log(n\frac{K^{3}}{\delta})}+\log(n\frac{K^{3}}{\delta})}.$$

For the second point, we will do as for Lemma \ref{lem:quad1} and use Hanson-Wright Lemma and union bounds over the sets $\mathcal{B}_{j}$ to get the following Lemma (see Section \ref{proofquad2} for a proof of this lemma).
\begin{lem}\label{quad2}
    There exists $c_{9}$ and $c_{10}$ two positive numerical constants such that the following holds. With probability higher than $1-\frac{c_{9}}{n^{2}}$, we have simultaneously for all $B\in\mathcal{B}$ such that $\|B^{*}-B^{*}B\|_{1}\geq m$, $$\langle EE^{T}-pI_{n},B^{*}(B-B^{*})\rangle \leq c_{10}\delta_{B}\left(\log(n/m)\vee \frac{m^{+}}{m}+\sqrt{\frac{p}{m}\pa{\log(n/m)\vee \frac{m^{+}}{m}}}\right).$$ The same result holds for the term in $- B^{*}(B-B^{*})B^{*}$.
\end{lem}
Combining the results from Lemma \ref{lem:quad1}-\ref{quad2}, we get the bound of Lemma \ref{lem:noise}.

\subsubsection{Proof of Lemma \ref{lem:quad1}}\label{prooflem:quad1}

Let us use Hanson-Wright Lemma, that we recall here (see e.g. \cite{HDS2})
\begin{lem}\label{HW} {\bf (Hanson-Wright inequality)}
    Let $\eps\sim\mathcal{N}(0,I_{d})$ for some d. Let S be a real symmetric $p\times p$ matrix. Then, for all $x>0$, $$\P\left[\eps^{T}S\eps -\tr(S)>\sqrt{8\|S\|_{F}^{2}x}\vee(8\|S\|_{op}x)\right]\leq e^{-x}.$$ 
\end{lem}
That inequality implies that, for a fixed matrix $B\in\mathcal{B}$, with probability at least $1-e^{-x}$,\begin{align*}
    \<EE^{T}-pI_{n},(I-B^{*})(B-B^{*})(I-B^{*})\>\leq&\sqrt{8px\|(I-B^{*})(B-B^{*})(I-B^{*})\|_{F}^{2}}\\
    &\vee (8\|(I-B^{*})(B-B^{*})(I-B^{*})\|_{op}x)\enspace. 
\end{align*}
Our task will be to bound both $\|(I-B^{*})(B-B^{*})(I-B^{*})\|_{op}$ and $\|(I-B^{*})(B-B^{*})(I-B^{*})\|_{F}$ with respect to $\delta_{B}$. We will then do an union bound on all $B\in\mathcal{B}_{j}:=\{B\in\mathcal{B}, \delta_{B}\in [j-1,j]\}$, and finally a union bound on all $j$ in $[1,2n]$. From Lemma 1 of \cite{giraud2019partial}, we have $\delta_{B}=2\sum_{l\neq k}|B_{G_{k}^{*}G_{l}^{*}}|_{1}$. For all $i\in G^{*}_{k}$, $\sum_{l\neq k}|B_{iG_{l}^{*}}|\leq 1$. This implies $\delta_{B}\leq 2n.$ Hence, it is sufficient to consider $j\leq 2n$.
\newline
We have, for $j\in[1,2n]$ and $B\in\mathcal{B}_{j}$, \begin{enumerate}
    \item $\|(I-B^{*})(B-B^{*})(I-B^{*})\|_{op}\leq \|(I-B^{*})(B-B^{*})(I-B^{*})\|_{*}$ where $\|\cdot\|_{*}$ denotes the nuclear norm. And, Lemma 1 of \cite{giraud2019partial} shows that $\|(I-B^{*})(B-B^{*})(I-B^{*})\|_{*}\lesssim \frac{1}{m}\delta_{B}$. Hence: $\|(I-B^{*})(B-B^{*})(I-B^{*})\|_{op}\lesssim \frac{1}{m}\delta_{B}$;
    \item Since $(I-B^{*})(B-B^{*})(I-B^{*})$ is a difference of product of projections, all its eigen-values are bounded by 2 in absolute value. Hence: $\|(I-B^{*})(B-B^{*})(I-B^{*})\|_{op}\lesssim \frac{1}{j}\delta_{B}$;
    \item $\|(I-B^{*})(B-B^{*})(I-B^{*})\|_{F}\lesssim \|(I-B^{*})(B-B^{*})(I-B^{*})\|_{*}$ and therefore, we also have $\|(I-B^{*})(B-B^{*})(I-B^{*})\|_{F}\lesssim \frac{1}{m}\delta_{B}$;
    \item Since all the eigen-values of $(I-B^{*})(B-B^{*})(I-B^{*})$ are bounded by 2 in absolute value, we have $\|(I-B^{*})(B-B^{*})(I-B^{*})\|_{F}\lesssim \sqrt{\|(I-B^{*})(B-B^{*})(I-B^{*})\|_{*}}$. Thus, $\|(I-B^{*})(B-B^{*})(I-B^{*})\|_{F}\lesssim \frac{1}{\sqrt{mj}}\delta_{B}$.
\end{enumerate}
These inequalities on the norms imply that, for $B\in\mathcal{B}_{j}$ and $x>0$, with probability at least $1-e^{-x}$, $$\<EE^{T}-I_{n},(I-B^{*})(B-B^{*})(I-B^{*})\>\lesssim \delta_{B}\left(\sqrt{\frac{p}{m}\frac{x}{m\vee j}}+\frac{x}{m\vee j}\right).$$
Doing an union bound on $\mathcal{B}_{j}$, it appears that for $x>0$, this inequality remains true with probability at least $1-|\mathcal{B}_{j}|e^{-x}$, simultaneously on all $B\in\mathcal{B}_{j}$. 
This implies that for a positive constant $c_{1}$, with probability at least $1-\frac{c_{1}}{2n^{3}}$, simultaneously on all $B\in\mathcal{B}_{j}$:$$\<EE^{T}-I_{n},(I-B^{*})(B-B^{*})(I-B^{*})\>\lesssim \delta_{B}\left(\sqrt{\frac{p}{m}\frac{\log(|\mathcal{B}_{j}|\vee n)}{j\vee m}}\vee \frac{\log(|\mathcal{B}_{j}|\vee n)}{j\vee m}\right)\enspace.$$
The next lemma, proved in Section \ref{proofdenomb}, bounds the cardinality of $|\mathcal{B}_{j}|$.
\begin{lem}\label{denomb}
    For any $j\in[1,2n]$, $|\mathcal{B}_{j}|\leq \binom{n}{j\wedge n}K^{3j}2^{j\frac{m^{+}}{m}}$.
\end{lem}
Hence, with probability at least $1-\frac{c_{1}}{2n^{3}}$, simultaneously on all $B\in\mathcal{B}_{j}$, \begin{align*}
    \<EE^{T}-I_{n},(I-B^{*})(B-B^{*})(I-B^{*})\>\lesssim&\delta_{B}(\sqrt{\frac{p}{m}\frac{\log(n)+\log(\binom{n}{j\wedge n})+j\log(K)+j\frac{m^{+}}{m}}{j\vee m}}\\
    &+\frac{\log(n)+\log(\binom{n}{j\wedge n})+j\log(K)+j\frac{m^{+}}{m}}{j\vee m})\enspace.
\end{align*}
And,\begin{itemize}
    \item We have $\log\binom{n}{j\wedge n}\lesssim j\log(\frac{n}{j\wedge n})$. The function $x\rightarrow x\log(\frac{n}{x})$ being increasing on $]0,\frac{n}{e}]$, we get that, when $j\leq m$ and $m\leq \frac{n}{e}$, $j\log(\frac{n}{j\wedge n})\lesssim m\ln(\frac{n}{m})$.  When $j\geq m$, we have $j\log(\frac{n}{j\wedge n})\leq j\log(\frac{n}{m})$. Moreover, if $m\geq \frac{n}{e}$ and $j\leq n$, $j\log(\frac{n}{j\wedge n})\leq \frac{n}{e}\log(e)\lesssim m\log(\frac{n}{m})$. If $m\geq \frac{n}{e}$ and $j\geq n$, then $j\log(\frac{n}{j\wedge n})=0\leq m\log(n/m)$. Hence, we get that $\log\binom{n}{j\wedge n}\lesssim (m\vee j)\log(\frac{n}{m})$;
    \item $j\log(K)\leq j\log(\frac{n}{m})$ since $K\leq \frac{n}{m}$;
    \item The fact that the function $x\rightarrow x\log(\frac{n}{x})$ is increasing on $]0,\frac{n}{e}]$ also implies that, if $m\leq \frac{n}{e}$, $\log(n)\lesssim m\log(\frac{n}{m})$. If $m\geq \frac{n}{e}$, we have $\log(n)\lesssim n\lesssim \frac{n}{e}\lesssim m\log(n/m)$. 
\end{itemize}
Thus, with the same probability $1-\frac{c_{1}}{2n^{3}}$, simultaneously on all $B\in\mathcal{B}_{j}$, $$\<EE^{T}-I_{n},(I-B^{*})(B-B^{*})(I-B^{*})\>\lesssim\delta_{B}\pa{\log(n/m)\vee \frac{m^{+}}{m}+\sqrt{\frac{p}{m}\pa{\log(n/m)\vee \frac{m^{+}}{m}}}}.$$
An union bound on $j\in[1,2n]$ concludes the proof of the lemma.

\subsubsection{Proof of Lemma \ref{denomb}}\label{proofdenomb}

In this section, we consider the partitions $G$ such that $B(G)\in\mathcal{B}_{j}$. For $B\in\mathcal{B}_{j}$, denote $B_{G^{*}_{k}G^{*}_{l}}$ the restriction of $B$ where we keep the rows belonging to $G^{*}_{k}$ and the columns belonging to $G^{*}_{l}$. From Lemma 1 of \cite{giraud2019partial}, we get that, $\delta_{B}=2\sum_{l\neq k}|B_{G^{*}_{k}G^{*}_{l}}|$, for $B\in\mathcal{C}$. This equality tells us that if a point $i\in G^{*}_{r}$, for some $r\in[1,K]$, is linked in $G$ to a majority of points not belonging to the same $G^{*}_{r}$, then the contribution of $2\sum_{k:k\neq r}|B_{iG_{r}^{*}}|$ in $\delta_{B}$ is at least of one. Indeed, denoting by $l$ the index such that $i\in G_{l}$, we have \begin{align*}
    2\sum_{k:k\neq r}|B_{iG_{r}^{*}}|&=2\sum_{j\in G_{l}\setminus G^{*}_{r}} \frac{1}{|G_{l}|}\\
    &=2\frac{|G_{l}\setminus G^{*}_{r}|}{|G_{l}|}\ \geq 1\enspace.
\end{align*}

In order to bound $|\mathcal{B}_{j}|$, we will find, for a partition $G$ whose normalized-partnership matrix is in $\mathcal{B}_{j}$, a labelling $G_{1},...,G_{K}$ for which we can upper-bound the possibilities for choosing points that are in some $G_{r}^{*}\cap G_{l}$ for $l\neq r$. 

The equality that fulfills $\delta_{B}$ will help us bound the possibilities of choosing the points that are in some $G_{r}^{*}\cap G_{l}$, with $l\neq r$ and $\frac{|G_{l}\cap G^{*}_{r}|}{|G_{l}|}\leq \frac{1}{2}$. On the other hand, the following lemma allows us to control the number of cotuple $(l,r)$, with $l\neq r$, satisfying $\frac{|G_{l}\cap G^{*}_{r}|}{|G_{l}|}>\frac{1}{2}$, by stating that for any of these cotuple, we can consider that all the points belonging to $G_{l}^{*}$ add a contribution of at least one to $\delta_{B}$.
\begin{lem}\label{lem:denomb2}
    Let $G=\{G_{1},...,G_{K}\}$ be a partition of $[1,n]$. There exists $\phi:[1,K]\to [1,K]$ a bijection such that the following holds.  For all $l\neq r$ such that  $\frac{|G_{\phi(l)}\cap G^{*}_{r}|}{|G_{\phi(l)}|}>\frac{1}{2}$, we have  $\max_{l'}\frac{|G^{*}_{l}\cap G_{\phi(l')}|}{|G_{\phi(l')}|}\leq \frac{1}{2}$.
\end{lem}

This lemma is proved in Section \ref{proofdenomb2}. Let $G$ such that $B(G)\in\mathcal{B}_{j}$. Without loss of generality, we can suppose that the bijection $\phi$ considered in Lemma \ref{lem:denomb2} is the identity. Fix any $l\neq r$ in $[1, K]$. We consider two cases
\begin{enumerate}
    \item $\frac{|G^{*}_{l}\cap G_{r}|}{|G_{r}|}\leq\frac{1}{2}$. For any  $i\in G^{*}_{l}\cap G_{r}$, we have 
    \[
    \sum_{k:k\neq l}|B_{iG^{*}_{k}}|= \sum_{t\in [1,n]\setminus G^{*}_{l}}B_{it}\geq \frac{1}{|G_{r}|}|G_{r}\setminus G^{*}_{l}|\geq \frac{1}{2}\ .
    \] 
    \item $\frac{|G^{*}_{l}\cap G_{r}|}{|G_{r}|}>\frac{1}{2}$. Consider any  $i\in G^{*}_{r}$.  Then, denoting $l'$ the index such that $i\in G_{l'}$, we have from Lemma \ref{lem:denomb2} that  $\frac{|G^{*}_{r}\cap G_{l'}|}{|G_{l'}|}\leq \frac{1}{2}$. Arguing  as above, this implies $\sum_{k\neq r}|B_{iG^{*}_{k}}|\geq \frac{1}{2}$. Hence, $\sum_{k:k\neq r}|B_{G^{*}_{r}G^{*}_{k}}|\geq \frac{m}{2}$. 
\end{enumerate}
Since $B\in\mathcal{B}_{j}$, we deduce from the above discussion that (i) there are at most $j\wedge n$ points that belong to $G^{*}_{l}\cap G_{r}$ for some $l\neq r$ such that $\frac{|G^{*}_{l}\cap G_{r}|}{|G_{r}|}\leq\frac{1}{2}$ and (ii) that  $|\{(l,r), l\neq r, \frac{|G^{*}_{l}\cap G_{r}|}{|G_{r}|}>\frac{1}{2}\}|\leq \frac{j}{m}$.

Hence, $|\mathcal{B}_{j}|$ is upper-bounded by the number of partitions satisfying these two conditions. Such partitions are fully defined by the points $i$ that are in some $G^{*}_{l}\cap G_{r}$, for $l\neq r$. So, it is sufficient to count the possibilities of choosing these points:\begin{itemize}
    \item We choose $j\wedge m$ points, and for each of these points we decide a group $G_{r}$ where it is sent ($r$ is not necessarily different than the index of the group of $G^{*}$ that contains $i$); we have $\binom{n}{j\wedge n}K^{j\wedge n}$ such possibilities,
    \item We choose $\frac{j}{m}$ couples $(l,r)$ and, for each of these couples, we choose a subset of $G^{*}_{l}$ that will be a subset of $G_{r}$. That gives less than $2^{j\frac{m^{+}}{m}}K^{2\frac{j}{m}}$ possibilities. 
\end{itemize} 
Hence, we conclude that $|\mathcal{B}_{j}|\leq \binom{n}{j\wedge n}K^{3j}2^{j\frac{m^{+}}{m}}$.

\subsubsection{Proof of Lemma \ref{lem:denomb2}}\label{proofdenomb2}

Let $G=\{G_{1},...,G_{K}\}$ a partition of $[1,n]$. For $l\in[1,K]$ such that there exists $l'\in[1,K]$ which satisfies $\frac{|G^{*}_{l}\cap G_{l'}|}{|G_{l'}|}>\frac{1}{2}$, we define $\phi(l)=l'$. The mapping $\phi$ is an injection since for all $l'\in[1,K]$, there is at most one index $l$ such that $\frac{|G^{*}_{l}\cap G_{l'}|}{|G_{l'}|}>\frac{1}{2}$. Then, we can expand $\phi$ to $[1,K]$ in order to have a permutation of $[1,K]$.

The permutation $\phi$ verifies the following assertion. For $l\in[1,K]$, if there exists $l'$ such that $\frac{|G^{*}_{l}\cap G_{\phi(l')}|}{|G_{\phi(l')}|}>\frac{1}{2}$, then, we also have $\frac{|G^{*}_{l}\cap G_{\phi(l)}|}{|G_{\phi(l)}|}>\frac{1}{2}$. So, for $l\neq r$ such that $\frac{|G_{\phi(l)}\cap G^{*}_{r}|}{|G_{\phi(l)}|}>\frac{1}{2}$, there exists no $l'$ such that $\frac{|G^{*}_{l}\cap G_{\phi(l')}|}{|G_{\phi(l')}|}>\frac{1}{2}$. Indeed, otherwise, we would also have $\frac{|G^{*}_{l}\cap G_{\phi(l)}|}{|G_{\phi(l)}|}>\frac{1}{2}$. That would contradict the hypothesis $\frac{|G^{*}_{r}\cap G_{\phi(l)}|}{|G_{\phi(l)}|}>\frac{1}{2}$. 

This concludes the proof of the lemma.

\subsubsection{Proof of Lemma \ref{quad2}}\label{proofquad2}
For any $B$ in $\mathcal{B}$, Lemma \ref{HW} implies that, for $L\geq 0$, with probability at least $1-e^{-L}$,
\begin{align*}
    \<EE^{T}-I_{n},B^{*}(B-B^{*})\>\leq&\sqrt{8pL\|B^{*}(B-B^{*})\|_{F}^{2}}\\
    &\vee (8\|B^{*}(B-B^{*})\|_{op}L)
\end{align*}
Following the same arguments as in Section \ref{prooflem:quad1}, it is sufficient to show that, for $j\geq m$ and $B\in\mathcal{B}_{j}$, $\|B^{*}(B-B^{*})\|_{F}\lesssim \frac{1}{\sqrt{mj}}\delta_{B}$ and $\|B^{*}(B-B^{*})\|_{op}\lesssim \frac{1}{j}\delta_{B}$.
\begin{enumerate}
    \item $B^{*}(B-B^{*})$ is a product of a projection and a difference of projections. That implies that all its eigen-values are bounded by 2 in absolute value. Hence: $\|B^{*}(B-B^{*})\|_{op}\lesssim \frac{1}{j}\delta_{B}$.
    \item All the coefficients of $B^{*}$ are non-negative and bounded by $\frac{1}{m}$. Since the coefficients of $B^{*}B$ and of $B^{*}B^{*}$ are convex combinations of coefficients of $B^{*}$, it comes forward that all its coefficients are also bounded by $\frac{1}{m}$ in absolute value. That implies that the coefficients of $B^{*}(B-B^{*})$ are bounded by $\frac{2}{m}$ in absolute value. Hence: $\|B^{*}(B-B^{*})\|_{F}\lesssim \frac{1}{\sqrt{m}}\sqrt{\|B^{*}(B-B^{*})\|_{1}}\lesssim\frac{1}{\sqrt{jm}}\delta_{B}$.
\end{enumerate}
That concludes the proof for the term in $B^{*}(B-B^{*})$.

For the term in $B^{*}(B-B^{*})B^{*}$,we use the same arguments and add the fact that $\|B^{*}(B-B^{*})B^{*}\|_{1}=\|B^{*}(B-B^{*})\|_{1}$ (see Lemma 1 of \cite{giraud2019partial}).

\subsection{Proof of Lemma \ref{lem:erreur}}\label{prooflem:error}

In this section, we bound the proportion of misclassified points $$err(G,G^{*})=\frac{1}{2n}\min_{\pi\in\mathcal{S}_{K}}\sum_{k=1}^{K}|G_{k}^{*}\Delta G_{\pi(k)}|\enspace ,$$
with respect to $\|B^{*}-B^{*}B\|_{1}$, for a given partition $G$ and its normalized-partnership matrix $B$.

We consider $G$ a partition. By Lemma \ref{lem:denomb2}, there exists a bijection $\phi: [1,K]\to [1,K]$ such that the following holds. For all $l\neq r$, if $\frac{|G_{\phi(l)}\cap G^{*}_{r}|}{|G_{\phi(l)}|}>\frac{1}{2}$, then for all $l'$, $\frac{|G_{\phi(l')}\cap G^{*}_{l}|}{|G_{\phi(l')}|}\leq \frac{1}{2}$. Without loss of generality, we suppose that this bijection is the identity. 

Then, by definition, $$err(G,G^{*})\leq \frac{1}{2n}\sum_{k=1}^{K}|G_{k}^{*}\Delta G_{k}|\enspace .$$
So, we have $$err(G,G^{*})\leq \frac{1}{n}\sum_{l\neq r}\sum_{i\in [1,n]}\1_{i\in G^{*}_{r}\cap G_{l}}\enspace.$$
Let $l\neq r$. If $\frac{|G^{*}_{r}\cap G_{l}|}{|G_{l}|}\leq \frac{1}{2}$, each $i\in G^{*}_{r}\cap G_{l}$ adds a contribution of at least 1 in $\|B^{*}-B^{*}B\|_{1}$. Moreover, if $\frac{|G^{*}_{r}\cap G_{l}|}{|G_{l}|}> \frac{1}{2}$, each point of $G^{*}_{l}$ adds a contribution of at least 1 in $\|B^{*}-B^{*}B\|_{1}$. So, we can match any set $G^{*}_{r}\cap G_{l}$ such that $\frac{|G^{*}_{r}\cap G_{l}|}{|G_{l}|}> \frac{1}{2}$, to a group $G_{l}^{*}$ which contains points all adding a contribution of at least $1$ to $\|B^{*}-B^{*}B\|_{1}$. Since $|G^{*}_{l}|\geq m$ and $|G^{*}_{r}\cap G_{l}|\leq m^{+}$, $|G_{l}^{*}|\geq \frac{m}{m^{+}}|(G^{*}_{r}\cap G_{l}|$. Let us denote $A_{0}$ the set of points that are in some $G^{*}_{r}\cap G_{l}$, with $l\neq r$, satisfying $\frac{|G^{*}_{r}\cap G_{l}|}{|G_{l}|}> \frac{1}{2}$. We also denote $A_{1}$ the set of points that are in some $G^{*}_{r}\cap G_{l}$, satisfying $\frac{|G^{*}_{r}\cap G_{l}|}{|G_{l}|}\leq \frac{1}{2}$. We then have
 \begin{align*}
    |A_{0}|&=\sum_{l\neq r: \frac{|G^{*}_{r}\cap G_{l}|}{|G_{l}|}> \frac{1}{2}} |G^{*}_{r}\cap G_{l}|\\
    &\leq  \frac{m^{+}}{m} \sum_{l\neq r: \frac{|G^{*}_{r}\cap G_{l}|}{|G_{l}|}> \frac{1}{2}} |G_{l}^{*}|\enspace.
\end{align*}
Given $l\neq r$ such that $\frac{|G^{*}_{r}\cap G_{l}|}{|G_{l}|}> \frac{1}{2}$, the following assertion is satisfied: For all $l'\in[1,K]$, 
$\frac{|G^{*}_{l}\cap G_{l'}|}{|G_{l'}|}\leq \frac{1}{2}$. Hence, $|G_{l}^{*}|=\sum_{l': \frac{|G^{*}_{l}\cap G_{l'}|}{|G_{l'}|}\leq \frac{1}{2}}|G_{l}^{*}\cap G_{l'}|$. Moreover, for all $l\in[1,K]$, there exists at most one index $r$ such that $\frac{|G^{*}_{r}\cap G_{l}|}{|G_{l}|}> \frac{1}{2}$. Hence,
\begin{align*}
    |A_{0}|\leq& \frac{m^{+}}{m}\sum_{l', l: \frac{|G^{*}_{l}\cap G_{l'}|}{|G_{l'}|}\leq \frac{1}{2}} |G^{*}_{l}\cap G_{l'}|\\
    \leq& \frac{m^{+}}{m}\|B^{*}-B^{*}B\|_{1}\enspace, 
\end{align*}
where the last inequality comes from the fact that, for $l,l'\in[1,K]$, if $\frac{|G^{*}_{l}\cap G_{l'}|}{|G_{l'}|}\leq \frac{1}{2}$, each point of $|G^{*}_{l}\cap G_{l'}|$ adds a contribution of at least 1 to $\|B^{*}-B^{*}B\|$.

Since all $i\in A_{1}$ adds a contribution of at least $1$ to $\|B^{*}-B^{*}B\|_{1}$, we have $|A_{1}|\leq \|B^{*}-B^{*}B\|_{1}$.

This leads to $\sum_{l\neq r}\sum_{i\in [1,n]}\1_{i\in G^{*}_{r}\cap G_{l}}=|A_{0}|+|A_{1}|\leq \|B^{*}-B^{*}B\|_{1}(1+\frac{m^{+}}{m})\leq 2\|B^{*}-B^{*}B\|_{1}\frac{m^{+}}{m}$.

Therefore, the proportion of misclustered points is upper-bounded by
$$err(G,G^{*})\leq 2 \frac{m^{+}}{m} \|B^{*}-B^{*}B\|_{1}\enspace .$$
This concludes the proof of Lemma \ref{lem:erreur}.

\subsection{The signal term (Proof of Lemma~\ref{lem:signal})}

The term $\<A\mu(A\mu)^{T},B^{*}-B\>$ is already dealt in \cite{giraud2019partial}. We re-derive Lemma~\ref{lem:signal} for the sake of completeness. We denote $S\in\R^{n\times n}$ the matrix defined by $S_{ab}=-0.5\|\mu_{k}-\mu_{l}\|^{2}$, if $a\in G^{*}_{k}$ and $b\in G^{*}_{l}$. For $B\in\mathcal{C}$, we get
\begin{align*}
    \<A\mu(A\mu)^{T},B^{*}-B\>&=\sum_{a,b\in [1,n]}\<\mu_{a},\mu_{b}\>\pa{B^{*}_{ab}-B_{ab}}\\
    &=\sum_{a,b\in [1,n]}\pa{-0.5 \|\mu_{a}-\mu_{b}\|^{2}+0.5\|\mu_{a}\|^{2}+0.5\|\mu_{b}\|^{2}}\pa{B^{*}_{ab}-B_{ab}}\\
    &=\sum_{a,b\in [1,n]}-0.5 \|\mu_{a}-\mu_{b}\|^{2}\pa{B^{*}_{ab}-B_{ab}}\\
    &= \<S,B^{*}-B\>\enspace,
\end{align*}
where the third equality comes for the fact that, for $a\in[1,n]$, $$\sum_{b\in[1,n]}\|\mu_{a}\|^{2}\pa{B_{ab}-B^{*}_{ab}}=0\enspace .$$
The term $S_{ab}$ being null when $a$ and $b$ are in the same group, we have 
\begin{align*}
    \<S,B-B^{*}\>=& 0.5\sum_{i\neq k}\sum_{a\in G^{*}_{k}}\sum_{b\in G^{*}_{i}}\|\mu_{k}-\mu_{i}\|^{2}B_{ab}\\
    \geq& \Delta^{2}\sum_{i\neq k}\sum_{a\in G^{*}_{k}}\sum_{b\in G^{*}_{i}}B_{ab}\\
    \geq& \Delta^{2}\frac{1}{2}\|B^{*}-B^{*}B\|_{1}\enspace ,
\end{align*}
where the last inequality comes from Lemma 1 of \cite{giraud2019partial}, which states that for any matrix $B$ belonging to the larger class of matrix $\mathcal{C}$, we have the equality $\|B^{*}-B^{*}B\|_{1}=2\sum_{i\neq k}\sum_{a\in G^{*}_{k}}\sum_{b\in G^{*}_{i}}B_{ab}.$ 

This concludes the proof of Lemma \ref{lem:signal}.

\section{Proof of Theorem \ref{thm:lowerboundpartial}}\label{prf:lowerboundpartial}

In this section, we prove Theorem \ref{thm:lowerboundpartial}. We suppose that $p\geq c\log(K)$, for $c$ a numerical constant that we will choose large enough later. Without loss of generality, we suppose throughout this proof that $\sigma=1$. We suppose $n\geq 2K$, $K\geq K_0$, for $K_0$ a numerical constant that we will choose large enough, and $\alpha\geq \frac{3}{2}$.

Given $\rho$ a probability distribution on $(\R^{p})^{K}$ and a partition $G$ of $[1,n]$ in $K$ groups, we define the probability distribution on $(\R^{p})^{n}$ by $$\P_{\rho,G}(B)=\int \P_{\mu,G}(B)d\rho(\mu)\enspace.$$
To prove Theorem \ref{thm:lowerboundpartial}, we will use three lemmas. The first one, proved in Section \ref{prf:Fano2}, is a consequence of Fano's lemma that we recall in Section \ref{prf:Fano2}.

\begin{lem}\label{lem:Fano2}
    Let $\rho$ be a probability measure on $\R^{p\times K}$. For any finite set $A\subset \mathcal{P}_{\alpha}$, any partition $G^{(0)}\in\mathcal{P}_{\alpha}$, we have the following inequality $$\inf_{\hat{G}}\frac{2}{|A|}\sum_{G\in A}\E_{\rho, G}[err(G,\hat{G})]\geq \min_{G\neq G'\in \mathcal{P}_{\alpha}}err(G,G')\pa{1-\frac{1+\frac{1}{|A|}\sum_{G\in A}KL\pa{\P_{\rho,G},\P_{\rho,G^{(0)}}}}{\log|A|}}.$$
\end{lem}
The second lemma is a reduction lemma, which plays the same role as Lemma \ref{lem:reduction} in the proof of Theorem \ref{thm:lowerboundexact}. We prove it in Section \ref{prf:reduction2}. 
\begin{lem}\label{lem:reduction2}
    Suppose that there exists a probability distribution $\rho$ on $\R^{p\times K}$ and a numerical constant $C>0$ satisfying $$\inf_{\hat{G}}\sup_{G\in\mathcal{P}_{\alpha}}\E_{\rho,G}\cro{err(\hat{G},G)}-\rho(\R^{p\times K}\setminus \Theta_{\Bar{\Delta}})\geq C\enspace.$$ Then, we have $$\inf_{\hat{G}}\sup_{\mu\in\Theta_{\Bar{\Delta}}}\sup_{G\in\mathcal{P}_{\alpha}}\E_{\mu,G}\cro{err(\hat{G},G)}\geq C\enspace.$$
\end{lem}
Finally, the third lemma helps us choose the set $A$ of partitions to whom we will apply Lemma \ref{lem:Fano2}. We prove it in Section \ref{prf:numberpartitions}. We define $\overline{G}$ the partition of $[1,n]$ defined by; $i\in \overline{G}_{k}$ if and only if $i\equiv k\enspace [K]$. It is clear that $\overline{G}\in\mathcal{P}_{\frac{3}{2}}$.

\begin{lem}\label{lem:numberpartitions}
    We suppose that the constant $K_0$ such that $K\geq K_{0}$ is large enough and that $n\geq 2K$. There exists $S\subset\mathcal{P}_{\frac{3}{2}}$ which satisfies:
    \begin{itemize}
        \item There exists a numerical constant $c'>0$ such that $\log|S|\geq c''n\log(K)$,
        \item There exists a numerical constant $a>0$ such that, for $G\neq G'\in S$, $err(G,G')\geq a$,
        \item For all $k\in [1,K]$, for all $G\in S$, $|G_k|=|\overline{G}_{k}|$.
    \end{itemize}
\end{lem}

We distinguish two cases. In the first one, we suppose that there exists a numerical constant $c_{1}$, that we will choose small enough, such that $\Bar{\Delta}^{2}\leq c_{1}\log(K)$. In the second one, we will suppose that $c_{1}\log(K)\leq \Bar{\Delta}^{2}\leq c_{2}\sqrt{\frac{p}{n}K\log(K)}$, for $c_{2}$ a numerical constant that we will also choose small enough.

\subsection{Case $\Bar{\Delta}^{2}\leq c_{1}\log(K)$}

In this section, we suppose that $\Bar{\Delta}^{2}\leq c_{1}\log(K)$. We suppose that the constant $c$ such that $p\geq c\log(K)$ is larger than $4/\log(2)$, for having $\frac{p}{4}\log(2)\geq \log(K)$. Our choice of the centers $\mu_{1},\ldots, \mu_{K}$ relies on the following lemma, proved in Section \ref{prf:choicepoints}.

\begin{lem}\label{lem:choicepoints}
    If $\frac{p}{4}\log(2)\geq \log(K)$, there exists $\mu_{1},\ldots ,\mu_{K}$ in $\R^{p}$ such that $\frac{1}{2}\min_{l\neq r}\|\mu_{l}-\mu_{r}\|^2\geq \Bar{\Delta}^{2}$ and $\frac{1}{2}\max_{l\neq r}\|\mu_{l}-\mu_{r}\|^{2}\leq 4\Bar{\Delta}^2$.
\end{lem}

We consider $\mu=(\mu_{1},\ldots ,\mu_{K})$ in $(\R^{p})^{K}$ given by this lemma. In particular, we have $\mu\in \Theta_{\Bar{\Delta}}$.

We recall that $\overline{G}$ is the partition of $[1,n]$ defined by; $i\in \overline{G}_{k}$ if and only if $i\equiv k\enspace [K]$. Given $G$ taken in the set $S$ defined in Lemma \ref{lem:numberpartitions}, let us compute $KL(\P_{\mu,G},\P_{\mu,\overline{G}})$. We denote $\P_{\mu,G}(Y_i)$ the marginal law of $Y_i$ under the joint law $\P_{\mu,G}$. By independence of all the $Y_i$, we have that 
$$KL(\P_{\mu,G},\P_{\mu,\overline{G}})=\sum_{i=1}^{n}KL(\P_{\mu,G}(Y_i),\P_{\mu,\overline{G}}(Y_i))\enspace.$$
Given $i\in[1,n]$, with $k$ and $l$ such that $i\in G_k\cap\overline{G}_{l}$, we have $KL(\P_{\mu,G}(Y_i),\P_{\mu,\overline{G}}(Y_i))=\frac{\|\mu_{k}-\mu_{l}\|^{2}}{2}$. A fortiori, using Lemma \ref{lem:choicepoints},
$$KL(\P_{\mu,G},\P_{\mu,\overline{G}})\leq 4n \Bar{\Delta}^{2}\leq 4c_{1} n\log(K)\enspace.$$ Applying this, together with $\log|S|\geq c''n\log(K)$ and Lemma \ref{lem:Fano2} with $\rho=\delta_{(\mu_1,\ldots,\mu_k)}$ leads to 
$$\inf_{\hat{G}}\frac{2}{|S|}\sum_{G\in S}\E_{\mu,G}\cro{err(\hat{G},G)}\geq a\pa{1-\frac{1+4c_{1}^{2}n\log(K)}{\log|S|}}\geq a\pa{1-\frac{1+4c_{1}^{2}n\log(K)}{c''n\log(K)}}\enspace.$$
The quantity $a\pa{1-\frac{1+4c_{1}^{2}n\log(K)}{c''n\log(K)}}$ being larger than $\frac{a}{2}$, supposing $c_{1}$ is small enough and $K_{0}$ large enough, there exists a constant $C$ such that $$\inf_{\hat{G}}\frac{2}{|S|}\sum_{G\in S}\E_{\mu,G}\cro{err(\hat{G},G)}\geq C\enspace.$$ A fortiori, since $\mu\in\Theta_{\Bar{\Delta}}$ and $S\subset\mathcal{P}_{\alpha}$, 
$$\inf_{\hat{G}}\sup_{G\in\mathcal{P}_{\alpha}}\sup_{\mu\in\Theta_{\Bar{\Delta}}}\E_{\mu,G}\cro{err(\hat{G},G)}\geq C\enspace.$$
This concludes the proof of the theorem in this case.

\subsection{Case $c_{1}\log(K)\leq \Bar{\Delta}^{2}\leq c_{2}\sqrt{\frac{p}{n}K\log(K)}$}

In this section, we suppose that $c_{1}\log(K)\leq \Bar{\Delta}^{2}\leq c_{2}\sqrt{\frac{p}{n}K\log(K)}$, with $c_{2}$ a numerical constant that we will choose small enough. We still suppose that $K\geq K_0$.

Let us define $\rho$ the uniform distribution on the hypercube $\mathcal{E}=\{-\eps, +\eps\}^{pK}$, where $\eps=\sqrt{\frac{2}{p}}\Bar{\Delta}$. We show in Section \ref{prf:ineqKLpartial} the following lemma, which controls $KL(\P_{\rho, G},\P_{\rho,\overline{G}})$, for $G\in S$.

\begin{lem}\label{lem:ineqKLpartial}
    If $c_{2}$ is small enough with respect to $c_1$ and $c_{1}\log(K)\leq \Bar{\Delta}^{2}\leq c_{2}\sqrt{\frac{p}{n}K\log(K)}$, there exists a numerical constant $c>0$ such that, for all $G\in S$, $KL(\P_{\rho, G},\P_{\rho,\overline{G}})\leq cc_{2}^{2}n\log(K)$.
\end{lem}

Combining this lemma with Lemma \ref{lem:Fano2} implies that $$\inf_{\hat{G}}\frac{2}{|S|}\sum_{G\in S}\E_{\rho,G}\cro{err(\hat{G},G)}\geq \min_{G\neq G'\in S}err(G,G')\pa{1-\frac{1+cc_{2}^{2}n\log(K)}{c''n\log(K)}}\enspace.$$ 
By definition of $S$ in Lemma \ref{lem:numberpartitions}, given $G\neq G'\in S$, we have $err(G,G')\geq a$. Hence, $$\inf_{\hat{G}}\frac{2}{|S|}\sum_{G\in S}\E_{\rho,G}\cro{err(\hat{G},G)}\geq a\pa{1-\frac{1+cc_{2}^{2}n\log(K)}{c''n\log(K)}}\enspace.$$
This last quantity is larger than $\frac{a}{2}$ provided $c_{2}$ is small enough and $K_0$ large enough. Hence, since $S\subset \mathcal{P}_{\alpha}$, this implies 
$$\inf_{\hat{G}}\sup_{G\in \mathcal{P}_{\alpha}}\E_{\rho,G}\cro{err(\hat{G},G)}\geq \frac{a}{2}\enspace.$$
In order to apply Lemma \ref{lem:reduction2}, it remains to upper-bound the quantity $\rho(\R^{p\times K}\setminus \theta_{\Bar{\Delta}})$. The next lemma, proved in Section \ref{prf:lowerboundseparation}, provides such an upper-bound.

\begin{lem}\label{lem:lowerboundseparation}
    For all $K\geq 2$, $p>1$ and $\Bar{\Delta}>0$, the probability distribution $\rho$ satisifes $$\rho(\R^{p\times K}\setminus \Theta_{\Bar{\Delta}})\leq \frac{K(K-1)}{2}e^{{-p}/{8}}\enspace.$$
\end{lem}

In particular, if the constant $c$ such that $p\geq c\log(K)$ is large enough, the quantity $\frac{K(K-1)}{2}e^{{-p}/{8}}$ is smaller than ${a}/{4}$. This, together with Lemma \ref{lem:reduction2}, leads to the existence of a numerical constant $C>0$ such that 
$$\inf_{\hat{G}}\sup_{\mu\in\Theta_{\Bar{\Delta}}}\sup_{G\in\mathcal{P}_{\alpha}}\E_{\mu,G}\cro{err(\hat{G},G)}>C\enspace.$$
This concludes the proof of the Theorem \ref{thm:lowerboundpartial}.

\subsection{Proof of Lemma \ref{lem:Fano2}}\label{prf:Fano2}
Lemma \ref{lem:Fano2} is a simple derivation from Fano's Lemma that we recall here (see e.g \cite{HDS2}).
\begin{lem}[Fano's Lemma]\label{lem:fano}
    Let $(\P_{j})_{j\in[1,M]}$ be a set of probability distributions on some set $\mathcal{Y}$. For any probability distribution $\mathbb{Q}$ such that for all $j\in[1,M]$, $\P_{j}<<\mathbb{Q}$, $$\inf_{\hat{J}:\mathcal{Y}\to [1,M]}\frac{1}{M}\sum_{j=1}^{M}\P_{j}\pa{\hat{J}(Y)\neq j}\geq 1-\frac{1+\frac{1}{M}\sum_{j=1}^{M}KL(\P_{j},\mathbb{Q})}{\log(M)}\enspace , $$ where we recall that $KL(\P,\mathbb{Q})=\int \log\pa{\frac{d\P}{d\mathbb{Q}}}d\P$ stands for the Kullback-Leibler divergence between $\P$ and $\mathbb{Q}$. 
\end{lem}

Let $A$ be a subset of $\mathcal{P}_{\alpha}$. Denote $G^{(1)},\ldots ,G^{(|A|)}$ the elements of $A$. Given any estimator $\hat{G}$, we denote $\hat{j}$ an index that minimises $err(\hat{G},G^{(j)})$.

For any $j\in[1,|A|]$, using the definition of $\hat{j}$ together with the fact that the function $err$ satisfies the triangular inequality (for the second inequality), we have 
\begin{align*}
    \min_{i\neq i'}err(G^{(i)},G^{(i')})\1_{\hat{j}\neq j}&\leq err(G^{(\hat{j})},G^{(j)})\\
    &\leq err(\hat{G},G^{(\hat{j})})+err(\hat{G},G^{(j)})\\
    &\leq 2err(\hat{G},G^{(j)})\enspace.
\end{align*}
Applying the expectation to this last inequality, we have, for $j\in[1,|A|]$, 
$$2\E_{\rho,G^{(j)}}\cro{err(\hat{G},G^{(j)})}\geq \min_{i\neq i'}err(G^{(i)},G^{(i')}) \P_{\rho,G^{(j)}}\cro{\hat{j}\neq j}\enspace. $$ 
Summing over $G\in A$ leads to 
$$\frac{2}{|A|}\sum_{G\in A}\E_{\rho,G}\cro{err(\hat{G},G)}\geq \min_{G,G'\in A}err(G,G')\min_{\hat{G}}\frac{1}{|A|}\sum_{G^{(j)}\in A} \P_{\rho,G^{(j)}}\cro{\hat{j}\neq j}.$$
Applying Fano's  Lemma \ref{lem:fano}, we get the sought inequality 
$$\inf_{\hat{G}}\frac{2}{|A|}\sum_{G\in A}\E_{\rho,G}\cro{err(\hat{G},G)}\geq \min_{G,G'\in A}err(G,G')\pa{1-\frac{1+\frac{1}{|A|}\sum_{G\in A}KL(\P_{\rho,G},\P_{\rho,G^{(0)}})}{\log(|A|)}}.$$
This concludes the proof of the Lemma  \ref{lem:Fano2} .

\subsection{Proof of Lemma \ref{lem:ineqKLpartial}}\label{prf:ineqKLpartial}

For $G\in S$, let us compute $KL(\P_{\rho, G},\P_{\rho,\overline{G}})$. We recall that $i\in\overline{G}_k$ if and only if $i\equiv k\enspace[K]$. From Lemma \ref{lem:numberpartitions}, we have that, for $k\in[1,K]$, $|G_{k}|=|\overline{G}_{k}|$. We define $m_k$ this quantity. Given $k,l\in[1,K]$, we write $m_{kl}=|\overline{G}_{k}\cap G_{l}|$. 

We recall the definition $KL(\P_{\rho, G},\P_{\rho,\overline{G}})=\int \log\pa{\frac{d\P_{\rho,G}}{d\P_{\rho,\overline{G}}}}d\P_{\rho,G}$. We write $\P_{0,G}$, or equivalentely $\P_{0,\overline{G}}$, the distribution of $(Y_i)_{i\in[1,n]}$ when $\mu_{1}=\ldots,\mu_{K}=0$ almost surely. Under this distribution, the $Y_i$'s are drawn independently according $\N\pa{0,I_p}$. First, we will compute the quantity 
\begin{equation}\label{eq:vraisemblancepartial}
    \frac{d\P_{\rho,G}}{d\P_{\rho,\overline{G}}}=\frac{\frac{d\P_{\rho,G}}{d\P_{0,G}}}{\frac{d\P_{\rho,\overline{G}}}{d\P_{0,G}}}=\frac{\frac{d\P_{\rho,G}}{d\P_{0,G}}}{\frac{d\P_{\rho,\overline{G}}}{d\P_{0,\overline{G}}}}\enspace,
\end{equation} 
where the second equality comes from the fact that $\P_{0,G}=\P_{0,\overline{G}}$. 
Given a probability distribution  $\P$ on some Euclidean space, which is absolutely continuous with respect to the Lebesgue measure, we write $d\P$ for the density of this distribution with respect to the Lebesgue measure.
  For the numerator in \eqref{eq:vraisemblancepartial}, we have
\begin{align*}
    \frac{d\P_{\rho,G}}{d\P_{0,G}}(Y)=&\frac{\E_{\rho}\cro{d\P_{\mu,G}(Y)}}{\E_{0}\cro{d\P_{\mu,G}(Y)}}\\
    =&\frac{\E_{\rho}\cro{\prod_{k\in[1,K]}\prod_{i\in G_{k}}\exp\pa{-\frac{1}{2}\|Y_{i}-\mu_{k}\|^{2}}}}{\E_{0}\cro{\prod_{k\in[1,K]}\prod_{i\in G_{k}}\exp\pa{-\frac{1}{2}\|Y_{i}-\mu_{k}\|^{2}}}}\\
    =&\frac{\E_{\rho}\cro{\prod_{d\in[1,p]}\prod_{k\in[1,K]}\prod_{i\in G_{k}}\exp\pa{-\frac{1}{2}\pa{Y_{i,d}-\mu_{k,d}}^{2}}}}{\prod_{d\in[1,p]}\prod_{k\in[1,K]}\prod_{i\in G_{k}}\exp\pa{-\frac{1}{2}Y_{i,d}^2}}\enspace.
\end{align*}
Using the independence of the $\mu_{k,d}$'s under the law $\rho$, we get that 
\begin{align*}
    \frac{d\P_{\rho,G}}{d\P_{0,G}}(Y)=&\prod_{d\in[1,p]}\prod_{k\in[1,K]}\frac{\E_{\rho}\cro{\prod_{i\in G_{k}}\exp\pa{-\frac{1}{2}(Y_{i,d}-\mu_{k,d})^{2}}}}{\prod_{i\in G_{k}}\exp\pa{-\frac{1}{2}Y_{i,d}^{2}}}\\
    =&\prod_{d\in[1,p]}\prod_{k\in[1,2]}\E_{\rho}\cro{\prod_{i\in G_{k}}\exp{\pa{-\frac{1}{2}\pa{(Y_{i,d}-\mu_{k,d})^{2}-(Y_{i,d})^{2}}}}}\\
    =&\prod_{d\in[1,p]}\prod_{k\in[1,K]}\E_{\rho}\cro{\prod_{i\in G_{k}}\exp{\pa{Y_{i,d}\mu_{k,d}-\frac{\eps^{2}}{2}}}}\\
    =&\prod_{d\in[1,p]}\prod_{k\in[1,K]}e^{\frac{-m_{k}\eps^{2}}{2}}\cosh{\pa{\sum_{i\in G_{k}}\eps Y_{i,d}}}\enspace.
\end{align*}
Similarly, we have
\begin{equation*}
    \frac{d\P_{\rho,\overline{G}}}{d\P_{0,\overline{G}}}(Y)=\prod_{d\in[1,p]}\prod_{k\in[1,K]}e^{\frac{-m_{k}\eps^{2}}{2}}\cosh{\pa{\sum_{i\in \overline{G}_{k}}\eps Y_{i,d}}}\enspace.
\end{equation*}
Combining these two equalities in \eqref{eq:vraisemblancepartial}, we end up with
\begin{equation}\label{eq:vraisemblance2partial}
    \frac{d\P_{\rho,G}}{d\P_{\rho,\overline{G}}}(Y)=\prod_{d=1}^{p}\prod_{k\in[1,K]}\frac{\cosh{\pa{\sum_{i\in G_{k}}\eps Y_{i,d}}}}{\cosh{\pa{\sum_{i\in \overline{G}_{k}}\eps Y_{i,d}}}}\enspace.
\end{equation}
We denote $\phi$ the standard Gaussian density $\phi(x)=\frac{1}{\sqrt{2\pi}}e^{\frac{-x^{2}}{2}}$. We recall that $m_{k}$ is the size of $G_{k}$ (and also of $|\overline{G}_{k}|$) and $m_{kl}=|\overline{G}_{k}\cap G_{l}|$. Under the law $\P_{\rho, G}$, conditionally on $\mu_{1},\ldots \mu_{K}\sim\rho$, we have that: 
\begin{itemize}
    \item $\sum_{i\in G_{k}}Y_{i,d}\sim \mathcal{N}(m_{k}\mu_{k,d},m_{k})$,
    \item $\sum_{i\in \overline{G}_{k}}Y_{i,d}\sim \mathcal{N}(\sum_{l\in[1,K]}m_{kl}\mu_{l},m_{k})$.
\end{itemize}
Plugging these two points, together with equality \eqref{eq:vraisemblance2partial}, in the definition of the Kullback-Leibler divergence leads to
\begin{align}\nonumber
    \frac{1}{p}KL(\P_{\rho, G},\P_{\rho,\overline{G}})=&\sum_{k\in[1,K]}\E_{\rho}\cro{\int\log\cosh(\eps(m_k \mu_{k,1}+\sqrt{m_k}x))\phi(x)dx}\\
    &-\sum_{k\in[1,K]}\E_{\rho}\cro{\int\log\cosh(\eps(\sum_{l\in[1,K]} m_{kl}\mu_{l,1}+\sqrt{m_k}x))\phi(x)dx}\enspace.\label{eq:KLpartial}
\end{align}
By symmetry, it is sufficient to upper-bound the term corresponding to the first group $\overline{G}_1$ in the sum above which is equal to 
$S_1=S_{11}-S_{12}$, with 
\begin{align*}
    S_{11}&=\E_{\rho}\cro{\int\log\cosh(\eps(m_1 \mu_{1,1}+\sqrt{m_1}x))\phi(x)dx},\\
     \text{and}\quad S_{12}&=\E_{\rho}\cro{\int\log\cosh(\eps(\sum_{l\in[1,K]} m_{1l}\mu_{l,1}+\sqrt{m_1}x))\phi(x)dx}.
\end{align*}
First, we upper-bound the term $S_{11}$. We will use the following inequality 
\begin{equation}\label{eq:Taylor2}
    \frac{x^{2}}{2}-\frac{x^4}{12}\leq  \log\cosh(x)\leq \frac{x^2}{2},\enspace \forall x\in\R\enspace.
\end{equation}
\begin{proof} [Proof of inequality \eqref{eq:Taylor2}]
    Let us first prove the upper-bound $\log\cosh(x)\leq \frac{x^{2}}{2}$. For $x\in\R$, $\cosh(x)=\sum_{t\in\N}\frac{x^{2t}}{(2t)!}\leq \sum_{t\in\N}\frac{x^{2t}}{t!2^{t}}=\exp\pa{\frac{t^{2}}{2}}$. Applying the logarithmic function leads to $\log\cosh(x)\leq \frac{x^2}{2}$.

    Let us now prove the lower-bound $\log\cosh(x)\geq \frac{x^{2}}{2}-\frac{x^{4}}{12}$. To do so, we write, for $x\geq 0$, $f(x)=\tanh(x)$ and $g(x)=x-\frac{x^{3}}{3}$. We have $f'(x)=1-(f(x))^{2}$. Besides, $g'(x)=1-x^{2}\leq 1-(g(x))^{2}$. Together with the fact that $f(0)=g(0)=0$, this implies 
    $$f(x)\geq g(x),\enspace \forall x\geq 0\enspace.$$ 
    Integrating these function leads to $$\log\cosh(x)\geq \frac{x^{2}}{2}-\frac{x^{4}}{12},\enspace \forall x\geq 0\enspace.$$ By parity of these functions, this last inequality is satisfied for all $x\in\R$. 
\end{proof}
Inequality \eqref{eq:Taylor2}, together with the independence of $x$ and $\mu_{1,1}$, imply that 
\begin{equation}
    S_{11}\leq\frac{1}{2} \E_{\rho}\cro{\int(\eps(m_1 \mu_{1,1}+\sqrt{m_1}x))^{2}\phi(x)dx}\leq \frac{1}{2}\eps^{2}\pa{m_1+m_{1}^{2}\eps^{2}}\enspace.
\end{equation}
We arrive at 
\begin{equation}\label{eq:S11}
    S_{11}\leq \frac{1}{2}m_1\eps^{2}(1+m_1 \eps^{2})\enspace.
\end{equation}
Let us now lower-bound $S_{12}$. Inequality \eqref{eq:Taylor2} induces
\begin{align*}
    S_{12}\geq&\frac{1}{2}\E_{\rho}\cro{\int\eps^2\bigg(\sum_{l\in[1,K]} m_{1l}\mu_{l,1}+\sqrt{m_1}x\bigg)^{2}\phi(x)dx}\\
    &-\frac{1}{12}\E_{\rho}\cro{\int\eps^4\bigg(\sum_{l\in[1,K]} m_{1l}\mu_{l,1}+\sqrt{m_1}x\bigg)^{4}\phi(x)dx}\\
    \geq & \frac{1}{2}\eps^{2}\pa{\sum_{l\in [1,K]}m_{1l}^{2}\eps^{2}+m_{1}}
    -\frac{1}{12}\eps^{4}\E_{\rho}\left[\bigg(\sum_{l\in[1,K]}m_{1l}\mu_{l,1}\bigg)^4\right]-\frac{1}{12}\eps^{4}3m_1^{2}\\
    &-\frac{1}{2}\eps^{4}m_1\E_{\rho}[(\sum_{l\in[1,K]}m_{1l}\mu_{l,1})^2]\\
    \geq & \frac{1}{2}\eps^{2}\pa{\sum_{l\in [1,K]}m_{1l}^{2}\eps^{2}+m_{1}}
    -\frac{1}{12}\eps^{8}\pa{\sum_{l\in[1,K]}m_{1l}^{4}+6\pa{\sum_{l\in[1,K]}m_{1l}^{2}}^{2}}-\frac{1}{4}\eps^{4}m_1^{2}\\
    &-\frac{1}{2}\eps^{6}m_1\sum_{l\in[1,K]}m_{1l}^2\enspace.
\end{align*}
We end up with
\begin{equation}\label{eq:S12}
    S_{12}\geq \frac{1}{2}m_1\eps^{2}-\frac{1}{4}m_{1}^{2}\eps^{4}-\frac{1}{2}\eps^{6}m_{1}\sum_{l\in [1,K]}m_{1,l}^{2}-\frac{1}{12}\eps^{8}\pa{\sum_{l\in[1,K]}m_{1l}^{4}+6\pa{\sum_{l\in[1,K]}m_{1l}^{2}}^{2}}\enspace.
\end{equation}
Combining inequalities \eqref{eq:S11} and \eqref{eq:S12}, together with the equality $m_1=\sum_{l\in[1,K]}m_{1l}$, leads to 
\begin{align*}
    S_{1}\leq &\frac{1}{2}m_1\eps^{2}(1+m_1 \eps^{2})-\frac{1}{2}m_1\eps^{2}
    +\frac{1}{4}m_{1}^{2}\eps^{4}+\frac{1}{2}\eps^{6}m_{1}\sum_{l\in [1,K]}m_{1,l}^{2}\\
    &+\frac{1}{12}\eps^{8}\pa{\sum_{l\in[1,K]}m_{1l}^{4}+6\pa{\sum_{l\in[1,K]}m_{1l}^{2}}^{2}}\\
    \leq &\frac{3}{4}m_1^{2}\eps^{4}+\frac{7}{12}\eps^{8}m_{1}^{4}+\frac{1}{2}\eps^{6}m_{1}^{3}\\
    \leq &c_{4}\frac{n^{2}}{K^{2}}\eps^{4}\pa{1+\eps^{2}\frac{n}{K}+\eps^{4}\frac{n^{2}}{K^{2}}}\enspace,
\end{align*}
where $c_4$ is a numerical constant, obtained using the fact that $m_{1}\leq \frac{3}{2}\frac{n}{K}$. Summing over $k\in[1,K]$ in \eqref{eq:KLpartial} leads to 
$$KL(\P_{\rho, G},\P_{\rho,\overline{G}})\leq c_{4}p\frac{n^{2}}{K}\eps^{4}\pa{1+\eps^{2}\frac{n}{K}+\eps^{4}\frac{n^{2}}{K^{2}}}\enspace.$$
From the definition of $\eps$, we have $p\frac{n^{2}}{K}\eps^{4}=4\frac{n^{2}}{Kp}\Bar{\Delta}^{4}$. The hypothesis $\Bar{\Delta}^{2}\leq c_2\sqrt{\frac{p}{n}K\log(K)}$ leads to $p\frac{n}{K}\eps^{4}\leq 4c_{2}^{2}n\log(K)$. Moreover, the hypothesis $c_{1}\log(K)\leq\Bar{\Delta}^{2}\leq c_2\sqrt{\frac{p}{n}K\log(K)}$ implies $\eps^{2}\frac{n}{K}=\frac{2n}{pK}\frac{\Bar{\Delta}^{4}}{\Bar{\Delta}^{2}}\leq 2\frac{c_{2}^{2}}{c_{1}}\leq 1$, if $c_{2}$ is small enough with respect to $c_1$. Thus, there exists a numerical constant $c>0$ such that, when $c_2$ is small enough, 
$$KL(\P_{\rho, G},\P_{\rho,\overline{G}})\leq cc_{2}^2n\log(K)\enspace.$$
This concludes the proof of the lemma.

\subsection{Proof of Lemma \ref{lem:reduction2}}\label{prf:reduction2}
We suppose that there exists a probability distribution $\rho$ on $\R^{P\times K}$ and $C>0$ satisfying 
$$\inf_{\hat{G}}\sup_{G\in\mathcal{P}_{\alpha}}\E_{\rho,G}\cro{err(\hat{G},G)}-\rho(\R^{p\times K}\setminus \Theta_{\Bar{\Delta}})\geq C\enspace.$$
Given an estimator $\hat{G}$, there exists a partition $G\in\mathcal{P}_{\alpha}$ such that $\E_{\rho,G}\cro{err(\hat{G}, G)}-\rho(\R^{p\times K}\setminus \Theta_{\Bar{\Delta}})\geq C$. Since $\E_{\rho,G}\cro{err(\hat{G}, G)}=\int_{\R^{p\times K}}\E_{\mu,G}\cro{err(\hat{G}, G)}d\rho(\mu)$ and $\E_{\mu,G}\cro{err(\hat{G}, G)}$ is upper bounded by $1$, we end up with $\int_{\Theta_{\Bar{\Delta}}}\E_{\mu,G}\cro{err(\hat{G}, G)}d\rho(\mu)\geq C$.

This implies the existence of $\mu\in\Theta_{\Bar{\Delta}}$ such that $\E_{\mu,G}\cro{err(\hat{G}, G)}\geq C$. A fortiori, $$\sup_{\mu\in\Theta_{\Bar{\Delta}}}\sup_{G\in\mathcal{P}_{\alpha}}\E_{\mu,G}\cro{err(\hat{G}, G)}\geq C\enspace .$$ 

This last inequality being true for all estimator $\hat{G}$, we get the sought inequality $$\inf_{\hat{G}}\sup_{\mu\in\Theta_{\Bar{\Delta}}}\sup_{G\in\mathcal{P}_{\alpha}}\E_{\mu,G}\cro{err(\hat{G},G)}\geq C\enspace.$$

\subsection{Proof of Lemma \ref{lem:numberpartitions}}\label{prf:numberpartitions}
We recall that $\overline{G}$ is the partition of $[1,n]$ which is defined by $i\in \overline{G}_{k}$ if and only if $i\equiv k\enspace [K]$. Throughout the proof of this lemma, we denote $m=\lfloor \frac{n}{K} \rfloor$. We define $V\subset \mathcal{P}_{\frac{3}{2}}$, the set of partitions $G$ which satisfy \begin{itemize}
    \item the restriction of $G$ to $[1,K\lfloor \frac{m}{2}\rfloor]\cup [Km+1,n]$ is equal to the restriction of $\overline{G}$ on the same set,
    \item $|G_{k}|=|\overline{G}_{k}|$ for all $k\in[1,K]$.
\end{itemize}
The number of partitions in $V$ is equal to the number of partitions $[K\lfloor \frac{m}{2}\rfloor+1, Km]$ in $K$ groups of size $m-\lfloor \frac{m}{2}\rfloor$. For the $\lceil \frac{K}{2}\rceil$ first groups, we choose $m-\lfloor \frac{m}{2}\rfloor$ elements amongst at least $\lfloor\frac{K}{2}\rfloor(m-\lfloor \frac{m}{2}\rfloor)$ elements. Hence, the number of such partitions is lower-bounded by $\binom{\lfloor\frac{K}{2}\rfloor(m-\lfloor \frac{m}{2}\rfloor)}{m-\lfloor \frac{m}{2}\rfloor}^{\lceil\frac{K}{2}\rceil}$. We arrive at $\log|V|\geq \frac{K}{2}\log\binom{\lfloor\frac{K}{2}\rfloor(m-\lfloor \frac{m}{2}\rfloor)}{m-\lfloor \frac{m}{2}\rfloor}$. Besides, if the constant $K_{0}$ is large enough, we get that, when $K\geq K_0$, $\binom{\lfloor\frac{K}{2}\rfloor(m-\lfloor \frac{m}{2}\rfloor)}{m-\lfloor \frac{m}{2}\rfloor}\geq \pa{\frac{Km}{8}}^{m-\lfloor\frac{m}{2}\rfloor}\frac{1}{\pa{m-\lfloor\frac{m}{2}\rfloor}!}$. Besides, $\log\pa{m-\lfloor\frac{m}{2}\rfloor}!\leq \pa{m-\lfloor\frac{m}{2}\rfloor}\log\pa{m-\lfloor\frac{m}{2}\rfloor}$. We arrive at $\log\pa{\binom{\lfloor\frac{K}{2}\rfloor(m-\lfloor \frac{m}{2}\rfloor)}{m-\lfloor \frac{m}{2}\rfloor}}\geq \frac{m}{2}\pa{\log\pa{\frac{Km}{8}}-\log\pa{m-\lfloor\frac{m}{2}\rfloor}}$. This implies the existence of a numerical constant $c_3>0$ such that, if $K\geq K_0$ with $K_0$ large enough,
\begin{equation}\label{eq:logV}
    \log|V|\geq c_3 n\log(K)\enspace.
\end{equation}
Let $S$ be a maximal subset of $V$ satisfying; for all $G,G'\in S$, the proportion of misclassified points between these two partitions $err(G,G')$ is lower-bounded by $a$, with $a>0$ a numerical constant that we will choose small enough. As a consequence, for all $G\in V$, there exists $G'\in S$ such that $err(G,G')\leq a$. For $G\in S$, we denote $B_{G}\subset V$ the set of partitions $G'\in V$ that satisfy $err(G,G')\leq a$. Then, 
\begin{equation}\label{eq:maximalfamily}
    |S|\geq \frac{|V|}{\max_{G\in S}|B_{G}|}\enspace.
\end{equation}
Given $G\in S$, let us upper-bound $|B_{G}|$. For $G'\in V$, we define $E(G')=\{i\in [1,n],\enspace \exists k\neq l\in[1,K],\enspace i\in G_{k}\cap G'_{l}\}\subset [K\lfloor \frac{m}{2}\rfloor+1, Km]$. We recall the definition 
$$err(G,G')=\min_{\pi\in\mathcal{S}_{K}}\frac{1}{2n}\sum_{k=1}^{K}|G_{k}\Delta G'_{\pi(k)}|\enspace.$$    
For $k\in [1,K]$, if $\pi(k)\neq k$, then $|G_{k}\Delta G'_{\pi(k)}|\geq 2|\lfloor \frac{m}{2}\rfloor|\geq|G'_{k}\cap E(G')|$. If $\pi(k)= k$, then $|G_{k}\Delta G'_{\pi\pa{k}}|\geq |G'_{k}\cap E(G')|$. Plugging this in the definition of $err(G,G')$ leads to 
$$err(G,G')\geq \frac{1}{2n}E(G')\enspace.$$
Hence, if $G'\in B_{G}$, then $E(G')\leq 2an$. For choosing a partition $G'$ such that $E(G')\leq 2an$, we choose $2an$ points amongst $[K\lfloor \frac{m}{2}\rfloor+1,Km]$ points and, for each of these points, we choose the group $G'_{k}$ which will contain it. We arrive at $|\{G'\in V,\enspace E(G')\leq 2an\}|\leq \binom{K(m-\lfloor \frac{m}{2}\rfloor)}{2an}K^{2an}$. We end up with $\log|B_{G}|\leq 2an\log(K)+\log\binom{K(m-\lfloor \frac{m}{2}\rfloor)}{2an}\leq 2an\log(K)+2an\log\pa{\frac{1}{2a}}$.
Combining this last inequality with inequalities \eqref{eq:logV} and \eqref{eq:maximalfamily} leads to 
\begin{equation*}
    \log|S|\geq c_3 n\log(K)-2an\log(K)-2an\log(\frac{1}{2a})\enspace. 
\end{equation*}
If the constant $a$ is small enough and the constant $K_{0}$ large enough, we have, when $K\geq K_{0}$, and a fortiori $n\geq 2K_{0}$, $\frac{c_3}{2}n\log(K)-2an\log(\frac{1}{2a})\geq \frac{c_3 n}{4}\log(K)$. 

Besides, if $a$ is small enough, we also have $\frac{c_3}{2}n\log(K)-2an\log(K)\geq \frac{c_3 n}{4}\log(K)$. This leads to the sought inequality
$$\log|S|\geq \frac{c_{3} n}{2}\log(K)\enspace.$$
For $G\in S$ and $k\in[1,K]$, we have $|G_{k}|=|\overline{G}_{k}|\in [\lfloor \frac{n}{K}\rfloor, \lfloor \frac{n}{K}\rfloor+1]$. Hence, $S\subset \mathcal{P}_{\frac{3}{2}}$. 
By definition of $S$, for $G\neq G'\in S$, we have the lower-bound $err(G,G')\geq a$.
Hence, the set $S$ satisifies all the conditions of Lemma \ref{lem:numberpartitions}. This concludes the proof of the lemma.

\subsection{Proof of Lemma \ref{lem:choicepoints}}\label{prf:choicepoints}
We suppose that $\frac{p}{4}\log(2)\geq \log(K)$ and we want to find vectors $\mu_{1},\ldots ,\mu_{K}$ in $\R^{p}$ such that $\frac{1}{2}\min_{l\neq r}\|\mu_{r}-\mu_{l}\|^{2}\geq \Bar{\Delta}^2$ and $\frac{1}{2}\max_{l\neq r}\|\mu_{l}-\mu_{r}\|^{2}\leq 4\Bar{\Delta}^{2}.$

Denote $\mathcal{H}=\Bar{\Delta} \sqrt{\frac{2}{p}}\{-1,1\}^{p}$. There are $2^{p}$ elements in $\mathcal{H}$. Consider $H$ a maximal subset of $\mathcal{H}$ such that the following holds. For $\theta$ and $\theta'$ two distinct elements of $H$, there exists at least $\frac{p}{4}$ index such that $\theta_{i}\neq \theta'_{i}$.

If $\theta$ and $\theta'$ are two distinct elements of $H$, we have from the definitions of $\mathcal{H}$ and $H$ that $$\Bar{\Delta}^2\leq \frac{1}{2}\|\theta-\theta'\|^{2}\leq 4\Bar{\Delta}^2\enspace.$$
It remains to lower-bound the cardinality of $H$. Given $\theta\in H$, denote $B_{\theta}$ the subset of $\mathcal{H}$ made of the elements $\theta'$ that have at least $\lfloor\frac{3p}{4}\rfloor$ index such that $\theta_{i}=\theta'_{i}$. From the definition of $H$, we deduce that $\mathcal{H}\subset \bigcup_{\theta\in H}B_{\theta}$. Since the $B_{\theta}$'s are all of the same cardinality, we get that, for $\theta\in H$, $|\mathcal{H}|\leq |H||B_{\theta}|$.  

Let us upper-bound $|B_{\theta}|$. Given a fixed set of index of cardinal $\lfloor\frac{3p}{4}\rfloor$, there are at most $2^{\frac{p}{4}}$ points in $\mathcal{H}$ that are equal to $\theta$ on these index. Hence, the cardinal of $B_{\theta}$ is upper-bounded by $\binom{p}{\lfloor \frac{3p}{4}\rfloor}2^{\frac{p}{4}}$. 
Plugging this inequality in $|\mathcal{H}|\leq |H||B_{\theta}|$ leads to $\log(|H|)\geq p\log(2)-\log\binom{p}{\lfloor \frac{3p}{4}\rfloor}-\frac{p}{4}\log(2)$. Using the inequality $\log\binom{p}{\lfloor \frac{3p}{4}\rfloor}\leq \log\binom{p}{\lfloor \frac{p}{2}\rfloor}\leq \frac{p}{2}\log(2)$, we end up with $\log(|H|)\geq \frac{p}{4}\log(2)$. This implies from our hypothesis on $p$ that $|H|\geq K$. Any $K$-tuple of vectors $\mu_{1},\ldots ,\mu_{K}$ taken from $|H|$ satisfies the sought conditions. This concludes the proof of the lemma.

\subsection{Proof of Lemma \ref{lem:lowerboundseparation}}\label{prf:lowerboundseparation}
Let $p>1$ and $K\geq 2$. We lower bound in this section $\min_{k\neq l}\|\mu_{k}-\mu_{l}\|$ when $\mu\sim \rho$. Let $k\neq l\in [1,K]$. Then, $\|\mu_{k}-\mu_{l}\|^{2}=4\eps^{2}\sum_{d=1}^{p}\1_{\mu_{k,d}\neq \mu_{l,d}}$. Hence, $\frac{1}{4\eps^{2}}\|\mu_{k}-\mu_{l}\|^{2}$ is a sum of $p$ independent Bernoulli random variables of parameter $\frac{1}{2}$. In particular, $\frac{1}{4\eps^{2}}\|\mu_{k}-\mu_{l}\|^{2}$ is the sum of $p$ independent random variables of mean $\frac{1}{2}$ and bounded in absolute value by $1$. Using Hoeffding's inequality leads to 
\begin{align*}
    \P\left[\frac{1}{4\eps^{2}}\|\mu_{k}-\mu_{l}\|^{2}\leq \frac{p}{4}\right]&=\P\left[\frac{1}{4\eps^{2}}\|\mu_{k}-\mu_{l}\|^{2}-\E\left[\frac{1}{4\eps^{2}}\|\mu_{k}-\mu_{l}\|^{2}\right]\leq -\frac{p}{4}\right]\\
    &\leq \exp\pa{\frac{-2\pa{\frac{p}{4}}^{2}}{p}}\\
    &\leq \exp\pa{-\frac{p}{8}}\enspace.
\end{align*}
Besides, $\frac{1}{4\eps^{2}}\|\mu_{k}-\mu_{l}\|^{2}=\frac{p}{8\Bar{\Delta}^{2}}\|\mu_{k}-\mu_{l}\|^{2}$. Hence, 
$$\P\left[\frac{1}{2}\|\mu_{k}-\mu_{l}\|^{2}\leq \Bar{\Delta}^{2}\right]\leq \exp\pa{-\frac{p}{8}}\enspace.$$
Using an union bound on the set of pairs $k\neq l\in[1,K]$ leads to 
$$\P[\exists k\neq l\in[1,K],\enspace \frac{1}{2}\|\mu_{k}-\mu_{l}\|^{2}\leq \Bar{\Delta}^{2}]\leq \frac{K(K-1)}{2}\exp\pa{-\frac{p}{8}}\enspace.$$
This concludes the proof of the lemma.

\section{Proof of Theorem \ref{thm:lowerboundexact}}\label{prf:lowerboundexact}
Without loss of generality, we suppose throughout this proof that $\sigma=1$. We suppose that $n\geq n_{0}$, with $n_{0}$ a constant that we will choose large enough. We also suppose that $\alpha\geq \frac{3}{2}$ and $n\geq 9K/2$.

As in Section \ref{prf:lowerboundpartial},  we will distinguish two cases. 
In a first time, we will prove that there exists numerical constants $c_{1}$ and $C$ such that when $\Bar{\Delta}^{2}\leq c_{1}\log(n)$, $$\inf_{\hat{G}}\sup_{\mu\in\Theta_{\Bar{\Delta}}}\sup_{G\in\mathcal{P}_{\alpha}}\P_{\mu,G}(\hat{G}\neq G)>C\enspace.$$
Then, in a second case, we will show that there exists a numerical constant $c_{2}$, such that this also holds when $c_{1}\log(n)\leq \Bar{\Delta}^{2}\leq c_{2}\sqrt{\frac{pK}{n}\log(n)}.$

In the following of this proof, we consider a partition $G^{(0)}\in \mathcal{P}_{\frac{3}{2}}$. In particular, $G^{(0)}\in \mathcal{P}_{\alpha}$. The existence of such a partition is ensured by the hypothesis $n\geq 2K$ (for example, we can take the partition $\overline{G}$ defined in the proof of Theorem \ref{thm:lowerboundpartial} by $i\in \overline{G}_{k}$ if and only if $i\equiv k\enspace [K]$).

\subsection{$\Bar{\Delta}^{2}\leq c_{1}\log(n)$.}
Let us suppose that $\Bar{\Delta}^{2}\leq c_{1}\log(n)$ for $c_{1}$ a positive numerical constant, that we will choose small enough later.

Let $e$ be a unit vector of $\R^{p}$ and define $\mu_{k}=\sqrt{2}k\Bar{\Delta} e$ for $k\in[1,K]$. It is clear that $\mu=\mu_1,\ldots ,\mu_{K}\in\Theta_{\Bar{\Delta}}$.
We will prove  our statement  using Fano's Lemma that is recalled page \pageref{lem:fano}. 
Our goal will be to find different partitions $G^{(1)},...,G^{(M)}\in\mathcal{P}_{\alpha}$, with $M$ as large as possible, such that $KL(\P_{\mu, G^{(r)}},\P_{\mu,G^{(0)}})$ remains small for all $r\in[1,M]$.

Given $k\in[1,K-1]$, $i\in G^{(0)}_{k}$ and $j\in G^{(0)}_{k+1}$, we define the partition $G^{(i,j)}$ as follows. For $l\in[1,K]$ distinct both from $k$ and $k+1$, we take $G^{(i,j)}_{l}=G^{(0)}_{l'}$. Besides, we take $G^{(i,j)}_{k}=\pa{G^{(0)}_{k}\setminus \{i\}}\cup \{j\}$ and $G^{(0)}_{k+1}=\pa{G^{(0)}_{k+1}\setminus \{i\}}\cup \{j\}$. This partition corresponds to the partition $G^{(0)}$ after shifting the points $i$ and $j$. We denote $Sh(G^{(0)})$ the set of all these partitions 
\begin{equation}\label{eq:setpartitions}
    Sh(G^{(0)})=\{G^{(i,j)},\enspace i\in G^{(0)}_{k}, j\in G^{(0)}_{k+1}, \enspace k\in[1,K-1] \}\enspace.
\end{equation}
For $G^{(i,j)}\in Sh(G^{(0)})$, the groups of $G^{(i,j)}$ are of the same size than the groups of $G^{(0)}$. Thus, $G^{(i,j)}\in\mathcal{P}_{\alpha}$. And, all the groups of $G^{(0)}$ are of size at least $3$. This implies that the $G^{(i,j)}$'s are all distinct. 

For $G^{(i,j)}\in Sh(G^{(0)})$, let us compute $KL(\P_{\mu, G^{(i,j)}}, \P_{\mu,G^{(0)}})$. Given $i'\in[1,n]$, we denote $\P_{\mu, G^{(i,j)}}(Y_{i'})$ the marginal law of the vector $Y_{i'}$ under the joint law $\P_{\mu, G^{(i,j)}}$. Using the independence of the $Y_{i'}$'s for $i'\in[1,n]$, we get
\begin{align*}
    KL(\P_{\mu, G^{(i,j)}}, \P_{\mu,G})=&\sum_{i'=1}^{n}KL(\P_{\mu, G^{(i,j)}}(Y_{i'}), \P_{\mu,G^{(0)}}(Y_{i'}))\\
    =& KL\pa{\P_{\mu, G^{(i,j)}}(Y_{i}),\P_{\mu, G^{(0)}}(Y_{i})}+KL\pa{\P_{\mu, G^{(i,j)}}(Y_{j}),\P_{\mu, G^{(0)}}(Y_{j})}\\
    =&KL\pa{\mathcal{N}(\mu_{k+1},I_p),\mathcal{N}(\mu_{k},I_{p})}+KL\pa{\mathcal{N}(\mu_{k},I_p),\mathcal{N}(\mu_{k+1},I_{p})}\\
    =& 2\Bar{\Delta}^2\ 
    \leq\ 2c_{1}^{2} \log(n)\enspace ,
\end{align*}
where we used the property $KL\pa{\mathcal{N}(f,I_p),\mathcal{N}(g,I_{p})}=\frac{\|f-g\|^{2}}{2}$, for all $f,g\in\R^{p}$.

For any estimator $\hat{G}$, we associate $\hat{J}$ the estimator that gives $(i,j)$ if $\hat{G}=G^{(i,j)}$ and $(1,2)$ elsewhere (we can suppose that $G^{(1,2)}\in Sh(G^{(0)})$).  Lemma \ref{lem:fano} implies that, for all estimator $\hat{G}$, the corresponding estimator $\hat{J}$ satisfies 
$$\frac{1}{|Sh(G^{(0)})|}\sum_{G^{(i,j)}\in Sh(G^{(0)})}\P_{\mu,G^{(i,j)}}(\hat{J}\neq (i,j))\geq 1-\frac{1+2c_{1}^{2} \log(n)}{\log(|Sh(G^{(0)})|)}\enspace.$$ 
For all $i\in G_k$, if $j\in G_{k+1}$, then $G^{(i,j)}\in Sh(G^{(0)})$. The groups of $G^{(0)}$ are of size at least $\frac{2n}{3K}$. Hence, since $K\geq 2$ and $n\geq 9K/2$, 
\begin{equation}\label{eq:sizeSh}
    |Sh(G^{(0)})|\geq (K-1)\pa{\frac{2n}{3K}}^{2}\geq \frac{4}{18}\frac{n^{2}}{K}\geq n\enspace.
\end{equation}
We arrive at 
$$\frac{1}{|Sh(G^{(0)})|}\sum_{G^{(i,j)}\in Sh(G^{(0)})}\P_{\mu,G^{(i,j)}}(\hat{J}\neq (i,j))\geq 1-\frac{1+2c_{1}^{2} \log(n)}{\log(n)}\enspace.$$
If $c_{1}$ is small enough and $n_{0}$ large enough, there exists a constant $C>0$ such that, for all integers $n\geq n_{0}$, we have $1-\frac{1+2c_{1}^{2} \log(n)}{\log(n)}\geq C.$ This implies
$$\frac{1}{|Sh(G^{(0)})|}\sum_{G^{(i,j)}\in Sh(G^{(0)})}\P_{\mu,G^{(i,j)}}(\hat{J}\neq (i,j))\geq C\enspace.$$
For any estimator $\hat{G}$ and its corresponding estimator $\hat{J}$, for any $G^{(i,j)}\in Sh(G^{(0)})$, we have $\P_{\mu,G^{(i,j)}}(\hat{G}=G^{(i,j)})\leq  \P_{\mu,G^{(i,j)}}(\hat{J}=(i,j))$. Thus, we arrive at 
$$\frac{1}{|Sh(G^{(0)})|}\sum_{G^{(i,j)}\in Sh(G^{(0)})}\P_{\mu,G^{(i,j)}}(\hat{G}\neq G^{(i,j)})\geq C\enspace .$$
This, with the fact that, for all estimator $\hat{G}$, $$\sup_{\mu\in\theta_{\Bar{\Delta}}}\sup_{G\in\mathcal{P}_{\alpha}}\P_{\mu,G}(\hat{G}\neq G)\geq \frac{1}{|Sh(G^{(0)})|}\sum_{G^{(i,j)}\in Sh(G^{(0)})}\P_{\mu,G^{(i,j)}}(\hat{G}\neq G^{(i,j)})\geq C\enspace.$$ This concludes the proof of the theorem in the regime where $\Bar{\Delta}^{2}\geq c_{1}\log(n)$.

\subsection{$c_{1}\log(n)\leq \Bar{\Delta}^{2}\leq c_{2}\sqrt{\frac{p}{n}K\log(n)}$.}

We suppose that $c_{1}\log(n)\leq \Bar{\Delta}^{2}\leq c_{2}\sqrt{\frac{p}{n}K\log(n)}$, with $c_{1}$ the numerical constant chosen just above and $c_{2}$ another numerical constant that we will choose small enough. Given $\rho$ a probability distribution on $(\R^{p})^{K}$ and a partition $G$ of $[1,n]$ in $K$ groups, we define the probability distribution on $(\R^{p})^{n}$ by $$\P_{\rho,G}(B)=\int \P_{\mu,G}(B)d\rho(\mu)\enspace.$$ The proof of Theorem \ref{thm:lowerboundexact} in this regime uses the following lemma. It is a reduction lemma which plays the same role as Lemma \ref{lem:reduction2} in the proof of Theorem \ref{thm:lowerboundpartial}.

\begin{lem}\label{lem:reduction}
    We suppose that there exists a probability distribution $\rho$ on $(\R^{p})^{K}$ and $a>0$ such that $$\inf_{\hat{G}}\sup_{G\in\mathcal{P}_{\alpha}}\P_{\rho,G}(\hat{G}\neq G)-\rho((\R^{p})^{K}\setminus\Theta_{\Bar{\Delta}})>a\enspace.$$
    Then, we have $$\inf_{\hat{G}}\sup_{\mu\in\Theta_{\Bar{\Delta}}}\sup_{G\in\mathcal{P}_{\alpha}}\P_{\mu,G}(\hat{G}\neq G)>a\enspace.$$
\end{lem}
We refer to Section \ref{prf:reduction} for a proof of this lemma. We will consider the same distribution on $(\R^{p})^{K}$ as in Section \ref{prf:lowerboundpartial}. We take $\eps=\sqrt{\frac{2}{p}}\Bar{\Delta}$ and $\rho$ the uniform distribution on the hypercube $\mathcal{E}=\{-\eps,\eps\}^{p\times K}$. 

We will use Fano's Lemma to lower bound $\inf_{\hat{G}}\sup_{G\in\mathcal{P}_{\alpha}}\P_{\rho,G}(\hat{G}\neq G)$. To do so, we need to find many partitions $G^{(1)},\ldots,G^{(M)}\in\mathcal{P}_{\alpha}$, with $M$ large, such that $KL(\P_{\rho,G^{(l)}},\P_{\rho,G^{(0)}})$ remains small. Again,  we use the set of partitions $Sh(G^{(0)})$, defined by \eqref{eq:setpartitions},  for $G^{(0)}\in\mathcal{P}_{\frac{3}{2}}$. 
For such a partition $G\in Sh(G^{(0)})$, the next lemma controls the quantity $$KL(\P_{\rho,G},\P_{\rho,G^{(0)}})=\int\log (\frac{d\P_{\rho,G}}{d\P_{\rho,G^{(0)}}})d\P_{\rho,G}\enspace.$$ We refer to Section \ref{prf:ineqKLexact} for a proof of this lemma. 

\begin{lem}\label{lem:ineqKLexact}
    We suppose that $c_2$ is small enough with respect to $c_1$. Then, there exists a numerical constant $c>0$ such that, for all $G\in Sh(G^{(0)})$, we have the inequality $$KL(\P_{\rho,G},\P_{\rho,G^{(0)}})\leq cc_{2}^{2}\log(n)\enspace.$$
\end{lem}
Together with Lemma \ref{lem:fano} applied to the set $Sh(G^{(0)})$, this lemma induces
$$\inf_{\hat{G}}\frac{1}{|Sh(G^{(0)})|}\sum_{G\in Sh(G^{(0)})}\P_{\rho,G^{(0)}}[\hat{G}\neq G^{(0)}]\geq 1-\frac{1+cc_{2}^{2}\log(n)}{\log(|Sh(G^{(0)})|)}\enspace.$$
Since for all estimator $\hat{G}$, the inequality 
$$\frac{1}{|Sh(G^{(0)})|}\sum_{G\in Sh(G^{(0)})}\P_{\rho,G}[\hat{G}\neq G]\leq \sup_{G\in\mathcal{P}_{\alpha}}\P_{\rho,G}[\hat{G}\neq G]$$ is satisfied, we have the following inequality 
$$\inf_{\hat{G}}\sup_{G\in\mathcal{P}_{\alpha}}\P_{\rho,G}[\hat{G}\neq G]\geq 1-\frac{1+cc_{2}^{2}\log(n)}{\log(|Sh(G^{(0)})|)}\enspace.$$
Finally, referring to equation \eqref{eq:sizeSh}, we have $\log(|Sh(G^{(0)})|)\geq \log\pa{n}$. Thus, if we choose $c_{2}$ small enough and if $n_{0}$ large enough, there exists a numerical constant $b>0$ satisfying 
$$\inf_{\hat{G}}\sup_{G\in\mathcal{P}_{\alpha}}\P_{\rho,G}[\hat{G}\neq G]\geq b\enspace.$$
In order to apply Lemma \ref{lem:reduction}, it remains to control $\rho(\R^{p\times K}\setminus \Theta_{\Bar{\Delta}})$. Lemma \ref{lem:lowerboundseparation} states that $\rho(\R^{p\times K}\setminus \Theta_{\Bar{\Delta}})\leq K(K-1)\exp({-p}/{8})$. Moreover, since $c_1 \log(n)\leq c_{2}\sqrt{\frac{p}{n}K\log(n)}$, we have $p\geq \frac{c_{1}^{2}}{c_{2}^{2}}\frac{n}{K}\log(n)$. This implies that $\rho(\R^{p}\setminus \Theta_{\Bar{\Delta}})\leq \frac{b}{2}$, provided $c_{2}$ is small enough with respect to $c_{1}$. Combining this inequality with Lemma \ref{lem:reduction} leads to 
$$\inf_{\hat{G}}\sup_{G\in\mathcal{P}_{\alpha}}\sup_{\mu\in\Theta_{\Bar{\Delta}}}\P_{\rho,\mu}[\hat{G}\neq G]\geq \frac{b}{2}\enspace.$$
This concludes the proof of Theorem \ref{thm:lowerboundexact}.

\subsection{Proof of Lemma \ref{lem:ineqKLexact}}\label{prf:ineqKLexact}

By symmetry, we can suppose that $1\in G^{(0)}_{1}$ and $2\in G^{(0)}_{2}$ and we compute $KL(\P_{\rho,G},\P_{\rho,G^{(0)}})$ for $G=G^{(1,2)}$. We proceed similarly as for the proof of Lemma \ref{lem:ineqKLpartial}, except that here we will use an upper-bound and a lower-bound of increments of the function $\log\cosh$, instead of bounding the function $\log\cosh$ itself. That will allow us to have a sharper bound on $KL(\P_{\rho,G^{(1,2)}},\P_{\rho,G^{(0)}})$. 

In the following, in order to ease the computations, we denote by $\rho'$ the probability distribution on $(\R)^{p\times K}$ that satisfies; if $(\mu_{1},\ldots,\mu_{K})\sim \rho'$, all the $\mu_{i}$'s are independent, $\mu_{1}=\mu_{2}=0$ and all the other $\mu_{i}$'s are drawn uniformly on the set $\{-\eps,+\eps\}^{p}$. We recall that $\eps^{2}=\frac{2}{p}\Bar{\Delta}$ and that $c_1 \log(n)\leq \Bar{\Delta}^2\leq c_2 \sqrt{\frac{p}{n}K\log(n)}$, with $c_2$ that we will choose small enough.

First, we compute the quantity 
\begin{equation}\label{eq:vraisemblance}
    \frac{d\P_{\rho,G^{(1,2)}}}{d\P_{\rho,G^{(0)}}}=\frac{\frac{d\P_{\rho,G^{(1,2)}}}{d\P_{\rho',G^{(1,2)}}}}{\frac{d\P_{\rho,G^{(0)}}}{d\P_{\rho',G^{(1,2)}}}}=\frac{\frac{d\P_{\rho,G^{(1,2)}}}{d\P_{\rho',G^{(1,2)}}}}{\frac{d\P_{\rho,G^{(0)}}}{d\P_{\rho',G^{(0)}}}}\enspace,
\end{equation} 
where the second equality comes from the fact that $\P_{\rho',G^{(0)}}=\P_{\rho',G^{(1,2)}}$. Given a probability distribution  $\P$ on some Euclidean space, which is absolutely continuous with respect to the Lebesgue measure, we write $d\P$ for the density of this distribution with respect to the Lebesgue measure.  For the numerator in \eqref{eq:vraisemblance}, we have
\begin{align*}
    \frac{d\P_{\rho,G^{(1,2)}}}{d\P_{\rho',G^{(1,2)}}}(Y)=&\frac{\E_{\rho}\cro{d\P_{\mu,G^{(1,2)}}(Y)}}{\E_{\rho'}\cro{d\P_{\mu,G^{(1,2)}}(Y)}}\\
    =&\frac{\E_{\rho}\cro{\prod_{k\in[1,K]}\prod_{i\in G^{(1,2)}_{k}}\exp\pa{-\frac{1}{2}\|Y_{i}-\mu_{k}\|^{2}}}}{\E_{\rho'}\cro{\prod_{k\in[1,K]}\prod_{i\in G^{(1,2)}_{k}}\exp\pa{-\frac{1}{2}\|Y_{i}-\mu_{k}\|^{2}}}}\\
    =&\frac{\E_{\rho}\cro{\prod_{d\in[1,p]}\prod_{k\in[1,K]}\prod_{i\in G^{(1,2)}_{k}}\exp\pa{-\frac{1}{2}\pa{Y_{i,d}-\mu_{k,d}}^{2}}}}{\E_{\rho'}\cro{\prod_{d\in[1,p]}\prod_{k\in[1,K]}\prod_{i\in G^{(1,2)}_{k}}\exp\pa{-\frac{1}{2}\pa{Y_{i,d}-\mu_{k,d}}^{2}}}}\enspace.
\end{align*}
Using the independence of the $\mu_{k,d}$'s both for the law $\rho$ and $\rho'$ together with the fact that, when $k>3$, $\mu_{k}$ has the same distribution under $\rho$ and $\rho'$, we get that 
\begin{align*}
    \frac{d\P_{\rho,G^{(1,2)}}}{d\P_{\rho',G^{(1,2)}}}(Y)=&\prod_{d\in[1,p]}\prod_{k\in[1,K]}\frac{\E_{\rho}\cro{\prod_{i\in G^{(1,2)}_{k}}\exp\pa{-\frac{1}{2}(Y_{i,d}-\mu_{k,d})^{2}}}}{\E_{\rho'}\cro{\prod_{i\in G^{(1,2)}_{k}}\exp\pa{-\frac{1}{2}(Y_{i,d}-\mu_{k,d})^{2}}}}\\
    =&\prod_{d\in[1,p]}\prod_{k\in\{1,2\}}\frac{\E_{\rho}\cro{\prod_{i\in G^{(1,2)}_{k}}\exp\pa{-\frac{1}{2}(Y_{i,d}-\mu_{k,d})^{2}}}}{\E_{\rho'}\cro{\prod_{i\in G^{(1,2)}_{k}}\exp\pa{-\frac{1}{2}(Y_{i,d}-\mu_{k,d})^{2}}}}\\
    =&\prod_{d\in[1,p]}\prod_{k\in\{1,2\}}\E_{\rho}\cro{\prod_{i\in G^{(1,2)}_{k}}\exp{\pa{-\frac{1}{2}\pa{(Y_{i,d}-\mu_{k,d})^{2}-(Y_{i,d})^{2}}}}}\\
    =&\prod_{d\in[1,p]}\prod_{k\in\{1,2\}}\E_{\rho}\cro{\prod_{i\in G^{(1,2)}_{k}}\exp{\pa{Y_{i,d}\mu_{k,d}-\frac{\eps^{2}}{2}}}}\\
    =&\prod_{d\in[1,p]}e^{\frac{-|G^{(1,2)}_{1}|\eps^{2}}{2}}\cosh{\pa{\sum_{i\in G^{(1,2)}_{1}}\eps Y_{i,d}}}e^{\frac{-|G^{(1,2)}_{2}|\eps^{2}}{2}}\cosh{\pa{\sum_{i\in G^{(1,2)}_{2}}\eps Y_{i,d}}},
\end{align*}
where the third equality comes from the fact that, for all $d\in[1,p]$ and when $i\leq2$, $\mu_{i,d}=0$ almost surely under the law $\rho'$. Similarly, we get that
\begin{align*}
    \frac{d\P_{\rho,G^{(0)}}}{d\P_{\rho',G^{(0)}}}(Y)=&\prod_{d\in[1,p]}e^{\frac{-|G^{(0)}_{1}|\eps^{2}}{2}}\cosh{\pa{\sum_{i\in G^{(0)}_{1}}\eps Y_{i,d}}}e^{\frac{-|G^{(0)}_{2}|\eps^{2}}{2}}\cosh{\pa{\sum_{i\in G^{(0)}_{2}}\eps Y_{i,d}}}.
\end{align*}
Combining these two equalities in \eqref{eq:vraisemblance}, and using the fact that the groups of $G^{(0)}$ and $G^{(1,2)}$ are of the same size, we get
\begin{equation}\label{eq:vraisemblance2}
    \frac{d\P_{\rho,G^{(1,2)}}}{d\P_{\rho,G^{(0)}}}(Y)=\prod_{d=1}^{p}\frac{\cosh{\pa{\sum_{i\in G^{(1,2)}_{1}}\eps Y_{i,d}}}\cosh{\pa{\sum_{i\in G^{(1,2)}_{2}}\eps Y_{i,d}}}}{\cosh{\pa{\sum_{i\in G^{(1,2)}_{1}}\eps Y_{i,d}}}\cosh{\pa{\sum_{i\in G^{(0)}_{2}}\eps Y_{i,d}}}}\enspace.
\end{equation}
Plugging equality \eqref{eq:vraisemblance2} in the definition of the Kullback-Leibler divergence leads to
\begin{align*}
    KL(\P_{\rho, G^{(1,2)}}, \P_{\rho,G^{(0)}})=&\E_{\rho,G^{(1,2)}}\cro{\log\pa{\frac{d\P_{\rho,G^{(1,2)}}}{d\P_{\rho,G^{(0)}}}}}\\
    =&\sum_{d\in[1,p]}\E_{\rho,G^{(1,2)}}\cro{\log\cosh\pa{\eps \sum_{i\in G_1^{(1,2)}}Y_{i,d}}+\log\cosh\pa{\eps \sum_{i\in G^{(1,2)}_{2}}Y_{i,d}}}\\
    &-\sum_{d\in[1,p]}\E_{\rho,G^{(1,2)}}\cro{\log\cosh\pa{\eps \sum_{i\in G^{(0)}_{1}}Y_{i,d}}+\log\cosh\pa{\eps \sum_{i\in G^{(0)}_{2}}Y_{i,d}}}.
\end{align*}
We recall that $\phi$ is the standard Gaussian density $\phi(x)=\frac{1}{\sqrt{2\pi}}e^{\frac{-x^{2}}{2}}$. We denote by $m_{1}$ the size of $G^{(0)}_{1}$ and $m_{2}$ the size of $G^{(0)}_{2}$. Under the law $\P_{\rho, G^{(1,2)}}$, conditionally on $\mu_{1},\ldots \mu_{K}\sim\rho$, we have  
\begin{itemize}
    \item $\sum_{i\in G^{(1,2)}_{1}}Y_{i,d}\sim \mathcal{N}(m_{1}\mu_{1,d},m_{1})$,
    \item $\sum_{i\in G^{(1,2)}_{2}}Y_{i,d}\sim \mathcal{N}(m_{2}\mu_{2,d},m_{2})$,
    \item $\sum_{i\in G^{(0)}_{1}}Y_{i,d}\sim \mathcal{N}((m_{1}-1)\mu_{1,d}+\mu_{2,d},m_{1})$,
    \item $\sum_{i\in G^{(0)}_{2}}Y_{i,d}\sim \mathcal{N}((m_{2}-1)\mu_{2,d}+\mu_{1,d},m_{2})$.
\end{itemize}
These four points, together with the fact that the $\mu_{k,d}$'s are identically distributed, lead to
\begin{align}\nonumber
    KL(\P_{\rho, G^{(1,2)}}, \P_{\rho,G^{(0)}})=& p \E_{\rho}\cro{\int \log\cosh(\eps(m_{1}\mu_{1,1}+\sqrt{m_{1}}x))\phi(x)dx}\\\nonumber
    &+ p \E_{\rho}\cro{\int \log\cosh(\eps(m_{2}\mu_{2,1}+\sqrt{m_{2}}x))\phi(x)dx}\\\nonumber
    &-p \E_{\rho}\cro{\int \log\cosh(\eps((m_{1}-1) \mu_{1,1}+\mu_{2,1}+\sqrt{m_{1}}x))\phi(x)dx}\\
    &-p \E_{\rho}\cro{\int \log\cosh(\eps((m_{2}-1) \mu_{2,1}+\mu_{1,1}+\sqrt{m_{2}}x))\phi(x)dx}\enspace.\label{eq:bigKL}
\end{align}
First, let us upper-bound the term $\E_{\rho}\cro{\int \log\cosh(\eps(m_{1}\mu_{1,1}+\sqrt{m_{1}}x))\phi(x)dx}\enspace.$
We denote $u=\eps ((m_{1}-1) \mu_{1,1}+\sqrt{m_{1}}x)$ and $h=\eps \mu_{1,1}$. Then, we have 
$$E_{\rho}\cro{\int \log\cosh(\eps(m_{1}\mu_{1,1}+\sqrt{m_{1}}x))\phi(x)dx}=\E_{\rho}\cro{\int \log\cosh(u+h)\phi(x)dx}\enspace.$$
We will use the Taylor expansion of the function $\log\cosh$ around $u$. We compute the following derivatives: 
\begin{itemize}
    \item For all $x\in \R$, $\log\cosh'(x)=\tanh(x)$,
    \item For all $x\in\R$, $\log\cosh ''(x)=1-\tanh^{2}(x)$ which is bounded by $2$ in absolute value.
\end{itemize}
Hence, Taylor-Lagrange inequality implies 
\begin{equation}\label{eq:Taylor}
    |\log\cosh(x+y)- \log\cosh(x)- \tanh(x)y|\leq y^{2}, \enspace \forall (x,y)\in\R^{2}\enspace.
\end{equation}
Plugging this inequality leads to 
\begin{equation}\label{eq:Taylor_1}
    \E_{\rho}\cro{\int \log\cosh(u+h)\phi(x)dx}\leq \E_{\rho}\cro{\log\cosh(u)}+\E_{\rho}\cro{\int\tanh{(u)}h\phi(x)dx}+ \E_{\rho}(h^{2})\enspace.
\end{equation}
First, since $h^{2}=\eps^{4}$ almost surely, we have $\E_{\rho}(h^{2})=\eps^{4}$. Now, we need to upper bound $\E_{\rho}\cro{\int\tanh{(u)}h\phi(x)dx}$. For any $y\in\R$, we have $\tanh'(y)=1-\tanh^{2}(y)$ and $\tanh''(y)=-2\tanh(y)(1-\tanh(y)^{2})$, which is bounded by $4$ in absolute value. Hence, Taylor-Lagrange inequality taken at $0$ leads to $$ |\tanh(y)-y|\leq 2y^{2}\enspace, \forall y\in\R\enspace.$$
This leads to $\E_{\rho}\cro{\int \tanh{(u)}h\phi(x)dx}\leq \E_{\rho}\cro{\int uh\phi(x)dx}+2\E_{\rho}\cro{\int u^{2}|h|\phi(x)dx}$. On the one hand, 
\begin{align*}
    \E_{\rho}\cro{\int uh\phi(x)dx}=&\eps^{2}\E_{\rho}\cro{\int((m_{1}-1)\mu_{1,1}+\sqrt{m_1}x)\mu_{1,1}\phi(x)dx}\\
    =&\eps^{2}\E_{\rho}\cro{(m_1-1)\mu_{1,1}^{2}}\\
    =&(m_1-1)\eps^{4}\enspace.    
\end{align*}
On the other hand, 
\begin{align*}
    \E_{\rho}\cro{\int u^{2}|h|\phi(x)dx}=&\eps^{4}\E_{\rho}\cro{\int((m_1-1)\mu_{1,1}+\sqrt{m_1}x)^{2}\phi(x)dx}\\
    =&\eps^{4}\pa{(m_1-1)^{2}\eps^{2}+m_1}\enspace.
\end{align*}
Plugging these inequalities in \eqref{eq:Taylor_1} leads to 
\begin{equation}\label{eq:KL_2}
    \E_{\rho}\cro{\int \log\cosh(u+h)\phi(x)dx}\leq \E_{\rho}\cro{\log\cosh(u)}+(m_1-1)\eps^{4}+2\eps^{4}\pa{(m_1-1)^{2}\eps^{2}+m_1}+ \eps^{4}.
\end{equation}
Now, let us lower-bound the term 
$$\E_{\rho}\cro{\int \log\cosh(\eps((m_1-1)\mu_{1,1}+\mu_{2,1}+\sqrt{m_1}x))\phi(x)dx}=\E_{\rho}\cro{\int \log\cosh(u+h')\phi(x)dx},$$
where we define $h'=\eps \mu_{2,1}$, which is independent of $u$. Using inequality \eqref{eq:Taylor} together with the independence of $u$ and $h'$ leads to 
$$\E_{\rho}\cro{\int \log\cosh(u+h')\phi(x)dx}\geq \E_{\rho}\cro{\log\cosh{u}}+\E_{\rho}\cro{\int \tanh{(u)}\phi(x)dx}\E_{\rho}\cro{h'}- \E_{\rho}\cro{h'^{2}}.$$
Since $\E_{\rho}\cro{h'}=0$ and $\E_{\rho}\cro{h'^{2}}=\eps^{4}$, we have 
\begin{equation}\label{eq:KL_1}
    \E_{\rho}\cro{\int \log\cosh(u+h')\phi(x)dx}\geq \E_{\rho}\cro{\log\cosh{u}}-\eps^{4}\enspace.
\end{equation}
Similarly, for the other terms in the equality \eqref{eq:bigKL}, we have, denoting $u_{2}=\eps((m_2-1)\mu_{2,1}+\sqrt{m_2}x)$, $h_{2}=\eps \mu_{2,1}$ and $h_{2}'=\eps \mu_{1,1}$:
\begin{align*}
    \E_{\rho}\cro{\int \log\cosh(u_{2}+h_{2})\phi(x)dx}&\leq \E_{\rho}\cro{\log\cosh(u_{2})}+(m_2-1)\eps^{4}+2\eps^{4}((m_2-1)^{2}\eps^{2}+m_2)+\eps^{4}\\
    \E_{\rho}\cro{\int \log\cosh(u_{2}+h_{2}')\phi(x)dx}&\geq \E_{\rho}\cro{\log\cosh(u_{2})}-\eps^{4}\enspace.
\end{align*}
We denote $m=m_1+m_2$. Plugging these two inequalities, together with inequalities \eqref{eq:KL_1} and \eqref{eq:KL_2}, in equality \eqref{eq:bigKL} leads to 
\begin{align*}
    KL(\P_{\rho, G^{(1,2)}}, \P_{\rho,G^{(0)}})\leq & p\pa{\eps^{4}(3m+2)+\eps^{6}\pa{(m_{1}-1)^{2}+(m_2-1)^{2}}}\\
    \leq & 4pm\eps^{4}(1+m\eps^{2})\enspace.
\end{align*}
Since $\eps^{2}=\frac{2}{p}\Bar{\Delta}^{2}$, we have 
$$KL(\P_{\rho, G^{(1,2)}}, \P_{\rho,G^{(0)}})\leq 16\Bar{\Delta}^{4}\frac{m}{p}(1+\frac{2m}{p}\Bar{\Delta}^{2})\enspace.$$
Besides, $G^{(0)}\in\mathcal{P}_{\frac{3}{2}}$. Thus, all the groups of $G^{(0)}$ are of size at most $\frac{3}{2}\frac{n}{K}$. Hence, $m\leq 3\frac{n}{K}$. We arrive at 
$$KL(\P_{\rho, G^{(1,2)}}, \P_{\rho,G^{(0)}})\leq48\Bar{\Delta}^{4}\frac{n}{Kp}(1+6\frac{n}{Kp}\Bar{\Delta}^{2})\enspace.$$
The hypothesis $c_{1}\log(n)\leq\Bar{\Delta}^{2}\leq c_{2}\sqrt{\frac{p}{n}K\log(n)}$ leads to 
$$\Bar{\Delta}^{4}\frac{n}{Kp}\leq c_{2}^{2}\log(n)\enspace ,$$ and to 
$$\frac{n}{Kp}\Bar{\Delta}^{2}=\frac{n}{Kp}\Bar{\Delta}^{4}\frac{1}{\Bar{\Delta}^{2}}\leq \frac{c_{2}^{2}\log(n)}{c_{1}\log(n)}\leq 1\enspace ,$$ when $c_{2}$ is chosen small enough with respect to $c_{1}$. Thus, there exists a numerical constant $c$ such that 
$$KL(\P_{\rho, G^{(1,2)}}, \P_{\rho,G^{(0)}})\leq c c_{2}^{2}\log(n)\enspace.$$ 
By symmetry, this inequality is satisfied by all partition $G\in Sh(G^{(0)})$. This concludes the proof of the lemma.

\subsection{Proof of Lemma \ref{lem:reduction}}\label{prf:reduction}
Let us suppose that there exists a probability distribution $\rho$ on $(\R^{p})^{K}$ and $a>0$ such that $$\inf_{\hat{G}}\sup_{G\in\mathcal{P}_{\alpha}}\P_{\rho,G}(\hat{G}\neq G)-\rho(\R^{p\times K}\setminus\Theta_{\Bar{\Theta}})\geq a\enspace.$$
Given an estimator $\hat{G}$, the previous hypothesis directly implies that there exists $G\in\mathcal{P}_{\alpha}$ such that $\P_{\rho,G}(\hat{G}\neq G)-\rho(\R^{p\times K}\setminus\Theta_{\Bar{\Theta}})\geq a.$ 

By definition, $\P_{\rho,G}(\hat{G}\neq G)=\int \P_{\mu, G}(\hat{G}\neq G)d\rho(\mu).$ The quantity $\P_{\mu, G}(\hat{G}\neq G)$ being bounded by $1$, we have $\P_{\rho,G}(\hat{G}\neq G)\leq \int_{\Theta_{\Bar{\Theta}}} \P_{\mu, G}(\hat{G}\neq G)d\rho(\mu)+\rho(\R^{p\times K}\setminus \Theta_{\Bar{\Theta}}).$ Therefore, $\int_{\Theta_{\Bar{\Theta}}} \P_{\mu, G}(\hat{G}\neq G)d\rho(\mu)\geq a.$ This implies the existence of $\mu\in\Theta_{\Delta}$ such that $\P_{\mu, G}(\hat{G}\neq G)\geq a$.

This being true for all estimator $\hat{G}$, we get the following inequality that concludes the proof of the lemma $$\inf_{\hat{G}}\sup_{\mu\in\Theta_{\Bar{\Theta}}}\sup_{G\in\mathcal{P}_{\alpha}}\P_{\mu,G}(\hat{G}\neq G)\geq a\enspace.$$

\section{Proof of Corollary \ref{cor: squareKmeans}}\label{prf:squareKmeans}
We suppose $n\geq cK^{2}\log(n)$, with $c>0$ a numerical constant that we will choose large enough, and $p\geq n$. Let $M^{\hat{G}}$ be the estimator of $M^*$ induced by the exact $K$-means estimator $\hat{G}$. Again, we suppose without loss of generality that $\sigma=1$.

We will show that, if $\Bar{\Delta}^{4}\geq c'\frac{pK}{n}\log(n)$, for $c'$ a numerical constant chosen large enough, the conditions of Theorem \ref{thm:error_K_means} will be satisfied. These conditions are a condition of balanceness of the partition $G^*$ and a condition on the separation of the $\mu_k$'s. 

First, Lemma \ref{lem:sizegroups} states that the partition induced by the $k_{i}$'s balanced. We refer to Section \ref{prf:sizegroups} for a proof of this lemma.

\begin{lem}\label{lem:sizegroups}
    We consider the partition $G^{*}$ induced by the $k_{i}$'s by the relation $G_{k}^{*}=\{i\in[1,n], \enspace k_{i}=k\}$. Then, there exists a numerical constant $c>0$, such that if $n\geq c K^{2}\log(n)$, the following holds with probability at least $1-\frac{2}{n^{2}}$. For all $k\in [1,K]$, the size of $G_{k}^{*}$ satisfies $\frac{n}{2K}\leq |G_{k}^{*}|\leq \frac{3n}{2K}$.
\end{lem}
Hence, conditionally on the $\mu_{k}$'s, Theorem \ref{thm:error_K_means} implies the existence of a constant $c_{1}>0$ such that, if $\min_{k\neq l} \frac{1}{4}\|\mu_{k}-\mu_{l}\|^{4}\geq c_{1}\frac{pK}{n}\log(n)$, the partition $\hat{G}$ recovers exactly the partition $G^{*}$, with probability larger than $1-\frac{c_{2}}{n^{2}}$, with $c_{2}>0$ a numerical constant. The next lemma shows that this condition on the separation of the $\mu_{k}$'s is satisfied with high probability. We refer to Section \ref{prf:separation} for a proof of this lemma.

\begin{lem}\label{lem:separation}
    We suppose $p\geq n$. There exists numerical constants $c_{3}>0$ and $c_{4}>0$ such that, if $\Bar{\Delta}^{4}\geq c_{3}\frac{pK}{n}\log(n)$, the following holds. With probability at least $1-\frac{c_{4}}{n^{2}}$, the separation between the clusters satisfies $\Delta^{4}=\min_{k\neq l}\frac{1}{4}\|\mu_{k}-\mu_{l}\|^{4}\geq c_{1}\frac{pK}{n}\log(n)$.
\end{lem}
Combining Lemma \ref{lem:sizegroups}, Lemma \ref{lem:separation} together with Theorem \ref{thm:error_K_means} leads to the following statement. If $n\geq cK^{2}\log(n)$, $p\geq n$ and $\Bar{\Delta}^{4}\geq c_{3}\frac{pK}{n}\log(n)$, the partition $\hat{G}$ recovers exactly the partition $G^{*}$ with probability at least $\frac{C}{n^{2}}$, with $C$ a numerical constant. This induces $\P[M^{\hat{G}}\neq M^*]\leq \frac{C}{n^{2}}$, and thus 
$$\E[\|M^{\hat{G}}\neq M^*\|_F^2]\leq \frac{C}{n^{2}}\enspace.$$ 
This concludes the proof of the corollary.

\subsection{Proof of Lemma \ref{lem:sizegroups}}\label{prf:sizegroups}

Let us denote $N_{k}=|\{i\in [1,n],\enspace k_{i}=k\}|$, for $k\in[1,K]$. We prove in this section that, if $n\geq cK^{2}\log(n)$, for $c>0$ a numerical constant that we will choose large enough, then, with probability at least $\frac{2}{n^{2}}$, simultaneously on all $k\in[1,K]$, $|N_{k}-\frac{n}{K}|\leq \frac{n}{2K}$.

Let $k\in[1,K]$. Then, $N_{k}=\sum_{i\in[1,n]}\1_{k_{i}=k}$ is a sum of $n$ independent Bernoulli random variables of parameter $\frac{1}{K}$. Hence, $\E[N_{k}]=\frac{n}{K}$ and Hoeffding's inequality implies
\begin{equation*}
    \P[|N_{k}-\frac{n}{K}|\geq \frac{n}{2K}]\leq 2\exp\pa{-\frac{2\frac{n^{2}}{4K^{2}}}{n}}\leq 2\exp\pa{-\frac{n}{2K^{2}}}\enspace.
\end{equation*}
Moreover, if the numerical constant $c$ such that $n\geq cK^{2}\log(n)$ is large enough, we have $\exp\pa{-\frac{n}{2K^{2}}}\leq \frac{1}{n^{2}K}$. An union bound on $k\in[1,K]$ induces that 
$$\P\cro{\exists k\in[1,K],\enspace |N_{k}-\frac{n}{K}|\geq \frac{n}{2K}}\leq \frac{2}{n^{2}}\enspace.$$
This concludes the proof of the lemma.

\subsection{Proof of Lemma \ref{lem:separation}}\label{prf:separation}
We proceed as for the proof of Lemma \ref{lem:lowerboundseparation} in Section \ref{prf:lowerboundseparation}. Let $k\neq l\in [1,K]$. Then, $\|\mu_{k}-\mu_{l}\|^{2}=4\eps^{2}\sum_{d\in[1,p]}\1_{\mu_{k,d}\neq \mu_{l,d}}$. Hence, $\frac{1}{4\eps^{2}}\|\mu_{k}-\mu_{l}\|^{2}$ is a sum of $p$ independent Bernoulli random variable of parameter $\frac{1}{2}$. Using Hoeffding's inequality leads to 
\begin{equation*}
    \P\cro{\frac{1}{4\eps^{2}}\|\mu_{k}-\mu_{l}\|^{2}\leq \frac{p}{4}}\leq \exp\pa{-\frac{p}{8}}\enspace.
\end{equation*}
Since $p\geq n$, there exists a numerical constant $c_{4}>0$ such that $\exp\pa{-\frac{p}{2}}\leq \frac{c_{4}}{K^{2}n^{2}}$. Using an union bound on the different pairs $k\neq l\in[1,K]$ implies that the following holds with probability at least $1-\frac{c_{4}}{n^{2}}$. For all $k\neq l\in[1,K]$, $\frac{1}{2}\|\mu_{k}-\mu_{l}\|^{2}\geq p\eps^{2}=\Bar{\Delta}^{2}$. This concludes the proof of the lemma.

\section{Hierarichical Clustering with single linkage}\label{sec:hierarchical}

For sake of completeness, we provide in this section an analysis of hierarchical clustering with single linkage in the isotropic Gaussian setup. We recall our setup: for a hidden partition $G^*$ and hidden vectors $\mu_{1},\ldots,\mu_{K}\in\R^p$, the $Y_i$'s are drawn independently and, if $i\in G^*_k$, $Y_i\sim\mathcal{N}\pa{\mu_{k},\sigma^2 I_p}$. 

Let us describe the algorithm considered. We build recursively a sequence of partitions as follows. Initially, we take the partition $G^{0}=\ac{\{1\},\ldots ,\{n\}}$. Then, as long as $G^{(t)}$ has more than $K$ groups, we construct the partition $G^{(t+1)}$  by merging two groups of $G^{(t)}$ with the two closest points. The algorithm stops when the number of groups of the partition $G^{(t)}$ is $K$, which occurs when $t=n-K$. Let us write more precisely this algorithm. We define the linkage function between two subsets $A,B\subset [1,n]$ as $l(A,B)=\min_{(i,j)\in A\times B}\|Y_i-Y_j\|^2$. Here is the hierarchical clustering algorithm considered. 

\medskip
\RestyleAlgo{ruled}
\begin{algorithm}
\caption{Hierarchical Clustering algorithm with single linkage}\label{alg:hier}
\KwData{$Y_{1},\ldots, Y_{n}$}
$t\gets 0$;\\
$G^{(0)} \gets \ac{\{1\},\ldots,\{n\}}$;\\
\While{$t<n-K$}{
    Find $\hat{a},\hat{b}$ minimizing $l\pa{G^{(t)}_{\hat{a}},G^{(t)}_{\hat{b}}}$;\\
    Build $G^{(t+1)}$ by merging the groups $G^{(t)}_{\hat{a}}$ and $G^{(t)}_{\hat{b}}$, the other groups remaining unchanged;\\$t\gets t+1$;\\
    }
\KwResult{The partition $G^{(n-K)}$.}
\end{algorithm}

\medskip


The next result gives a sufficient condition on the separation $\Delta$ for recovering exactly the partition $G^*$ with high probability using Algorithm \ref{alg:hier}.

\begin{prop}\label{prop:hierarchical}
    There exists numerical constants $c_1$ and $c_2$ such that the following holds. If $\Delta^2\geq c_1\pa{\log(n)+\sqrt{p\log(n)}}$, hierarchical clustering recovers exactly the partition $G^*$ with probability at least $1-\frac{c_2}{n^2}$.
\end{prop}

\begin{proof}[Proof of Proposition \ref{prop:hierarchical}]
    Without loss of generality, we suppose that $\sigma=1$. We recall that, by definition, for $i\in G^*_k$, we have $k^*_i=k$. Let $i\neq j\in [1,n]$. Then

    \begin{equation*}
        \|Y_i-Y_j\|^2=\|E_i-E_j\|^2+2\<E_i-E_j,\mu_{k^*_i}-\mu_{k^*_j}\>+\|\mu_{k^*_i}-\mu_{k_j^*}\|^2\enspace. 
    \end{equation*}
    In order to prove Proposition \ref{prop:hierarchical}, we shall prove that, with high probability, the above quantity is uniformly smaller when $k_i^*=k_j^*$ than when $k^*_i\neq k_j^*$. For $i\neq j\in[1,n]$, using Lemma \ref{HW}, we get that for some numerical constants $c>0$ and for all $x>0$,

    $$\P[|\|E_i-E_j\|^2-2p|>c\pa{\sqrt{px}+x}]\leq 2e^{-x}\enspace.$$ Setting $e^{-x}=\frac{1}{n^4}$ and doing an union bound on all possible couples $i\neq j$, we get

    $$\P\cro{\forall i\neq j\in[1,n],\enspace |\|E_i-E_j\|^2-2p|>4c\pa{\sqrt{p\log(n)}+\log(n)}}\leq \frac{1}{n^2}\enspace.$$

    Let us know control the cross term $\<E_i-E_j,\mu_{k^*_i}-\mu_{k^*_j}\>$ uniformly on all $i\neq j\in[1,n]$. For such $i\neq j\in[1,n]$, $\<E_i-E_j,\mu_{k^*_i}-\mu_{k^*_j}\>\sim \sqrt{2}\|\mu_{k^*_i}-\mu_{k^*_j}\|\mathcal{N}\pa{0,I_p}$. Hence, for some numerical constant $c'>0$, with probability at least $1-e^{-x^2}$, we have $\<E_i-E_j,\mu_{k^*_i}-\mu_{k^*_j}\>\geq -c'x\|\mu_{k^*_i}-\mu_{k^*_j}\| $. Setting $e^{-x^2}=\frac{1}{n^4}$ and doing an union bound on all $i\neq j$, we end up with $$\P\cro{\forall i\neq j,\enspace \<E_i-E_j,\mu_{k^*_i}-\mu_{k^*_j}\>\geq -2c'\sqrt{\log(n)}\|\mu_{k^*_i}-\mu_{k^*_j}\|}\geq 1-\frac{1}{n^2}\enspace.$$

    Hence, with probability at least $1-\frac{2}{n^2}$, simultaneously on all $i\neq j\in [1,n]$, we have the two inequalities \begin{align*}
        \<E_i-E_j,\mu_{k^*_i}-\mu_{k^*_j}\>&\geq -2c'\sqrt{\log(n)}\|\mu_{k^*_i}-\mu_{k^*_j}\|\enspace;\\
        |\|E_i-E_j\|^2-2p|&\leq 4c\pa{\sqrt{p\log(n)}+\log(n)}\enspace.
    \end{align*}
    Let us restrict ourselves to this event of probability at least $1-\frac{2}{n^2}$ on which these two inequalities are satisfied. For $i\neq j\in [1,n]$, \begin{itemize}
        \item If $k_i^*=k_j^*$, then $\|Y_i-Y_j\|^2=\|E_i-E_j\|^2\leq 2p+4c\pa{\sqrt{p\log(n)}+\log(n)}$,
        \item If $k_i^*\neq k_j^*$, then $\|Y_i-Y_j\|^2\geq 2p-4c\pa{\sqrt{p\log(n)}+\log(n)}-4c'\sqrt{\log(n)}\|\mu_{k^*_i}-\mu_{k^*_j}\|+\|\mu_{k^*_i}-\mu_{k^*_j}\|^2$.
        
    \end{itemize}

    Thus, if $\Delta^2\geq c_1\pa{\log(n)+\sqrt{p\log(n)}}$, for $c_1$ a numerical constant chosen large enough, we get, for all $i\neq j$, \begin{itemize}
        \item If $k_i^*=k_j^*$, then $\|Y_i-Y_j\|^2\leq 2p+\frac{\Delta^2}{3}$,
        \item If $k_i^*\neq k_j^*$, then $\|Y_i-Y_j\|^2\geq 2p+\frac{2\Delta^2}{3}$.
    \end{itemize}
    Therefore, for all $i\neq j$ such that $k_i^*=k_j^*$ and $i'\neq j'$ such that $k_{i'}^*\neq k_{j'}^*$, we have with probability at least $1-\frac{2}{n^2}$
    \begin{equation}\label{eq:hier}
        \|Y_i-Y_j\|^2<\|Y_{i'}-Y_{j'}\|^2\enspace.
    \end{equation} In other words, Algorithm \ref{alg:hier} will always choose, when it is possible, to merge groups that both intersect a same cluster of the partition $G^*$. By induction on $t\in[0,n-K]$, we deduce from this that $G^{(t)}$ is a subpartition of $G^*$, ie that each group of $G^{(t)}$ is a subset of a group of $G^*$ .

    \textbf{Initialization:} The partition $G^{(0)}=\{1\},\ldots,\{n\}$ is indeed a subpartition of $G^*$.

    \textbf{Induction step:} Let $t\in[0, n-K-1]$ and let us suppose that $G^{(t)}$ is a subpartition of $G^*$ and let us prove that so is $G^{(t+1)}$. Since $t\leq n-K-1$, $|G^{(t)}|\geq K+1$. Hence, there exists at least two groups of $G^{(t)}$ that are subsets of the same group of $G^*$. Equation \eqref{eq:hier} ensures that Algorithm \ref{alg:hier} will choose to merge such a pair of groups. Hence, $G^{(t+1)}$ is also a subpartition of $G^*$. This concludes the induction.

    \medskip
    In particular, the output partition $G^{(n-K)}$ is a subpartition of $G^*$. Combining this with $|G^{(n-K)}|=K$ leads to $G^{(n-K)}=G^*$. This concludes the proof of the proposition.
\end{proof}

\end{document}